\documentclass[bj,preprint,nameyear,seceqn,noshowframe]{imsart}

\usepackage{mathtools}
\usepackage{tikz}
\usetikzlibrary{decorations.markings}
\usepackage{pgfplots}
\pgfplotsset{compat=1.11}
\usepackage{enumitem}

\volume{00}
\issue{00}
\pubyear{0000}
\firstpage{1}
\lastpage{27}
\doi{10.3150/26-BEJ1983}

\startlocaldefs
\numberwithin{equation}{section}
\theoremstyle{plain}
\newtheorem{thmVTEX}{Theorem}[section]
\newtheorem{lem}[thmVTEX]{Lemma}
\newtheorem{cor}[thmVTEX]{Corollary}
\theoremstyle{definition}
\newtheorem{rem}{Remark}
\newtheorem{example}{Example}

\newcommand{\Real}{\mathbb R}
\renewcommand{\Re}{\mathrm{Re}}
\renewcommand{\Im}{\mathrm{Im}}
\newcommand{\sign}{\mathrm{sign}}
\newcommand{\one}[1]{\mathbf{1}_{\{#1\}}}
\newcommand{\eps}{\varepsilon}
\newcommand{\arccot}{\mathrm{arccot}}

\def\Xint#1{\mathchoice {\XXint\displaystyle\textstyle{#1}}%
{\XXint\textstyle\scriptstyle{#1}}%
{\XXint\scriptstyle\scriptscriptstyle{#1}}%
{\XXint\scriptscriptstyle\scriptscriptstyle{#1}}%
\!\int}
\def\XXint#1#2#3{{\setbox0=\hbox{$#1{#2#3}{\int}$}
\vcenter{\hbox{$#2#3$}}\kern-.5\wd0}}
\def\dashint{\Xint-}

\endlocaldefs

\begin{document}
\begin{frontmatter}

\title{Asymptotic analysis of the finite predictor for fractional Gaussian noise}
\runtitle{Asymptotic analysis of the finite predictor for fGn}


\begin{aug}
%
%
%
\author[A]{\inits{P.}\fnms{Pavel}~\snm{Chigansky}\ead[label=e1]{Pavel.Chigansky@mail.huji.ac.il}}
\author[B]{\inits{M.}\fnms{Marina}~\snm{Kleptsyna}\ead[label=e2]{Marina.Kleptsyna@univ-lemans.fr}}
%
%
\address[A]{Department of Statistics,
The Hebrew University of Jerusalem, Jerusalem,
Israel\printead[presep={,\ }]{e1}}
\address[B]{Laboratoire Manceau de Math\'{e}matiques,
Le Mans Universit\'{e}, Le Mans,
France\printead[presep={,\ }]{e2}}
\end{aug}

\received{\smonth{4} \syear{2025}}
\revised{\smonth{2} \syear{2026}}

\begin{abstract}
This paper proposes a new approach to the asymptotic analysis of the finite
predictor for stationary sequences. Our method yields the exact asymptotics
of both the relative prediction error and the partial correlation coefficients.
The underlying assumptions are analytic in nature, making the approach
applicable to processes with long-range dependence. The ARMA-type process
driven by fractional Gaussian noise (fGn), which had previously remained
elusive, is used as a case study.
\end{abstract}


\end{frontmatter}
\section{Introduction}
\label{sec1}

\subsection{Prediction}
\label{sec1.1}

Consider a centered, weakly stationary, real-valued random process
$X:=(X_{n})_{n\in \mathbb Z}$. The covariance sequence is defined by
$ \gamma (k):=\mathsf{E}X_{0}X_{k} $ and the spectral density is given
by
\begin{equation*}
f(\lambda ) := \frac {1} {2\pi} \sum _{k=-\infty}^{\infty }\gamma (k) e^{i
\lambda k},\quad \lambda \in (-\pi ,\pi ].
\end{equation*}
For a pair of integers $k\le m$, let $P_{[k,m]}$ be the projection operator
onto the linear subspace spanned by $\{X_{k},...,X_{m}\}$. The optimal
one-step predictor of $X_{n}$ is the projection $P_{[1,n-1]}X_{n}$ and
its mean squared error
\begin{equation*}
\sigma ^{2}(n) = \mathsf{E}\big(X_{n} - P_{[1,n-1]}X_{n}\big)^{2}
\end{equation*}
is minimal among all linear predictors based on the data
$\{X_{1},...,X_{n-1}\}$. The sequence $\sigma ^{2}(n)$ is non-increasing
and its limit is given by the Szeg\"{o}-Kolmogorov geometric mean formula
%
\begin{equation}
\label{SKfla}
\sigma ^{2}(n) \xrightarrow[n\to \infty ]{} 2\pi \exp \left (
\frac {1}{2\pi } \int _{-\pi}^{\pi} \log f(\lambda ) d\lambda \right )
=: \sigma ^{2}.
\end{equation}

Another important quantity related to the prediction problem is the sequence
of partial correlation coefficients:
%
\begin{equation}
\label{alpha}
\alpha (1) :=   \gamma (1)/\gamma (0),\qquad 
\alpha (n) :=   \rho \big(X_{0}-P_{[1,n-1]}X_{0}, X_{n}-P_{[1,n-1]}X_{n}
\big), \quad n>1,
%
\end{equation}
where
$\rho (\xi , \eta ) = \mathsf{Cov}(\xi ,\eta )/\sqrt{\mathsf{Var}(
\xi )\mathsf{Var}(\eta )}$ denotes the Pearson correlation between the
random variables $\xi $ and $\eta $. The prediction error and the partial
correlation are related through the identity
\begin{equation*}
\sigma ^{2}(n)= \prod _{j=1}^{n} \big(1-\alpha (j)^{2}\big).
\end{equation*}
The asymptotic behavior of the partial correlation coefficients and the
{\em relative} prediction error,
\begin{equation*}
\delta (n) := \sigma ^{2}(n)-\sigma ^{2}, \quad \text{as}\ n\to
\infty ,
\end{equation*}
has been the subject of extensive research for decades. A comprehensive
account of this and other aspects of the prediction problem can be found
in the recent survey \cite{BG23}.

\subsection{Prior related results}
\label{sec1.2}

If the spectral density of the process is analytic and strictly positive,
the convergence rate of $\delta (n)$ is at least geometric, as shown in
\cite[Ch 10 \S 10]{GS84}. However, the presence of zeros or singularities
in the spectral density slows the convergence at a hyperbolic rate.

\begin{thmVTEX}[\cite{IS68}]
\label{thm:IS}
Assume the spectral density has the form
\begin{equation*}
f(\lambda )= f_{1}(\lambda ) \prod _{k=1}^{m} \big|e^{i\lambda}-e^{i
\lambda _{k}}\big|^{-2d_{k}},
\end{equation*}
where the function $f_{1}(\lambda )$ is strictly positive and
$\alpha $-Lipschitz with $\alpha \ge \frac {1} {2}$, the points
$\lambda _{k}\in [-\pi ,\pi ]$ are distinct, and the exponents
$d_{k}< \frac {1} {2}$ are not all zeros. Then
\begin{equation*}
\delta (n) \asymp 1/n, \quad n\to \infty ,
\end{equation*}
where $x_{n}\asymp y_{n}$ denotes that
$0<\varliminf x_{n}/y_{n} \le \varlimsup x_{n}/y_{n}<\infty $.
\end{thmVTEX}

The special case where $m=1$ and $\lambda _{1}=0$ and
$d_{1}:=d\in (0,\frac {1}{2} )$ corresponds to processes with long memory
(or long-range dependence), for which the covariance sequence is not absolutely
summable, $ \sum _{k=-\infty}^{\infty}|\gamma (k)|=\infty  $,
\cite{PT17}. In this case, under certain conditions, the asymptotics of
the relative error and partial correlation can be made precise; this reveals
an interesting universality discovered in \cite{In00}.

To formulate this result, let us recall a few basic notions from the theory
of stationary processes \cite{BDbook91}. A stationary process $X$ is purely
non-deterministic if the integral in \eqref{SKfla} is finite, that is,
$\sigma ^{2}>0$. Any such process admits the moving average MA($
\infty $) representation
\begin{equation*}
X_{n} = \sum _{j=-\infty}^{n} c_{n-j} \xi _{j}, \quad n\in \mathbb Z,
\end{equation*}
and the autoregressive AR($\infty $) representation
\begin{equation*}
\sum _{j=-\infty}^{n} a_{n-j} X_{j} + \xi _{n} =0, \quad n\in
\mathbb Z,
\end{equation*}
where $\xi _{j}$'s are orthogonal standardized random variables. The real
numbers $(c_{j})$ and $(a_{j})$ are called, respectively, the MA($
\infty $) and AR($\infty $) coefficients of $X$.

A positive measurable function $\ell (\cdot )$, defined on some neighborhood
of infinity, is called slowly varying if
$\lim _{x\to \infty}\ell (tx)/\ell (x)=1$ for all $t>0$. In the next theorem,
the condition $\int ^{\infty }s^{-1}\ell (s) ds =\infty $ indicates that
$\int _{B}^{\infty }s^{-1} \ell (s) ds =\infty $ for some $B>0$ such that
$\ell (\cdot )$ is locally bounded on $[B,\infty )$. In this case, the
function $\widetilde \ell (x) := \int _{B}^{x} s^{-1} \ell (s)ds$ for
$x\ge B$ is also slowly varying.

\begin{thmVTEX}[Theorem 6.1 in \cite{In08}]%
\label{thm:I2}
Let $X$ be a purely non-deterministic process whose covariance satisfies\footnote{Here and below, $x_{n}\sim y_{n}$ means that
$x_{n} = y_{n}(1+o(1))$ as $n\to \infty $.}
\begin{equation*}
\gamma (k) \sim k^{2d-1}\ell (k), \quad k\to \infty ,
\end{equation*}
for some constant $d\in (-\infty , \frac {1} {2})$ and a function
$\ell (\cdot )$, slowly varying at infinity. Assume that the MA($
\infty $) and AR($\infty $) coefficients of $X$ satisfy the following conditions:
%
\begin{equation}
\label{ac}
\begin{aligned}
& c_{k} \ge 0 \ \text{for all}\ k\ge 0;
\\
& (c_{k}) \text{\ is eventually decreasing to zero};
\\
& (a_{k}) \text{\ is eventually decreasing to zero}.
\end{aligned}
%
\end{equation}

1. If $d\in (0,\frac {1} {2})$, then
%
\begin{equation}
\label{lrd}
\alpha (n) \sim \frac {d}{n}, \quad n\to \infty .
\end{equation}

2. If $d=0$ and $\int ^{\infty }s^{-1}\ell (s)ds =\infty $, then
\begin{equation*}
\alpha (n) \sim \frac{\ell (n)}{2n\widetilde \ell (n)}, \quad n\to
\infty .
\end{equation*}

3. If $d=0$ with $\int ^{\infty }s^{-1} \ell (s)ds <\infty $ or
$d\in (-\infty ,0)$, then
\begin{equation*}
\alpha (n)\sim
\frac{n^{2d-1}\ell (n)}{\sum _{k\in \mathbb Z}\gamma (k)}, \quad n
\to \infty .
\end{equation*}
\end{thmVTEX}

A striking feature of the asymptotics in the long memory case
\eqref{lrd} is that it does not depend on any details of the covariance
(or spectral density) other than $d$, including the slowly varying part
$\ell (\cdot )$. In this case, the relative prediction error can be shown
to satisfy the asymptotics
%
\begin{equation}
\label{dan}
\delta (n) \sim \sigma ^{2} \frac{d^{2}}{n}, \quad n\to \infty .
\end{equation}

A sufficient condition for \eqref{ac} is that the covariance admits the
representation $ \gamma (k) = \int _{0}^{1} t^{|k|}\mu (dt) $, for some
finite Borel measure $\mu (\cdot )$ on $[0,1)$. This property, known as
\emph{reflection positivity}, is satisfied, for instance, by the prototypical
example $\gamma (k)=(1+|k|)^{2d-1}$, see \cite[\S 6]{In08} for further
details. In general, the conditions in \eqref{ac} may not be easily verifiable
or may not hold for certain concrete processes of interest.

In particular, \eqref{ac} is not satisfied by the FARIMA$(p,d,q)$ process,
which plays a central role in the theory and applications of long-memory
time series \cite{PT17}. The spectral density of this process has the form
%
\begin{equation}
\label{farima}
f(\lambda ) = \left |
\frac{\theta (e^{i\lambda}) }{ \phi (e^{i\lambda}) }\right |^{2}
\frac {1}{2\pi}\big|1-e^{i\lambda}\big|^{-2d}, \quad \lambda \in (-
\pi , \pi ],
\end{equation}
with parameter $d\in (-\frac {1} {2}, \frac {1} {2})\setminus \{0\}$. The
MA and AR polynomials $\theta (\cdot )$ and $\phi (\cdot )$ are of degrees
$q$ and $p$, respectively, with real-valued coefficients and are normalized
so that $\theta (0)=\phi (0)=1$. The FARIMA process exhibits long memory
when $d\in (0,\frac {1} {2})$, which manifests as a singularity at the
origin of its spectral density \eqref{farima}.

\begin{thmVTEX}[Theorem 2.5 in \cite{In08}]%
\label{thm:I3}
Assume that the polynomials $\phi (\cdot )$ and $\theta (\cdot )$
\begin{enumerate}
\renewcommand{\theenumi}{A\arabic{enumi}}
\renewcommand{\labelenumi}{(\theenumi)}
\item[(A1)]
are relatively prime,
\item[(A2)]
have no zeros in the closed unit disk $\{z\in \mathbb C: |z|\le 1\}$.
\end{enumerate}
\noindent
Let $X$ be the FARIMA$(p,d,q)$ process with
$p,q\in \mathbb N\cup \{0\}$ and
$d\in (-\frac {1} {2},\frac {1} {2})\setminus \{0\}$. Then
%
\begin{equation}
\label{alphan}
\alpha (n) \sim \frac {d}{n}, \quad n\to \infty ,
\end{equation}
and \eqref{dan} holds with $\sigma ^{2}=1$.
\end{thmVTEX}

The main tool underlying the method pioneered in the series of papers
\cite{In00}-\cite{In23} is von Neumann's alternating projection theorem.
This theorem asserts that for a pair of closed subspaces $A$ and $B$ of
a Hilbert space $\mathcal H$
%
\begin{equation}
\label{VN}
\lim _{n} (P_{B}P_{A})^{n} x = P_{A\cap B}x, \quad \forall x\in
\mathcal H,
\end{equation}
where $P_{W}$ stands for the orthogonal projection operator onto a subspace
$W\subseteq \mathcal H$. In the context of the prediction problem, the
Hilbert space is
$\mathcal H := \overline {\mathrm{span}}\{X_{k}: k\in \mathbb N\}$ with
the standard inner product
$\langle \xi , \eta \rangle := \mathsf{E}\xi \eta $ for
$\xi , \eta \in \mathcal H$. Under appropriate conditions, the subspaces
$A := \overline {\mathrm{span}}\{X_{k}: k\le n-1\}$ and
$B :=\overline {\mathrm{span}}\{X_{k}: k\ge 1\}$ satisfy the ``intersection
of past and future'' property:
%
\begin{equation}
\label{AcapB}
A\cap B = \mathrm{span}\big\{X_{1},...,X_{n-1}\big\}.
\end{equation}

The projections onto the infinite past and future, $P_{A}$ and
$P_{B}$, can be readily expressed in terms of the MA($\infty $) and AR($
\infty $) coefficients, $(c_{k})$ and $(a_{k})$, of the process $X$. Consequently,
in view of equations \eqref{VN} and \eqref{AcapB}, the quantities associated
with the finite predictor $P_{[1,n-1]}$ can also be expressed using these
coefficients. This facilitates the derivation of a useful representation
for the partial correlation $\alpha (n)$ in terms of $(c_{k})$ and
$(a_{k})$. Theorem \ref{thm:I2} and Theorem \ref{thm:I3} are then proved
by means of an asymptotic analysis of this representation as
$n\to \infty $.

\subsection{This paper}
\label{sec1.3}

The principal contribution of this paper is a novel approach to the asymptotic
analysis of the prediction problem for stationary sequences. Our method
applies to ARMA-type processes with spectral densities of the form
%
\begin{equation}
\label{f}
f(\lambda ) = \left |
\frac{\theta (e^{i\lambda}) }{ \phi (e^{i\lambda}) }\right |^{2} f_{0}(
\lambda ), \quad \lambda \in (-\pi , \pi ],
\end{equation}
where $\theta (\cdot )$ and $\phi (\cdot )$ are polynomials as in
\eqref{farima} and $f_{0}(\lambda )$ is a given spectral density. Central
to our approach is the assumption that $f_{0}(\lambda )$ admits a \emph{sectionally
holomorphic} extension to the complex plane. More precisely, there must
exist a function $Q(z)$ such that\vspace*{-2.5pt}
%
\begin{equation}
\label{f0Q}
Q(e^{i\lambda})=f_{0}(\lambda ), \quad \lambda \in (-\pi ,\pi ],
\end{equation}
and which is holomorphic everywhere except, possibly, along a curve, where
it may exhibit a jump discontinuity. The implementation of our method further
requires certain qualitative information about $Q(z)$, such as its growth
rates at the origin and at infinity, the configuration of its zeros, and
its relevant symmetries.

The generality under which such an extension exists and the structure it
may take are important questions left for future research. However, an
explicit construction is often possible. One such example is provided by
the FARIMA$(p,d,q)$ process with spectral density \eqref{farima}, for which\vspace*{-2.5pt}
%
\begin{equation}
\label{Qfarima}
Q(z) = \frac {1} {2\pi} \big((1-z)(1-z^{-1})\big)^{-d}, \quad z\in
\mathbb C \setminus \mathbb R_{+}.
\end{equation}
With the standard choice of the principal branch of the power function,
$Q(z)$ is sectionally holomorphic with a jump discontinuity across the
positive semi-axis $\mathbb R_{+}$. Applying our method to this case provides
an alternative proof of Theorem \ref{thm:I3}.

In this paper, we consider an ARMA-type process driven by fractional Gaussian
noise (fGn), which is the sequence of increments of fractional Brownian
motion (fBm). Like the FARIMA process, the fGn is a key element in the
study of long-memory processes (see \cite{PT17}). Its spectral density
takes a more complicated form (see \eqref{fGnf} below); nevertheless, the
corresponding extension $Q(z)$ can still be constructed explicitly as a
series \eqref{Qdef} involving special functions. We will establish the
validity of the asymptotics in \eqref{dan} and \eqref{alphan} for the fGn
case - a result that has, until now, remained elusive.

The fBm is the centered Gaussian process
$B^{H}=(B^{H}_{t} , t\in \mathbb R_{+})$ with continuous paths and covariance\vspace*{-2.5pt}
%
\begin{equation}
\label{fBm}
\mathsf{E}B^{H}_{s} B^{H}_{t} = \tfrac {1} {2} \big(s^{2H} + t^{2H} - |t-s|^{2H}
\big), \quad s,t\in \mathbb R_{+},
\end{equation}
where $H\in (0,1)$ is the Hurst exponent. It is the unique Gaussian self-similar
process whose increments $\Delta B^{H}_{n} := B^{H}_{n}-B^{H}_{n-1}$ form
a stationary sequence \cite[Theorem 1.3.3]{EM02}. To establish an analogy
with the FARIMA process, we adopt the parametrization
$d:= H-\frac {1} {2}\in (-\frac {1} {2}, \frac {1} {2})$. The covariance
sequence of the fGn is readily deduced from \eqref{fBm}:\vspace*{-2.5pt}
%
\begin{equation}
\label{fGnc}
\gamma _{0}(k) = \tfrac {1} {2} \big(|k+1|^{2d+1} -2|k|^{2d+1} + |k-1|^{2d+1}
\big).
\end{equation}
This sequence satisfies\vspace*{-2.5pt}
%
\begin{equation}
\label{asym_gamma_fGn}
\gamma _{0}(k) \sim d(2d+1) k^{2d-1}, \quad k\to \infty ,
\end{equation}
and, consequently, the fGn exhibits long memory for
$d\in (0,\frac {1} {2})$. Its spectral density is given by the series\vspace*{-2.5pt}
%
\begin{equation}
\label{fGnf}
f_{0}(\lambda ) = c(d) |1-e^{i\lambda}|^{2} \sum _{k=-\infty}^{
\infty }|\lambda +2\pi k|^{-2d-2}, \quad \lambda \in (-\pi ,\pi ]
\end{equation}
where $c(d) = \dfrac{1}{2\pi}\Gamma (2d+2)\cos (\pi d)$. Thus, the spectral
density of the associated fGn-driven ARMA-type process \eqref{f} has an
integrable power singularity at the origin for
$d\in (0,\frac {1}{2})$; specifically,\vspace*{-2.5pt}
\begin{equation*}
f(\lambda ) \sim \big| \theta (1)/\phi (1)\big|^{2}c(d) |\lambda |^{-2d},
\quad \lambda \to 0,
\end{equation*}
resembling the FARIMA$(p,d,q)$ density \eqref{farima}.\vadjust{\goodbreak}

\medskip
\noindent
We shall prove the following analogue of Theorem \ref{thm:I3}.

\begin{thmVTEX}%
\label{main-thm}
Let $X$ be the fGn-diven ARMA-type process \eqref{f} where
$f_{0}(\lambda )$ is the spectral density \eqref{fGnf} of the fGn with
$d\in (-\frac {1} {2}, \frac {1} {2})\setminus \{0\}$. Assume that the
polynomials $\phi (\cdot )$ and $\theta (\cdot )$ satisfy condition
\textup{(A1)} and
\begin{enumerate}

\renewcommand{\theenumi}{$A2_{\phi}$}
\renewcommand{\labelenumi}{(\theenumi)}
\item[($A2_{\phi }$)]
$\phi (\cdot )$ has no zeros in the closed unit disk,

\renewcommand{\theenumi}{$A2_{\theta}$}
\renewcommand{\labelenumi}{(\theenumi)}
\item[($A2_{\theta }$)]
$\theta (\cdot )$ has no zeros on the unit circle.
\end{enumerate}

Then, the partial correlation of $X$ satisfies \eqref{alphan} and its relative
prediction error follows the asymptotics
%
\begin{equation}
\label{deluc}
\delta (n) \sim \left (\prod _{j: |z_{j}|<1}\frac {1} {z_{j}^{2}}
\right ) \sigma ^{2}_{0} \frac{d^{2}}{n}, \quad n\to \infty ,
\end{equation}
where $z_{1},...,z_{q}$ are the zeros of $\theta (\cdot )$ and
$\sigma ^{2}_{0}$ is given by \eqref{SKfla} with $f(\cdot )$ replaced by
$f_{0}(\cdot )$.
\end{thmVTEX}

\begin{rem}
\label{rem1}
Unlike the processes covered by Theorem \ref{thm:I2}, the fGn-driven ARMA-type
process satisfies the asymptotics in \eqref{alphan} for all
$d\in (-\frac {1} {2}, \frac {1} {2})\setminus \{0\}$, consistent with
the behavior of the FARIMA process. Consequently, the conditions of Theorem
\ref{thm:I2} do not hold in this case either.
\end{rem}

\begin{rem}
\label{rem2}
The product factor in \eqref{deluc} arises because, unlike in condition
\textup{(A2)}, the MA polynomial
%
\begin{equation}
\label{theta_poly}
\theta (z) = \prod _{j=1}^{q} (1-z_{j}^{-1}z),
\end{equation}
is not assumed to have all its zeros outside the unit disk. This factor
emerges naturally in our calculations; hence, we will proceed under the
weaker assumption, \textup{($A2_{\theta }$)}. It is worth noting, however, that assumption
\textup{(A2)} is not restrictive, as the same factor can be derived through
a simple transformation (see \cite[Remark 5 in \S 3.1]{BDbook91}).
\end{rem}

\begin{rem}
\label{rem3}
Theorem \ref{thm:IS} shows that the zeros of the MA polynomial on the unit
circle contribute to a hyperbolic rate. Other zeros have a geometrically
negligible effect, as demonstrated in equations \eqref{zeta} and
\eqref{geom} in the proofs below. Our method remains applicable when there
are zeros on the unit circle, although the calculations become more involved.
To illustrate this, consider the following relatively simple example, which
demonstrates how a zero at $-1$ interacts with the fGn singularity at 1
at the level of the asymptotic constants.

\begin{example}%
\label{ex:1}
Let $X$ be an fGn-driven ARMA$(q,1)$ process with
$d\in (-\frac {1} {2},0)$, an AR polynomial $\phi (z)$ as in Theorem
\ref{main-thm} and the MA polynomial $\theta (z)=1+z$. Then
%
\begin{equation}
\label{hyp}
\begin{aligned}
\delta (n) =\, \sigma ^{2}_{0}(d^{2}+1)n^{-1}+ O(n^{-2}),\quad
\alpha (n) =\, \big(d-(-1)^{n}\big)n^{-1} + O(n^{-2}),
\end{aligned}
\qquad n\to \infty .
\end{equation}
These formulas are derived in Appendix \ref{app:E}. 
\end{example}
\end{rem}

\subsection{Frequent notations}
\label{sec1.4}

Throughout the paper, we adopt the following notations and conventions:

\begin{itemize}
\item Generic constants are denoted by $C,C_{1},C_{2}$, etc. Their values
are of no importance and may change from line to line.
\item The open unit disk in the complex plane is denoted by
$D:=\{z\in \mathbb C: |z|<1\}$, the unit circle by
$\partial D:=\{z\in \mathbb C: |z|=1\}$, and the closed disk by
$\overline D = D\cup \partial D$.
\item We use the standard hat notation for functions defined by Fourier
series:
\begin{equation*}
\widehat a(\lambda ) := \frac {1} {2\pi} \sum _{k=-\infty}^{\infty }a_{k}
e^{i\lambda k}, \quad \lambda \in (-\pi ,\pi ],
\end{equation*}
with the exception of spectral densities, for which the hat is omitted.
\item In the proofs, we employ standard tools from complex analysis
\cite{Sbook03} and, in particular, those pertaining to boundary value problems
\cite{Gahov}, such as the Sokhotski-Plemelj theorem
\cite[Ch1, \S 4.2]{Gahov}.
\item A complex function $F(\cdot )$ is said to be sectionally holomorphic
in $\mathbb C\setminus L$, where $L$ is a simple curve, if it is holomorphic
in $\mathbb C\setminus L$ and has finite limits at all points
$t\in L$ (except possibly at the endpoints, where it may have singularities).
In this paper, $L$ is typically a finite or infinite interval on the real
line $\mathbb R$. We denote the boundary values by
\begin{equation*}
F^{+}(t) = \lim _{z\to t^{+}}F(z), \quad F^{-}(t) = \lim _{z\to t^{-}}F(z),
\quad t\in L,
\end{equation*}
where $z\to t^{+}$ and $z\to t^{-}$ indicate that $z$ approaches
$t\in \mathbb R$ from the upper and lower half-planes, respectively.
\end{itemize}

\medskip

The rest of the paper is organized as follows. Section
\ref{sec:pre-proof} presents three theorems that together imply the assertion
of Theorem \ref{main-thm}; this section serves as a roadmap for the proof.
Each of these theorems is proved in a separate section thereafter. Section
\ref{sec:Q} summarizes the relevant properties of the extension
$Q(\cdot )$ of the fGn spectral density. Some calculations and auxiliary
results have been relegated to the appendices. 

\section{Proof of Theorem \ref{main-thm}}
\label{sec:pre-proof}

Our approach is inspired by the spectral methods for weakly singular integral
operators, pioneered in \cite{Ukai} and \cite{P74,P03}, and their recent
applications to processes with fractional covariance structure (\cite{ChK},
\cite{AChKM22}). It applies to problems in which the quantity of interest
can be expressed as a functional of the solution to a linear equation.
The main idea is to reduce this equation to an equivalent Hilbert boundary
value problem in complex analysis, which becomes asymptotically more tractable
as $n\to \infty $. The implementation consists of three main steps.

\begin{enumerate}
\item[{1.}]
A Hilbert problem is formulated such that a particular solution can be
constructed based on the linear equation in question. This is achieved
by considering the generating functions of the sequences associated with
the equation. The target functional is then expressed in terms of the relevant
components of the Hilbert problem.
\item[{2.}] The general solution to the Hilbert problem is expressed as a system
of coupled integral and algebraic equations. The value of the target functional
is directly related to the unknowns within the algebraic component of this
system.
\item[{3.}] The integro-algebraic system is shown to admit a unique solution
for all sufficiently large $n$. Consequently, the solution to the Hilbert
problem is also unique; it can therefore be identified with the particular
solution from Step \textup{(1)}. The limiting value of the functional
is determined through an asymptotic analysis as $n\to \infty $.
\end{enumerate}

In this section, we present three theorems that encapsulate the results
of each step of this program as applied to the problem under consideration.
Collectively, these theorems establish the assertion of Theorem
\ref{main-thm}. Their detailed proofs are provided in the subsequent sections
of the paper.

\begin{rem}%
\label{rem:3}
We assume that all $q$ zeros of the MA polynomial $\theta (\cdot )$ are
simple. Theorem \ref{main-thm} remains true without this assumption; the
necessary adjustments to the proof in the case of zeros with multiplicities
are detailed in Appendix \ref{app:D}. 
\end{rem}

\subsection{The predictor equations}
\label{sec:peq}

Our starting point is the system of linear equations for the predictor
coefficients \cite{BDbook91}. The forward and backward predictors are the
linear forms
\begin{equation*}
P_{[1,n-1]}X_{n}=\sum _{j=1}^{n-1} g_{n}(n-j) X_{j} \quad \text{and}
\quad P_{[1,n-1]}X_{0}=\sum _{j=1}^{n-1} g_{n}(j) X_{j},
\end{equation*}
where the weights $g_{n}(1),...,g_{n}(n-1)$ solve the system of linear
equations
%
\begin{equation}
\label{maineq}
\sum _{k=1}^{n-1} g_{n}(k) \gamma (j-k)=\gamma (j), \quad j= 1,...,n-1.
\end{equation}
The corresponding prediction errors coincide,
\begin{equation*}
\mathsf{E}(X_{0}-P_{[1,n-1]}X_{0})^{2} = \mathsf{E}(X_{n}-P_{[1,n-1]}X_{n})^{2}
=\sigma ^{2}(n),
\end{equation*}
and can be expressed in terms of the solution to \eqref{maineq} using the
formula
%
\begin{equation}
\label{sigex}
\sigma ^{2}(n) = \gamma (0)-\sum _{j=1}^{n-1} g_{n}(j) \gamma (j).
\end{equation}
A similar formula determines the covariance of the prediction errors,
%
\begin{equation}
\label{covex}
\mathsf{E}\big(X_{0}-P_{[1,n-1]}X_{0}\big)\big(X_{n} - P_{[1,n-1]}X_{n}
\big) = \gamma (n) -\sum _{j=1}^{n-1} g_{n}(j) \gamma (n-j),
\end{equation}
and consequently the partial correlation coefficients in
\eqref{alpha}.

Let us define a pair of auxiliary sequences
%
\begin{equation}
\label{gLgR}
\begin{aligned}
g^{L}_{n}(j) :=
\begin{cases}
\gamma (j)-\sum _{k=1}^{n-1} g_{n}(k) \gamma (j-k), & j\le 0,
\\
0, & j >0,
\end{cases}
\\
g^{R}_{n}(j) :=
\begin{cases}
0, & j < n,
\\
\gamma (j)-\sum _{k=1}^{n-1} g_{n}(k) \gamma (j-k), & j \ge n,
\end{cases}
\end{aligned}
%
\end{equation}
and extend the definition of $g_{n}(\cdot )$ to all integers by setting
$g_{n}(k)=0$ for $k\in \mathbb Z \setminus \{1,...,n-1\}$. Then equation
\eqref{maineq} can be rewritten as
%
\begin{equation}
\label{form}
g^{L}_{n}(j)+g^{R}_{n}(j) + \sum _{k=-\infty}^{\infty} g_{n}(k)
\gamma (j-k)=\gamma (j), \quad j\in \mathbb Z.
\end{equation}
It follows from \eqref{sigex}--\eqref{covex} that the prediction error
and the partial correlation are related to the sequences in
\eqref{gLgR} through the formulas:
%
\begin{equation}
\label{sa}
\sigma ^{2}(n) = g^{L}_{n}(0) \quad \text{and}\quad \alpha (n) =
\frac{g^{R}_{n}(n)}{g^{L}_{n}(0)} \quad n\ge 2.
\end{equation}

\subsection{The Hilbert problem}
\label{sec2.2}

As explained in the Introduction, our method requires that the spectral
density $f_{0}(\cdot )$ of the base process that drives the ARMA model
\eqref{f} admits a sectionally holomorphic extension to the complex plane,
cut along a simple curve. In the case of fGn with spectral density
\eqref{fGnf}, this extension is shown in Section \ref{sec:Q} to be given
by the formula
%
\begin{equation}
\label{Qdef}
Q(z) = \frac{1}{4\pi} (z^{-1} -2 + z) \big(\mathrm{Li}_{-2d-1}(z)+
\mathrm{Li}_{-2d-1}(z^{-1})), \quad z\in \mathbb C\setminus \mathbb R_{+},
\end{equation}
where $\mathrm{Li}_{s}(z)$ is the polylogarithm \cite{L81}. In Theorem
\ref{lem:Q} below, we will argue that $Q(z)$ is sectionally holomorphic
in $\mathbb C \setminus \mathbb R_{+}$, satisfies \eqref{f0Q}, and, for
$d\in (0,\frac {1} {2})$, it has a power-type singularity at $z=1$. Furthermore,
it is non-vanishing for $d\in (-\tfrac {1} {2},0)$, and has a pair of real
reciprocal zeros $\{s_{0},s_{0}^{-1}\}$ for
$d\in (0,\tfrac {1} {2})$, with some $s_{0} \in (-1,0)$.

Define the generating functions of the sequences in \eqref{gLgR}:
%
\begin{equation}
\label{GG}
G_{0}(z) := \sum _{j=-\infty}^{\infty }g^{L}_{n}(j) z^{j}, \quad G_{1}(z)
:= \sum _{j=-\infty}^{\infty }g^{R}_{n}(n-j)z^{j}, \qquad z\in
\mathbb C\setminus \overline D.
\end{equation}
These functions are holomorphic outside the unit disk. Let us extend their
definition to the open unit disk with a slit:
%
\begin{equation}
\label{GinD}
\begin{aligned}
G_{0}(z) :=\, & 2\pi \big(1- G(z) \big)
\frac{\theta (z)\theta (z^{-1})} {\phi (z)\phi (z^{-1})} Q(z)-z^{n}G_{1}(z^{-1}),
\\
G_{1}(z) :=\, & 2\pi \big(1- G(z^{-1})\big)z^{ n}
\frac{\theta (z)\theta (z^{-1})} {\phi (z)\phi (z^{-1})} Q(z) - z^{ n}
G_{0}(z^{-1}),
\end{aligned}
\quad z\in D\setminus [0,1],
\end{equation}
where $G(z)$ is the generating function (polynomial) of the solution to
\eqref{maineq}:
%
\begin{equation}
\label{Gz}
G(z) := \sum _{j=-\infty}^{\infty }g_{n}(j) z^{j}=\sum _{j=1}^{n-1} g_{n}(j)
z^{j}, \quad z\in \mathbb C.
\end{equation}
As will be shown later in Section \ref{sec:H1}, the definitions in
\eqref{GinD} are tailored so that $G_{0}(z)$ and $G_{1}(z)$ extend holomorphically
to the unit circle. This is due to equation \eqref{form} or, more precisely,
its Fourier domain equivalent \eqref{cond}. Thus, in fact, these functions
are sectionally holomorphic in $\mathbb C\setminus [0,1]$.

In view of \eqref{sa}, the quantities of interest are recovered from the
generating functions \eqref{GG} through the limits
%
\begin{equation}
\label{saG}
\sigma ^{2}(n) = \lim _{z\to \infty}G_{0}(z) \quad \text{and}\quad
\alpha (n) =
\frac{\lim _{z\to \infty}G_{1}(z)}{\lim _{z\to \infty}G_{0}(z)}.
\end{equation}

The next theorem, whose proof appears in Section \ref{sec:pr1}, shows that
the functions defined in \eqref{GG}-\eqref{GinD}, with a slight modification,
solve a certain Hilbert boundary value problem.

\begin{thmVTEX}%
\label{thm1}
Assume condition \textup{($A2_{\phi }$)} holds. Then, the functions
%
\begin{equation}
\label{Phi01}
\Phi _{0}(z) :=\, z^{q} \phi (z^{-1})G_{0}(z), \quad \Phi _{1}(z) :=
\, z^{q} \phi (z^{-1})G_{1}(z),
\end{equation}
are sectionally holomorphic in $\mathbb C\setminus [0,1]$ and solve the
Hilbert boundary value problem with the following specifications.

\medskip

\begin{enumerate}
\renewcommand{\theenumi}{H\arabic{enumi}}
\renewcommand{\labelenumi}{(\theenumi)}
\item[(H1)]
The boundary condition:
%
\begin{equation}
\label{bnd}
\begin{aligned}
& \Phi _{0}^{+}(t) -\frac{Q^{+}(t)}{Q^{-}(t)}\Phi _{0}^{-}(t) = t^{n+2q}
\Phi _{1}(t^{-1}) \frac{\phi (t^{-1})}{\phi (t)} \Big(
\frac{Q^{+}(t)}{Q^{-}(t)}- 1\Big),
\\
& \Phi _{1}^{+}(t)- \frac{Q^{+}(t)}{ Q^{-}(t)}\Phi _{1}^{-}(t) = t^{ n+2q}
\Phi _{0}(t^{-1})\frac{\phi (t^{-1})}{\phi (t)}\Big(
\frac{Q^{+}(t)}{ Q^{-}(t)} -1\Big),
\end{aligned}
\qquad t\in (0,1).
\end{equation}
\item[(H2)]
The growth rates:
%
\begin{equation}
\label{Phi_est}
\begin{aligned}
\big\{\Phi _{0}(z), \Phi _{1}(z)\big\} =
\begin{cases}
O(z^{-1} (\log z^{-1})^{-1-2d}), & z\to 0,
\\
O((z-1)^{-2(d\vee 0)}), & z\to 1,
\\
O(z^{q}), & z\to \infty .
\end{cases}
\end{aligned}
%
\end{equation}
\item[(H3)]
The algebraic condition:
%
\begin{equation}
\label{algcond}
\Phi _{0}(z)\phi (z) + z^{n+2q} \Phi _{1}(z^{-1}) \phi (z^{-1})=0,
\quad \forall \, z\in \{z_{1},...,z_{q}\}\cup \{z_{1}^{-1},...,z_{q}^{-1}
\}\cup Z_{Q},
\end{equation}
where $z_{j}$'s are the zeros of the MA polynomial $\theta (\cdot )$ and
$Z_{Q}$ is the zero set of $Q(\cdot )$.
\item[(H4)]
The scaling condition:
%
\begin{equation}
\label{scond}
\lim _{z\to 0} \frac{\Phi _{0}(z)}{Q(z)} = 2\pi \prod _{j=1}^{q}
\big(- z_{j}^{-1}\big), \qquad \lim _{z\to 0}
\frac{\Phi _{1}(z)}{Q(z)} =0.
\end{equation}
\end{enumerate}
\end{thmVTEX}

\subsection{Solution to the Hilbert problem}
\label{sec2.3}

In this section, we will argue that any solution to the Hilbert problem
\textup{(H1)}-\textup{(H4)} can be expressed in terms of solutions to a certain
system of integral and algebraic equations. To formulate the precise result
we will need to introduce several objects, some of which have a slightly
different form depending on the sign of $d$. Define the function
%
\begin{equation}
\label{X0def}
X_{0}(z) = \exp \left (\frac {1} {\pi }\int _{0}^{1}
\frac{\arg (Q^{+}(\tau ))}{\tau -z} d\tau \right ), \quad z\in
\mathbb C\setminus [0,1],
\end{equation}
where $\arg (\cdot )$ takes values in $(-\pi ,\pi ]$, and let
%
\begin{equation}
\label{XX0}
X(z) :=
\begin{cases}
z^{-1} X_{0}(z), & d\in (0,\tfrac {1} {2}),
\\
X_{0}(z), & d\in (-\tfrac {1} {2},0).
\end{cases}
%
\end{equation}
This function is sectionally holomorphic in
$\mathbb C\setminus [0,1]$ and satisfies the homogeneous boundary condition
%
\begin{equation}
\label{XcQ}
\frac{X^{+}(t)}{X^{-}(t)} = \frac{Q^{+}(t)}{Q^{-}(t)}, \quad t\in (0,1).
\end{equation}
Define the function
%
\begin{equation}
\label{hdef}
h(s):= \frac {1}{2i}\frac{1}{ \sin (\pi d)}
\frac{\phi (e^{s})}{\phi (e^{-s})} \left (
\frac {X(e^{s})}{X^{+}(e^{-s})}-\frac{X(e^{s})}{X^{-}(e^{-s})}
\right )e^{-2q s}, \quad s\in \mathbb R_{+}.
\end{equation}
A calculation in Lemma \ref{lem:tildeh} below shows that it is real-valued
and differentiable, with $\lim _{s\to 0} h(s)=1$. Consider the integral
equations, $j=0,...,q+1$,
%
\begin{equation}
\label{uweq}
\begin{aligned}
& u_{j }(t) =\
\phantom+
\frac {\sin (\pi d) }{ \pi } \int _{0}^{\infty }
\frac{h(s)e^{-n s}}{e^{s+t}-1} u_{j }(s) ds + e^{jt},
\\
& w_{j }(t) = -\frac {\sin (\pi d) }{ \pi } \int _{0}^{\infty }
\frac{h(s)e^{-n s}}{e^{s+t}-1} w_{j }(s) ds + e^{jt},
\end{aligned}
\qquad t\in \mathbb R_{+},
\end{equation}
It will be shown that, for all sufficiently large $n$, these equations
have unique solutions $u_{j,n}$ and $w_{j,n}$, such that the functions
$ u_{j,n}(t) -e^{jt} $ and $ w_{j,n}(t) -e^{jt} $ belong to
$L^{2}(\mathbb R_{+})$. Using these solutions, define
%
\begin{equation}
\label{SjDj}
\begin{aligned}
S_{j,n}(z) := &\
\phantom{+}
\frac {\sin (\pi d) }{ \pi } \int _{0}^{\infty }
\frac{h(s)e^{-n s}}{z e^{s }-1} u_{j,n}(s) ds + z^{j},
\\
D_{j,n}(z) := & -\frac {\sin (\pi d) }{ \pi } \int _{0}^{\infty }
\frac{h(s) e^{-n s}}{ze^{s}-1} w_{j,n}(s) ds + z^{j},
\end{aligned}
\qquad z\in \mathbb C\setminus [0,1].
\end{equation}

As mentioned above, the function $Q(z)$ defined in \eqref{Qdef} has a pair
of reciprocal zeros $\{s_{0}, s_{0}^{-1}\}$ with $s_{0}\in (-1,0)$ if
$d\in (0,\tfrac {1} {2})$, and it has no zeros if
$d\in (-\tfrac {1} {2},0)$. Accordingly, let us define
$z_{q+1} := s_{0}^{-1}$ for $d\in (0, \frac {1} {2})$ and
\begin{equation*}
q(d)=
\begin{cases}
q+1, & d\in (0,1/2),
\\
q, & d\in (-1/2,0).
\end{cases}
\end{equation*}

Consider a pair of systems of linear algebraic equations
%
\begin{equation}
\label{aeq}
\begin{aligned}
& \sum _{j=0}^{q(d)} \left ( X(z_{k})\phi (z_{k}) S_{j,n}(z_{k})+ z_{k}^{n+2q}
X(z_{k}^{-1}) \phi (z_{k}^{-1}) S_{j,n}(z_{k}^{-1})\right ) a_{j} =0,
\quad k=1,...,q(d),
\\
& \sum _{j=0}^{q(d)} S_{j,n}(0) a_{j} = \frac {1} {2} \sigma _{0}^{2}
\prod _{j=1}^{q(d)} (-1/z_{j}),
\end{aligned}
%
\end{equation}
%
\begin{equation}
\label{beq}
\begin{aligned}
& \sum _{j=0}^{q(d)} \left ( X(z_{k})\phi (z_{k}) D_{j,n}(z_{k}) - z_{k}^{n+2q}
X(z_{k}^{-1}) \phi (z_{k}^{-1}) D_{j,n}(z_{k}^{-1})\right ) b_{j} =0,
\quad k=1,...,q(d),
\\
& \sum _{j=0}^{q(d)} D_{j,n}(0) b_{j} = \frac {1} {2} \sigma _{0}^{2}
\prod _{j=1}^{q(d)} (-1/z_{j}),
\end{aligned}
%
\end{equation}
with respect to the unknowns $a_{0},...,a_{q(d)}$ and
$b_{0},..., b_{q(d)}$. The next theorem, proved in Section
\ref{sec:5}, provides a general solution to the Hilbert problem from Theorem
\ref{thm1}.

\begin{thmVTEX}%
\label{thm2}
Assume condition \textup{($A2_{\phi }$)} holds. Let
$\{\Phi _{0}(z),\Phi _{1}(z)\}$ be any solution to the Hilbert problem
\textup{(H1)}-\textup{(H4)} from Theorem \ref{thm1}. Then
\begin{align*}
\Phi _{0}(z) =\, & X(z) \Big(\sum _{j=0}^{q(d)} a_{j,n} S_{j,n}(z)+
\sum _{j=0}^{q(d)} b_{j,n} D_{j,n}(z)\Big),
\\
\Phi _{1}(z) =\, & X(z) \Big(\sum _{j=0}^{q(d)} a_{j,n} S_{j,n}(z)-
\sum _{j=0}^{q(d)} b_{j,n} D_{j,n}(z)\Big),
\end{align*}
where $(a_{0,n},...,a_{q(d),n})$ and $(b_{0,n},...,b_{q(d),n})$ solve
\eqref{aeq} and \eqref{beq}, respectively. For the particular solution
given by the functions defined in \eqref{Phi01},
%
\begin{equation}
\label{ab2salpha}
\sigma ^{2}(n) = a_{q(d),n}+b_{q(d),n} \quad \text{and}\quad \alpha (n)
= \frac{a_{q(d),n}-b_{q(d),n}}{a_{q(d),n}+b_{q(d),n}}.
\end{equation}
\end{thmVTEX}

\begin{rem}
\label{rem5}
Systems \eqref{aeq} and \eqref{beq} are guaranteed to have at least one
solution, corresponding to the functions in \eqref{Phi01}. Let us stress
that \eqref{ab2salpha} is claimed to hold only for this solution. Note,
however, that at this stage, uniqueness is not claimed. Such uniqueness
is crucial if the asymptotics of $\sigma ^{2}(n)$ and $\alpha (n)$ are
to be derived from \eqref{ab2salpha}. It is verified asymptotically as
$n\to \infty $ in Theorem \ref{thm3} in the next section.
\end{rem}

\subsection{Asymptotic analysis}
\label{sec2.4}

While excessively complicated for a fixed $n$, the systems of linear equations
from the previous subsection are more tractable asymptotically as
$n\to \infty $. In fact, they turn out to have a certain convenient Vandermonde
structure, see Section \ref{sec:V}, which yields the following result.

\begin{thmVTEX}%
\label{thm3}
Assume conditions \textup{(A1)} and \textup{($A2_{\phi }$)}, \textup{($A2_{\theta }$)} hold. For
$d\in (-\frac {1}{2}, \frac {1} {2})\setminus \{0\}$, systems
\eqref{aeq} and \eqref{beq} have unique solutions for all sufficiently
large $n$. The solutions satisfy
%
\begin{equation}
\label{ablim}
\begin{aligned}
a_{q(d),n} = & \frac {\sigma ^{2} } {2} \Big( 1+ \frac {d (1+d)} {n}
\Big) + O(n^{-2}),
\\
b_{q(d),n} = & \frac {\sigma ^{2} } {2} \Big(1 - \frac {d (1-d)} {n}
\Big) + O(n^{-2}),
\end{aligned}
\qquad n\to \infty
\end{equation}
where
$ \sigma ^{2} = \sigma ^{2}_{0}\bigg(\displaystyle \prod _{j:|z_{j}|<1}
z_{j}^{-2}\bigg) $ and $\sigma ^{2}_{0}$ is given by the Szeg\"{o}-Kolmogorov
formula
\begin{equation*}
\sigma ^{2}_{0} = 2\pi \exp \left (\frac {1}{2\pi } \int _{-\pi}^{\pi}
\log f_{0}(\lambda ) d\lambda \right ),
\end{equation*}
with $f_{0}(\cdot )$ being the fGn spectral density \eqref{fGnf}.
\end{thmVTEX}

The assertion of Theorem \ref{main-thm} follows by plugging estimates
\eqref{ablim} into \eqref{ab2salpha}.
%
\begin{cor}%
\label{cor:main}
The relative prediction error and the partial correlation coefficients
satisfy
\begin{equation*}
\sigma ^{2}(n) - \sigma ^{2} = \sigma ^{2} \frac{d^{2}} {n} + O(n^{-2}),
\quad \alpha (n) = \frac {d}{n} + O(n^{-2}), \qquad n\to \infty .
\end{equation*}
\end{cor}

\section{The function $\mathbf{Q(z)}$ and its properties}
\label{sec:Q}

As previously mentioned, the key element of our approach is the sectionally
holomorphic extension $Q(z)$ of the spectral density of the innovations
driving the ARMA process. For fGn, this extension takes a rather complicated
form \eqref{Qdef} involving special functions - specifically, polylogarithms.
The following theorem derives the main properties of $Q(z)$ required for
our analysis.

\begin{thmVTEX}%
\label{lem:Q}
Let $d\in (-\frac {1} {2}, \frac {1} {2})\setminus \{0\}$.

\begin{enumerate}
\renewcommand{\theenumi}{\roman{enumi}}
\renewcommand{\labelenumi}{(\theenumi)}
\item[(i)]
The function $Q(\cdot )$ satisfies the symmetries
%
\begin{equation}
\label{Qsym}
Q(z)=Q(z^{-1}), \quad \overline{Q(z)}=Q(\overline z), \qquad z\in
\mathbb C\setminus \mathbb R_{+}.
\end{equation}
Its limits $Q^{\pm}(t) = \lim _{z\to t^{\pm}}Q(z)$ satisfy
%
\begin{equation}
\label{Qpmsym}
Q^{\pm }(t) = Q^{\mp}(t^{-1}), \quad \overline{Q^{\pm}(t)}= Q^{\mp}(t),
\qquad t\in \mathbb R_{+}\setminus \{1\},
\end{equation}
and
%
\begin{equation}
\label{ImQ}
\mathrm{sign}\big(\mathrm{Im}(Q^{+}(t))\big) =
\begin{cases}
-\mathrm{sign}(d), & t\in (0,1),
\\
\phantom +
\mathrm{sign}(d), & t\in (1,\infty ).
\end{cases}
%
\end{equation}
\item[(ii)]
Identity \eqref{f0Q} holds for the fGn spectral density
$f_{0}(\cdot )$ from \eqref{fGnf}.
\item[(iii)]
The following growth estimates hold:
%
\begin{equation}
\label{Qest}
Q(z) =
\begin{cases}
O\big(z^{-1} (\log z^{-1})^{-1-2d}\big), & z\to 0,
\\
O\big((z-1)^{-2d}\big), & z\to 1,
\\
O\big(z(\log z)^{-1-2d}\big), & z\to \infty .
\end{cases}
%
\end{equation}
\item[(iv)]
For $d\in (0,\frac {1} {2})$, $Q(\cdot )$ has at least one zero in the
interval $(-1,0)$.
\item[(v)]
The function $\eta (t)=\arg (Q^{+}(t))\in (-\pi ,\pi ]$ has derivatives
of all orders, satisfies the property
%
\begin{equation}
\label{etaeta}
\eta (t)=-\eta (t^{-1}), \quad t\in \mathbb R_{+}\setminus \{1\},
\end{equation}
and the estimates
%
\begin{align}
\label{etalim1}
\eta (t) =\, &- d\pi + O((1-t)^{2+2d}), & t\nearrow 1,
\\
\label{etalim0}
\eta (t) =\, & \pi \mathbf{1}_{\{d<0\}}+\frac {c} {\log t^{-1}} + O((
\log t^{-1})^{-2}), & t \searrow 0,
\end{align}
for some constant $c\in \mathbb R$ (dependent on $d$).
\end{enumerate}
\end{thmVTEX}

\begin{proof}
See Appendix \ref{ssec:Qprop}. 
\end{proof}

\subsection{The function $\mathbf{X_{0}(z)}$}
\label{sec3.1}

Another important element, closely related to $Q(\cdot )$, is the function
introduced in \eqref{X0def},
%
\begin{equation}
\label{X0}
X_{0}(z) = \exp \left (\frac {1}{ \pi }\int _{0}^{1}
\frac{\eta (\tau )}{\tau -z}d\tau \right ), \quad z\in \mathbb C
\setminus [0,1],
\end{equation}
where $\eta (\tau ) =\arg (Q^{+}(\tau ))$, cf. Theorem \ref{lem:Q}\textup{(v)}.
In view of \eqref{Qpmsym},
\begin{equation*}
\frac{Q^{+}(t)}{Q^{-}(t)}= \exp \big(2i\arg (Q^{+}(t))\big), \quad t
\in (0,1).
\end{equation*}
Therefore, by the Sokhotski-Plemelj theorem, $X_{0}(z)$ solves the homogeneous
Hilbert problem \eqref{XcQ}. The next lemma summarizes its essential properties.

\begin{lem}%
\label{lem:X0}
For each $d\in (-\frac {1} {2}, \frac {1} {2})\setminus \{0\}$, there exist
nonzero constants $a\in \mathbb R$ and $c_{1},c_{2}\in \mathbb C$ such
that
%
\begin{equation}
\label{X0est}
X_{0}(z) =
\begin{cases}
c_{1} z^{- \mathbf{1}_{\{d<0\}}} (\log z)^{a} \big(1+O\big((\log z)^{-1}
\big)\big), & z\to 0,
\\
c_{2} (z-1)^{-d} \big(1+O(z-1)\big), & z\to 1,
\\
1+O(z^{-1}), & z\to \infty .
\end{cases}
%
\end{equation}
\end{lem}

\begin{proof}
See Appendix \ref{ssec:X0}. 
\end{proof}

\subsection{One identity and its implications}
\label{sec3.2}

Define
%
\begin{equation}
\label{psizdef}
\psi (z) := \exp \left (-\frac{1}{2\pi i}\oint _{\partial D}
\frac{\log Q(\zeta )}{\zeta -z}d\zeta \right ), \quad z\in \mathbb C
\setminus \partial D.
\end{equation}
This function is sectionally holomorphic in
$\mathbb C\setminus \partial D$. It arises in the analysis of the corresponding
infinite predictor problem. In particular, due to the Szeg\"{o}-Kolmogorov
formula \eqref{SKfla} to $f_{0}(\cdot )$, we have
%
\begin{equation}
\label{psi0}
\psi (0) = \exp \left (-\frac{1}{2\pi i}\int _{-\pi}^{\pi}
\frac{\log f_{0}(\lambda )}{e^{i\lambda} }de^{i\lambda}\right )=
\frac{2\pi}{\sigma _{0}^{2}}.
\end{equation}
The next lemma formulates a useful relation between $X_{0}(z)$ and
$Q(z)$. In the course of its proof, the precise number of zeros of
$Q(z)$ is revealed, cf. Theorem \ref{lem:Q} \textup{(iv)}.

\begin{lem}%
\label{lem:identity}
\
\medskip

1. For $d\in (0,\frac {1} {2})$, $Q(z)$ has a pair of zeros
$\{s_{0}, s_{0}^{-1}\}$ with $s_{0} \in (-1,0)$ and satisfies
%
\begin{equation}
\label{XQfla}
X_{0}(z) = \psi (z) Q(z)\dfrac{z}{z-s_{0}}, \quad z\in D.
\end{equation}

2. For $d\in (-\frac {1} {2},0)$, $Q(z)$ is non-vanishing in
$\mathbb C\setminus \mathbb R_{+}$, and satisfies
%
\begin{equation}
\label{XQfla2}
X_{0}(z) = \psi (z) Q(z), \quad z\in D.
\end{equation}
\end{lem}
\begin{proof}
See Appendix \ref{app:A}. 
\end{proof}

This lemma has the following important implication, which along with
\eqref{Qest}, refines the asymptotic bound for $X_{0}(z)$ as
$z\to 0$ in \eqref{X0est}.
%
\begin{cor}
\label{cor3.4}
%
\begin{equation}
\label{X0Q}
\begin{aligned}
\lim _{z\to 0}\frac{z^{-1}X_{0}(z)}{Q(z)} & = -
\frac{ 2\pi }{\sigma ^{2}_{0}}\frac {1} {s_{0}}, && d\in (0,
\tfrac {1} {2}),
\\
\lim _{z\to 0}\frac{X_{0}(z)}{Q(z)} & =
\frac{ 2\pi }{\sigma ^{2}_{0}}, & & d\in (-\tfrac {1} {2},0).
\end{aligned}
%
\end{equation}
\end{cor}

\begin{proof}
In view of \eqref{psi0}, the claim follows by rearranging the formulas
in Lemma \ref{lem:identity} and taking the limit $z\to 0$.
\end{proof}

\section{Proof of Theorem \ref{thm1}}
\label{sec:pr1}

\subsection{Proof of \textup{(H1)}}
\label{sec:H1}

In view of asymptotics \eqref{asym_gamma_fGn}, the spectral density of
the fGn is defined pointwise for all $\lambda \ne 0$. Since the summation
in \eqref{form} is, in fact, finite, the Fourier series with the coefficients
$g^{L}_{n}(\cdot )$ and $g^{R}_{n}(\cdot )$ are defined pointwise as well
and, by the convolution theorem, satisfy the equation:
%
\begin{equation}
\label{cond}
\widehat g^{L}_{n}(\lambda )+\widehat g^{R}_{n}(\lambda ) + 2\pi
\widehat g_{n}(\lambda ) f(\lambda )=f(\lambda ), \quad \lambda \in (-
\pi ,\pi ]\setminus \{0\}.
\end{equation}
By Abel's theorem, the generating functions in equation \eqref{GG} have
the following limits at any $\lambda \ne 0$:
%
\begin{align}
\label{limG0}
\lim _{z\to e^{i\lambda},\, |z|>1}G_{0}(z) =&\ 2\pi \widehat g^{L}_{n}(
\lambda ),
\\
\label{limG1}
\lim _{z\to e^{i\lambda},\, |z|<1}G_{1}(z^{-1}) =&\ 2\pi \widehat g^{R}_{n}(
\lambda )e^{-i\lambda n}.
\end{align}
The function $G_{0}(z)$ is continuous across the unit circle: for any
$\lambda \in (-\pi ,\pi ]\setminus \{0\}$,
\begin{equation*}
\lim _{z\to e^{i\lambda},\, |z|<1}G_{0}(z)=2\pi \big(1-2\pi
\widehat g_{n}(\lambda ) \big) f(\lambda )- 2\pi \widehat g^{R}_{n}(
\lambda ) = 2\pi \widehat g^{L}(\lambda ) = \lim _{z\to e^{i\lambda},
\, |z|>1}G_{0}(z).
\end{equation*}
The first equality holds by definitions \eqref{GinD} - \eqref{Gz}, limit
\eqref{limG1} and property \eqref{f0Q}, the second equality is due to
\eqref{cond}, and the last equality holds by \eqref{limG0}. Thus, by continuity
principle, $G_{0}(z)$ can be extended holomorphically to
$\partial D\setminus \{1\}$. The same argument applies to $G_{1}(z)$.

Therefore $G_{0}(z)$ and $G_{1}(z)$ are meromorphic on
$\mathbb C\setminus [0,1]$, with poles at the zeros of
$\phi (z^{-1})$ and, by construction \eqref{GinD}, satisfy the equation
%
\begin{equation}
\label{QQ}
G_{0}(z) +z^{n}G_{1}(z^{-1})+ 2\pi \big( G(z)-1\big)
\frac{\theta (z)\theta (z^{-1})} {\phi (z)\phi (z^{-1})} Q(z)= 0,
\quad z\in \mathbb C\setminus \mathbb R_{+}.
\end{equation}
Consequently, by assumption \textup{($A2_{\phi }$)}, the functions defined in
\eqref{Phi01} are holomorphic in $\mathbb C\setminus [0,1]$ and have finite
limits across the interval $(0,1)$ in both the lower and upper half-planes,
with possible singularities at the endpoints. This implies that
$\Phi _{0}(z)$ and $\Phi _{1}(z)$ are sectionally holomorphic in
$\mathbb C\setminus [0,1]$ and satisfy the equation
%
\begin{equation}
\label{zeq}
2\pi \big(1- G(z)\big)\theta (z)z^{q}\theta (z^{-1})=
\frac{\Phi _{0}(z)\phi (z) + z^{n+2q} \Phi _{1}(z^{-1}) \phi (z^{-1})}{ Q(z)},
\quad z\in \mathbb C \setminus \mathbb R_{+},
\end{equation}
obtained by rearranging \eqref{QQ}.

The function in the left hand side of \eqref{zeq} is a polynomial; thus,
all singularities in the right hand side are removable. Removal of the
jump discontinuity on $\mathbb R_{+}$ implies that for any
$\tau \in \mathbb R_{+}\setminus \{1\}$
%
\begin{equation}
\label{limlim}
\lim _{z\to \tau ^{+}}
\frac{\Phi _{0}(z)\phi (z) + z^{n+2q} \Phi _{1}(z^{-1}) \phi (z^{-1}) }{ Q(z)}
= \lim _{z\to \tau ^{-}}
\frac{\Phi _{0}(z)\phi (z) + z^{n+2q} \Phi _{1}(z^{-1}) \phi (z^{-1})}{ Q(z)}.
\end{equation}
For $\tau :=t\in (0,1)$, a direct calculation reduces this condition to
the first equation in \eqref{bnd}. For $\tau \in (1,\infty )$,
\eqref{limlim} yields
\begin{equation*}
\frac{\Phi _{0}(\tau )\phi (\tau ) + \tau ^{n+2q} \Phi _{1}^{-}(\tau ^{-1}) \phi (\tau ^{-1})}{ Q^{+}(\tau )}
=
\frac{\Phi _{0}(\tau )\phi (\tau ) + \tau ^{n+2q} \Phi _{1}^{+}(\tau ^{-1}) \phi (\tau ^{-1})}{ Q^{-}(\tau )}.
\end{equation*}
The second equation in \eqref{bnd} is obtained by changing the variable
to $t=\tau ^{-1}\in (0,1)$ and using the symmetries in \eqref{Qpmsym}.

\subsection{Proof of \textup{(H2)}}
\label{sec4.2}

By definitions \eqref{GinD} and \eqref{Phi01},
\begin{equation*}
\Phi _{0}(z)/Q(z) = O(1)\quad \text{and}\quad \Phi _{1}(z)/Q(z) = O(1),
\quad z\to 0,
\end{equation*}
and the first rate in \eqref{Phi_est} follows from \eqref{Qest}. To check
the second rate, substitute \eqref{gLgR} into \eqref{GG}. A direct calculation
shows that
\begin{equation*}
G_{0}(z) = \Big(\tfrac {1} {2} (z-2+z^{-1})\mu (z^{-1}) +
\tfrac {1} {2}\Big)\left (1 -G(z)\right )+\sum _{k=1}^{n-1} g_{n}(k)
\sum _{j=1}^{k}\gamma (k-j)z^{j}, \quad |z|>1,
\end{equation*}
where we defined $ \mu (z) :=\sum _{k=1}^{\infty }k^{2d+1} z^{k} $. It
satisfies, see Appendix \ref{ssec:Qprop}, 
\begin{equation*}
\big|(z-1)^{2+2d} \mu (z)\big| \xrightarrow[z\to 1]{}\Gamma (2+2d),
\end{equation*}
and thus, for $d\in (0,\tfrac 1 2)$,
\begin{equation*}
G_{0}(z)=O((1-z)^{-2d}), \quad z\to 1, \ z\in \mathbb C\setminus
\overline D.
\end{equation*}
This along with definition \eqref{GinD} and asymptotics \eqref{Qest} implies
that
\begin{equation*}
G_{1}(z)=O((1-z)^{-2d}), \quad z\to 1, \ z\in D.
\end{equation*}
The same argument, applied first to $G_{1}(z)$ defined in \eqref{GG} outside
$D$ and then to $G_{0}(z)$ defined in \eqref{GinD} inside $D$ verifies
the same rates with $G_{0}(z)$ and $G_{1}(z)$ being reversed. Plugging
them into \eqref{Phi01} yields the second rate in \eqref{Phi_est} for $d\in (0,\tfrac 1 2)$. The rate for $d\in (-\tfrac 1 2,0)$ is obtained similarly. The
last rate follows from \eqref{Phi01}, since, by definitions
\eqref{GG}, $ G_{0}(z) \sim g^{L}(0) $ and $ G_{1}(z)\sim g^{R}(n) $ as
$z\to \infty $.

\subsection{Proof of \textup{(H3)}}
\label{sec4.3}

The expression in \eqref{algcond} must vanish at $z_{j}$'s and their reciprocals,
since the left hand side in \eqref{zeq} vanishes at the zeros of both the
MA polynomial $\theta (z)$ and its reciprocal polynomial
$\widetilde \theta (z):=z^{q} \theta (z^{-1})$. In addition, if
$Q(z)$ has zeros, they must be shared with the numerator in
\eqref{zeq}, for the poles to be removable.

\subsection{Proof of \textup{(H4)}}
\label{sec4.4}

By definition, the polynomial in \eqref{Gz} satisfies $G(0)=0$ and
$\deg (G)<n$, and hence, for all $n$ large enough, if follows from
\eqref{zeq} that
\begin{equation*}
\lim _{z\to 0}\frac{\Phi _{0}(z)\phi (z) }{ Q(z)}=\, 2\pi \big(1- G(0)
\big)\theta (0)\widetilde \theta (0) =2\pi \theta (0)
\widetilde \theta (0) = 2\pi \prod _{j=1}^{q} \Big(\frac {1}{-z_{j}}
\Big),
\end{equation*}
where we used \eqref{theta_poly}. This verifies the first condition in
\eqref{scond} since $\phi (0)=1$. The second condition holds since
$Q(z)=Q(z^{-1})$ and
\begin{equation*}
\lim _{z\to 0}\frac{\Phi _{1}(z) \phi (z)}{ Q(z)} = \lim _{z\to
\infty}\frac{\Phi _{1}(z^{-1}) \phi (z^{-1})}{ Q(z^{-1})}= \lim _{z
\to \infty}\frac{\Phi _{1}(z^{-1}) \phi (z^{-1})}{ Q(z)}= \lim _{z
\to \infty}z^{-n}2\pi \big(1- G(z)\big)\widetilde \theta (0) \theta (0)=
0.
\end{equation*}

\section{Proof of Theorem \ref{thm2}}
\label{sec:5}

We consider the case $d\in (0,\frac {1}{2})$ in full detail. The complementary
case, $d\in (-\frac {1}{2}, 0)$, is treated similarly as briefly explained
in Section \ref{sec:case2}.

\subsection{The integral equations}
\label{sec5.1}

For $X(z)$ in \eqref{XX0}, define the functions
%
\begin{equation}
\label{SD}
S(z): = \frac{\Phi _{0}(z) +\Phi _{1}(z)}{2X(z)}, \quad D(z): =
\frac{\Phi _{0}(z) -\Phi _{1}(z)}{2X(z)},
\end{equation}
where $\Phi _{0}(z)$ and $\Phi _{1}(z)$ solve the Hilbert problem
\textup{(H1)}-\textup{(H4)} formulated in Theorem \ref{thm1}. All functions in
the right-hand side of \eqref{SD} are sectionally holomorphic in
$\mathbb C\setminus [0,1]$, and since $X(z)$ is non-vanishing, the functions
$S(z)$ and $D(z)$ are also sectionally holomorphic in
$\mathbb C\setminus [0,1]$. In view of \eqref{Phi_est} and
\eqref{X0est},
%
\begin{equation}
\label{SD_est}
\begin{aligned}
\big\{S(z), D(z)\big\} =
\begin{cases}
O(1), & z\to 0,
\\
O((z-1)^{-d}), & z\to 1,
\\
O(z^{q+1}), & z\to \infty .
\end{cases}
\end{aligned}
%
\end{equation}
where the rate as $z\to 0$ is due to \eqref{Qest} and \eqref{X0Q}. Substituting
equation \eqref{XcQ} into \eqref{bnd} and rearranging the terms shows that
these functions satisfy the \emph{decoupled} boundary conditions:
%
\begin{equation}
\label{SDeq}
\begin{aligned}
S^{+}(t)\, -S^{-}(t) = & - 2i \sin (\pi d) \widetilde h(t) t^{n} S(t^{-1}),
\\
D^{+}(t) -D^{-}(t) = &\
\phantom+
2i \sin (\pi d) \widetilde h(t) t^{n} D(t^{-1}),
\end{aligned}
\qquad t\in (0,1),
\end{equation}
where we defined
%
\begin{equation}
\label{tildeh}
\widetilde h(t):= -\frac{t^{2q}}{2i\sin (\pi d)}
\frac{\phi (t^{-1})}{\phi (t)} \frac{X(t^{-1})}{X^{+}(t)} \Big(
\frac{X^{+}(t)}{X^{-}(t)}- 1\Big).
\end{equation}
%
\begin{lem}%
\label{lem:tildeh}
The function in \eqref{tildeh} is real-valued and differentiable, satisfying
%
\begin{equation}
\label{hest}
\widetilde h(t) =
\begin{cases}
1+ o(1-t), & t\nearrow 1,
\\
O(t^{2q-p}), & t\searrow 0.
\end{cases}
%
\end{equation}
\end{lem}

\begin{proof}
Applying the Sokhotski-Plemelj theorem to \eqref{X0} gives
\begin{align*}
& \frac{X(t^{-1})}{X^{+}(t)} \Big(\frac{X^{+}(t)}{X^{-}(t)}- 1\Big)=t^{2}
\frac{X_{0}(t^{-1})}{X^{+}_{0}(t)} \Big(
\frac{X_{0}^{+}(t)}{X_{0}^{-}(t)}- 1\Big) =
\\
& t^{2}\exp \left (\frac {1} { \pi} \int _{0}^{1}
\frac{\eta (\tau )}{\tau -t^{-1}}d\tau -\frac {1} { \pi} \dashint _{0}^{1}
\frac{\eta (\tau )}{\tau -t}d\tau - \frac {2i} {2} \eta (t)\right )
\Big(e^{2 i \eta (t)}- 1\Big) =
\\
& 2i \sin \eta (t) t^{2} \exp \left (\frac {1} { \pi} \int _{0}^{1}
\frac{\eta (\tau )}{\tau -t^{-1}}d\tau -\frac {1} { \pi} \dashint _{0}^{1}
\frac{\eta (\tau )}{\tau -t}d\tau \right ),
\end{align*}
where the dash integral stands for the Cauchy principal value. By Theorem
\ref{lem:Q} \textup{(v)}, $\eta (\cdot )$ is smooth, and hence this function
is differentiable. It follows that under assumption \textup{($A2_{\phi }$)}, the
function $\widetilde h(\cdot )$ is real-valued and differentiable on
$(0,1)$. The limit as $t\to 1$ is obtained by computing the principle value
and using the estimate \eqref{etalim1}. The asymptotics as $t\to 0$ holds
due to \eqref{X0est}.
\end{proof}

Given the estimates in \eqref{SD_est} and Lemma \ref{lem:tildeh}, the functions
on the right hand side of \eqref{SDeq} are H\"older continuous and integrable
for all $n$ large enough. We can therefore apply the Sokhotski-Plemelj
theorem to \eqref{SDeq} to obtain the representations
%
\begin{equation}
\label{SDrep}
\begin{aligned}
S(z) = & -\frac {\sin (\pi d)} {\pi }\int _{0}^{1}
\frac{ \widetilde h(\tau )\tau ^{n}}{\tau -z}S(\tau ^{-1})d\tau +
\sum _{j=0}^{q+1} a_{j} z^{j},
\\
D(z) = &
\phantom +
\frac {\sin (\pi d)} {\pi }\int _{0}^{1}
\frac{ \widetilde h(\tau )\tau ^{n}}{\tau -z}D(\tau ^{-1})d\tau +
\sum _{j=0}^{q+1} b_{j} z^{j},
\end{aligned}
\qquad z\in \mathbb C\setminus [0,1],
\end{equation}
where $a_{j}$ and $b_{j}$ are some constants, determined by the particular
solution to the Hilbert problem
$\Phi _{0}(\cdot ), \Phi _{1}(\cdot )$ in \eqref{SD}. Indeed, denote by
$I(z)$ the integral term in the right hand side of the equation for
$S(z)$ in \eqref{SDrep}. The standard properties of the Cauchy integral
imply that $I(z)$ satisfies the same boundary condition as $S(z)$ in
\eqref{SDeq} and also the estimates in \eqref{SD_est} as $z\to 0$ and
$z\to 1$. Consider the difference $\delta (z)=S(z)-I(z)$. This function
is sectionally holomorphic in $\mathbb C\setminus [0,1]$ and it is continuous
on $(0,1)$. Hence $\delta (\cdot )$ can be extended holomorphically to
$(0,1)$. Moreover, $\delta (z)=O(1)$ as $z\to 0$ and
$\delta (z)=O((z-1)^{-d})$ as $z\to 1$ and hence $\delta (z)$ can be extended
holomorpically to $z=0$ and $z=1$ as well by Riemann's extension theorem.
It follows that upon extension, $\delta (z)$ becomes an entire function
and $\delta (z)=O(z^{q+1})$ as $z\to \infty $. Hence by Liouville's theorem
$\delta (z)$ is a polynomial of degree $q+1$, that is, the representation
for $S(z)$ in \eqref{SDrep} holds for some constants $a_{j}$. The same
is true for the representation of $D(z)$.

\begin{rem}
\label{rem6}
The application of the Sokhotski-Plemelj theorem and verification of the
resulting representation \eqref{SDrep} rely on the specific choice of the
factor $z^{-1}$ in \eqref{XX0}. To some extent, this choice is a matter
of convenience. More generally, we could have defined
\begin{equation*}
X(z)=z^{m} (z-1)^{k} X_{0}(z)
\end{equation*}
for some integers $k$ and $m$. As before, such a function would be sectionally
holomorphic in $\mathbb C\setminus [0,1]$, non-vanishing, and satisfy
\eqref{XcQ}. The integers $k$ and $m$ control the behavior of $X(z)$ as
$z$ approaches $0$ and $1$, respectively, and, once chosen, they also determine
the growth of $X(z)$ as $z$ approaches infinity. This, in turn, governs
the asymptotics of $S(z)$ and $D(z)$ in \eqref{SD_est}. Any $k\le 0$ ensures
that
\begin{equation*}
\{S(z), D(z)\}=O((z-1)^{-k-d}), \quad z\to 1,
\end{equation*}
rendering the functions on the right-hand side of \eqref{SDeq} integrable
and thus permitting the application of the Sokhotski--Plemelj theorem.
The function in \eqref{XX0} corresponds to the simplest choice,
$k=0$. Choosing $m=-1$ implies that $S(z)=O(1)$ and $D(z)=O(1)$ as z approaches
0. This matches the behavior of the Cauchy integrals near $z=0$, which
is essential for establishing uniqueness. For other choices of $m$ and
$k$, the corresponding representations can be more cumbersome, potentially
involving polynomials of higher degrees.
\end{rem}

By evaluating equations \eqref{SDrep} at $z= e^{t}$ with
$t\in (0,\infty )$ and changing the integration variable to
$r=\log \tau ^{-1}$ we obtain
%
\begin{equation}
\label{SDet}
\begin{aligned}
S(e^{t}) = &
\phantom +
\frac {\sin (\pi d)} {\pi }\int _{0}^{\infty }
\frac{ h(r)e^{-n r}}{ e^{r+t}-1}S(e^{r}) dr + \sum _{j=0}^{q+1} a_{j} e^{jt},
\\
D(e^{t}) = & -\frac {\sin (\pi d)} {\pi }\int _{0}^{\infty }
\frac{ h(r)e^{-n r}}{ e^{r+t}-1}D(e^{r}) dr + \sum _{j=0}^{q+1} b_{j} e^{jt},
\end{aligned}
\qquad t\in \mathbb R_{+},
\end{equation}
where $h(r):= \widetilde h(e^{-r})$, cf. \eqref{hdef}. Define the integral
operator
%
\begin{equation}
\label{An}
(A_{n}f)(t) = \frac {\sin (\pi d)} {\pi }\int _{0}^{\infty }
\frac{ h(r)e^{-n r}}{ e^{r+t}-1}f(r) dr,
\end{equation}
acting on real-valued functions on $\mathbb R_{+}$.

\begin{lem}%
\label{lem:An}
There exist constants $n_{0}\ge 1$ and $\varepsilon \in (0,1)$, such that
for all $n> n_{0}$, the operator $A_{n}$ is a contraction in
$L^{2}(\mathbb R_{+})$:
%
\begin{equation}
\label{contr}
\|A_{n}f\|\le (1-\varepsilon ) \|f\|, \quad \forall f\in L^{2}(
\mathbb R_{+}).
\end{equation}
\end{lem}

\begin{proof}
See Appendix \ref{app:Aop}. 
\end{proof}

Let us now consider the integral equations
%
\begin{equation}
\label{pqeq}
\begin{aligned}
p(t) =
\phantom +
(A_{n}p)(t) + \sum _{j=0}^{q+1} a_{j} e^{jt},
\\
q(t) = -(A_{n}q)(t) + \sum _{j=0}^{q+1} b_{j} e^{jt},
\end{aligned}
\qquad t\in \mathbb R_{+}.
\end{equation}
It follows from \eqref{SDet} that the functions $S(e^{t})$ and
$D(e^{t})$ are their particular solutions. The next lemma shows that these
solutions belong to the function class
$ \mathcal L_{N} = \bigcap _{n\ge N}\big\{f: \|A_{n}f\|<\infty \big\} $
for some $N$.

\begin{lem}%
\label{lem:SDsol}
The functions $S(e^{t})$ and $D(e^{t})$, defined in \eqref{SD}, belong
to $\mathcal L_{N}$ with $N=n_{0}+q+3$, where $n_{0}$ is defined in Lemma
\ref{lem:An}.
\end{lem}

\begin{proof}
Since $S(z)$ is holomorphic in $\mathbb C\setminus [0,1]$, the function
$f(t):= S(e^{t})$ is continuous on $(0,\infty )$. From \eqref{SD_est},
it follows that $f(t)=O(t^{-d})$ as $t\to 0$ and
$f(t)=O(e^{(q+1)t})$ as $t\to \infty $. Consequently the function
$\widetilde f(t):= e^{-(q+2)t}f(t)$ belongs to
$L^{2}(\mathbb R_{+})$. The claim holds for $S(e^{t})$ since, due to Lemma
\ref{lem:An},
\begin{equation*}
\|A_{n}f\|=\|A_{n-q-2}\widetilde f\|<\infty ,\quad \forall n\ge n_{0}+q+3.
\end{equation*}
The same argument applies to $D(e^{t})$.
\end{proof}

\begin{lem}%
\label{lem:eqs}
Equations \eqref{uweq} and \eqref{pqeq} have unique solutions in
$\mathcal L_{N}$ for all $n\ge N:=n_{0}+q+3$. These solutions satisfy
\begin{equation*}
p_{n}(t) = \sum _{j=0}^{q+1} a_{j} u_{j,n}(t),\quad q_{n}(t) = \sum _{j=0}^{q+1}
b_{j} w_{j,n}(t).
\end{equation*}
\end{lem}

\begin{proof}
Consider, e.g., the first equation in \eqref{pqeq}:
%
\begin{equation}
\label{pAnf}
p = A_{n} p + f
\end{equation}
where $\displaystyle f(t) = \sum _{j=0}^{q+1} a_{j} e^{jt}$ and rewrite
it as
%
\begin{equation}
\label{p-f}
p -f = A_{n} (p-f) + A_{n} f.
\end{equation}
As in the proof of Lemma \ref{lem:SDsol}, the function $A_{n} f$ belongs
to $L^{2}(\mathbb R_{+})$ for all $n\ge N$. By Lemma \ref{lem:An},
$A_{n}$ is a contraction in $L^{2}(\mathbb R_{+})$ for all $n\ge N$. Thus,
the unique solution $p-f\in L^{2}(\mathbb R^{+})$ to equation
\eqref{p-f} is given by the Neumann series. From \eqref{pAnf} it follows
that $A_{n} p \in L^{2}(\mathbb R_{+})$ for all $n\ge N$. The same argument
applies to the rest of the equations. The claimed identities follow from
linearity and uniqueness of solutions.
\end{proof}

\begin{cor}
\label{cor5.5}
Let $S_{j,n}(z)$ and $D_{j,n}(z)$ be defined by \eqref{SjDj} where
$u_{j,n}$ and $w_{j,n}$ are the unique solutions to \eqref{uweq} guaranteed
by Lemma \ref{lem:eqs} for all sufficiently large $n$. Then the functions
defined in \eqref{SD} satisfy
%
\begin{equation}
\label{SDjn}
S(z) = \sum _{j=0}^{q+1} a_{j} S_{j,n}(z), \quad D(z) = \sum _{j=0}^{q+1}
b_{j} D_{j,n}(z), \qquad z\in \mathbb \setminus [0,1].
\end{equation}
\end{cor}

\begin{proof}
By Lemma \ref{lem:SDsol} and Lemma \ref{lem:eqs}
\begin{equation*}
S(e^{t}) = \sum _{j=0}^{q+1} a_{j} S_{j,n}(e^{t}), \quad D(e^{t}) =
\sum _{j=0}^{q+1} b_{j} D_{j,n}(e^{t}), \qquad t\in \mathbb R_{+},
\end{equation*}
which implies \eqref{SDjn} by the Identity theorem.
\end{proof}

\subsection{The algebraic conditions}
\label{sec5.2}

At this stage the constants $a_{j}$ and $b_{j}$, determined by
$S(z)$ and $D(z)$ themselves, remain unknown. In view of \eqref{SD}, conditions
\eqref{algcond} for $z\in \{z_{1},...,z_{q}\}$ imply
\begin{equation*}
X(z_{k})(S(z_{k})+D(z_{k}))\phi (z_{k}) + z_{k}^{n+2q} X(z_{k}^{-1})(S(z_{k}^{-1})-D(z_{k}^{-1}))
\phi (z_{k}^{-1})=0, \quad k=1,...,q,
\end{equation*}
and, for $z\in \{z_{1}^{-1},...,z_{q}^{-1}\}$,
\begin{equation*}
X(z_{k} )(S(z_{k} )-D(z_{k} )) \phi (z_{k})+ z_{k}^{ n+2q } X(z_{k}^{-1})(S(z_{k}^{-1})+D(z_{k}^{-1}))
\phi (z_{k}^{-1}) =0, \quad k=1,...,q.
\end{equation*}
Adding and subtracting these equations yields
\begin{equation*}
\begin{aligned}
& X(z_{k})\, S(z_{k})\, \phi (z_{k}) + z_{k}^{n+2q} X(z_{k}^{-1})\, S(z_{k}^{-1})
\, \phi (z_{k}^{-1})=0,
\\
& X(z_{k}) D(z_{k}) \phi (z_{k}) - z_{k}^{n+2q} X(z_{k}^{-1}) D(z_{k}^{-1})
\phi (z_{k}^{-1})=0,
\end{aligned}
\qquad k=1,...,q.
\end{equation*}
Substituting the expressions from \eqref{SDjn} we obtain the equivalent
conditions:
%
\begin{align}
&
\nonumber
\sum _{j=0}^{q+1} \Big( X(z_{k}) \phi (z_{k}) S_{j,n}(z_{k}) + z_{k}^{n+2q}
X(z_{k}^{-1})\phi (z_{k}^{-1}) S_{j,n}(z_{k}^{-1})\Big) a_{j} =0,
\qquad k=1,...,q.
\\
&
\label{eqq}
\sum _{j=0}^{q+1} \Big(X(z_{k}) \phi (z_{k}) D_{j,n}(z_{k}) - z_{k}^{n+2q}
X(z_{k}^{-1}) \phi (z_{k}^{-1}) D_{j,n}(z_{k}^{-1})\Big)b_{j} =0,
\end{align}
A similar calculation applied to the zeros of $Q(z)$ yields two additional
equations:
%
\begin{equation}
\label{s0eq}
\begin{aligned}
& \sum _{j=0}^{q+1} \Big( X(s_{0}) \phi (s_{0}) S_{j,n}(s_{0}) + s_{0}^{n+2q}
X(s_{0}^{-1})\phi (s_{0}^{-1}) S_{j,n}(s_{0}^{-1})\Big) a_{j} =0,
\\
& \sum _{j=0}^{q+1} \Big(X(s_{0}) \phi (s_{0}) D_{j,n}(s_{0}) - s_{0}^{n+2q}
X(s_{0}^{-1}) \phi (s_{0}^{-1}) D_{j,n}(s_{0}^{-1})\Big)b_{j} =0.
\end{aligned}
%
\end{equation}
Finally, \eqref{X0Q} implies
\begin{equation*}
\begin{aligned}
& \lim _{z\to 0} \frac{X(z)(S(z)+D(z))}{Q(z)} = -
\frac{ 2\pi }{\sigma ^{2}_{0}}\frac {1} {s_{0}} (S(0)+D(0)),
\\
& \lim _{z\to 0} \frac{X(z)(S(z)-D(z))}{Q(z)} = -
\frac{ 2\pi }{\sigma ^{2}_{0}}\frac {1} {s_{0}} (S(0)-D(0)),
\end{aligned}
\end{equation*}
which, in view of \eqref{scond}, yields
\begin{equation*}
S(0) =D(0) = -\frac {1} {2} \sigma ^{2}_{0} s_{0} \prod _{j=1}^{q}
\big(- z_{j}^{-1}\big),
\end{equation*}
that is,
%
\begin{equation}
\label{moreeq}
\begin{aligned}
& \sum _{j=0}^{q+1} S_{j,n}(0)a_{j} = -\frac {1} {2} \sigma ^{2}_{0} s_{0}
\prod _{j=1}^{q} \big(- z_{j}^{-1}\big),
\\
& \sum _{j=0}^{q+1} D_{j,n}(0)b_{j} = -\frac {1} {2} \sigma ^{2}_{0} s_{0}
\prod _{j=1}^{q} \big(- z_{j}^{-1}\big).
\end{aligned}
%
\end{equation}
To recap, the $2q+4$ unknowns $a_{j}$ and $b_{j}$ in \eqref{SDjn} satisfy
the system of $2q+4$ linear algebraic equations which consists of
\eqref{eqq}, \eqref{s0eq} and \eqref{moreeq}.

The functions defined in \eqref{Phi01} correspond to a particular solution
to this system, denote it by $(a_{0,n},...,a_{q+1,n})$ and
$(b_{0,n},...,b_{q+1,n})$. It follows from \eqref{SDrep} that
\begin{equation*}
S(z) \sim \, a_{q+1,n}z^{q+1}, \quad D(z) \sim \, b_{q+1,n}z^{q+1},
\qquad z\to \infty .
\end{equation*}
Due to definitions \eqref{SD} and \eqref{XX0} and the estimate
$X_{0}(z)\sim 1$ as $z\to \infty $ from \eqref{X0est}, this translates
to
\begin{equation*}
\Phi _{0}(z) \sim \big(a_{q+1,n}+b_{q+1,n}\big) z^{q},\quad \Phi _{1}(z)
\sim \big(a_{q+1,n}-b_{q+1,n}\big) z^{q}, \qquad z\to \infty ,
\end{equation*}
and, in turn, due to definitions \eqref{Phi01}, to
\begin{equation*}
G_{0}(z) \sim \, a_{q+1,n}+b_{q+1,n},\quad G_{1}(z) \sim \, a_{q+1,n}-b_{q+1,n},
\qquad z\to \infty ,
\end{equation*}
where we used the normalization $\phi (0)=1$. The formulas
\eqref{ab2salpha} now follow from \eqref{saG}. This completes the proof
of Theorem \ref{thm2} for $d\in (0,\tfrac {1}{2})$.

\subsection{The case $d\in (-\tfrac {1} {2},0)$}
\label{sec:case2}

The proof for $d\in (-\tfrac {1} {2},0)$ is similar, but with a key difference:
in this case, the function $Q(z)$ has no zeros. Consequently, systems
\eqref{aeq} and \eqref{beq} have only $2q+2$ equations. However, unlike
in \eqref{SD_est}, the estimates from \eqref{Phi_est}, \eqref{X0est},
\eqref{X0Q} and definition \eqref{XX0} imply that $S(z)$ and $D(z)$ from
\eqref{SD} now satisfy $\{S(z),D(z)\}= O(z^{q})$ as $z\to \infty $. This
implies that there are also $2q+2$ unknowns, $a_{j}$ and $b_{j}$, now as
well.

\section{Proof of Theorem \ref{thm3}}
\label{sec6}

\subsection{The key approximations}
\label{sec6.1}

The asymptotic analysis as $n\to \infty $ relies on the approximation of
the coefficients for systems \eqref{aeq} and \eqref{beq}, as given by the
following theorem.

\begin{thmVTEX}%
\label{thethm}
For any fixed $z\in \mathbb C\setminus [0,1]$, the functions defined in
\eqref{SjDj} satisfy
%
\begin{equation}
\label{prove1}
\begin{aligned}
S_{j,n}(z) =\, & z^{j}+ \frac {\sin (\pi d) }{ \pi }
\frac{\lambda _{0}}{z -1} n^{-1}+ O(n^{-2}),
\\
D_{j,n}(z) =\, & z^{j} -\frac {\sin (\pi d) }{ \pi }
\frac{\mu _{0}}{z-1} n^{-1} + O(n^{-2}),
\end{aligned}
\qquad \text{as}\ n \to \infty ,
\end{equation}
where
%
\begin{equation}
\label{lamu}
\lambda _{0} = \int _{0}^{\infty }q_{1} (\tau ) e^{- \tau} d\tau ,
\quad \mu _{0} = \int _{0}^{\infty }p_{1}(\tau ) e^{- \tau} d\tau ,
\end{equation}
and $q_{1}(\cdot )$ and $p_{1}(\cdot )$ are the unique solutions to the
integral equations
%
\begin{equation}
\label{q1p1}
\begin{aligned}
q_{1} (t) = & \
\phantom{+}
\frac {\sin (\pi d) }{ \pi } \int _{0}^{\infty }\frac{e^{- r}}{r+t} q_{1}(r)
dr +1,
\\
p_{1} (t) = & -\frac {\sin (\pi d) }{ \pi } \int _{0}^{\infty }
\frac{e^{- r}}{r+t} p_{1}(r) dr +1,
\end{aligned}
\qquad t\in \mathbb R_{+},
\end{equation}
such that $q_{1}-1, p_{1}-1\in L^{2}(\mathbb R_{+})$.
\end{thmVTEX}

\begin{proof}
See Appendix \ref{app:B}. 
\end{proof}

The constants in \eqref{lamu} can be computed explicitly.

\begin{thmVTEX}%
\label{thm:const}%
\
The constants defined in \eqref{lamu} are given by
%
\begin{equation}
\label{values}
\lambda _{0} = \frac{\pi}{\sin (\pi d)}d (1+d), \quad \mu _{0} =
\frac{\pi}{\sin (\pi d)}d(1-d).
\end{equation}
\end{thmVTEX}

\begin{proof}
See Appendix \ref{app:C}. 
\end{proof}

\subsection{The asymptotic Vandermonde system}
\label{sec:V}

As in the previous section, we will detail the proof in the case
$d\in (0, \tfrac {1} {2})$ and briefly discuss the adjustments to the complementary
case $d\in (-\tfrac {1} {2},0)$ in Section \ref{sec:case2again}. To proceed,
let us define
%
\begin{equation}
\label{zeta}
\begin{aligned}
& \zeta _{k} :=
\begin{cases}
z_{k}, & |z_{k}| <1,
\\
z_{k}^{-1}, & |z_{k}|>1,
\end{cases}
\qquad k=1,...,q,
\\
& \zeta _{q+1} := s_{0},
\\
& \rho := \max _{1\le k\le q+1} |\zeta _{k}| \quad \text{and}\quad
\beta :=-\sigma ^{2}_{0}s_{0} \prod _{j=1}^{q} (-1/z_{j}),
\end{aligned}
%
\end{equation}
where $\rho <1$ due to assumption \textup{($A2_{\theta }$)}. Then, asymptotically
as $n\to \infty $, system \eqref{aeq} takes the form:
%
\begin{equation}
\label{geom}
\begin{aligned}
& \sum _{j=0}^{q+1} \Big( S_{j,n}(\zeta _{k})+ O(\rho ^{n})\Big) a_{j}
=0, \qquad k=1,...,q+1,
\\
& \sum _{j=0}^{q+1} S_{j,n}(0) a_{j} = \frac {1} {2} \beta ,
\end{aligned}
%
\end{equation}
where we used the property $X(z)\ne 0$ and assumption \textup{(A1)}. Furthermore,
by Theorems \ref{thethm} and \ref{thm:const}, this system can be reduced
to
%
\begin{equation}
\label{asys}
\begin{aligned}
& \sum _{j=0}^{q+1} \Big( \zeta _{k}^{j}+
\frac {\sin (\pi d) }{ \pi } \frac{\lambda _{0}}{\zeta _{k} -1} n^{-1}+
O(n^{-2}) \Big) a_{j} =0, \qquad k=1,...,q+1,
\\
& \sum _{j=0}^{q+1} \Big(\mathbf{1}_{\{j=0\}} -
\frac {\sin (\pi d) }{ \pi } \lambda _{0} n^{-1}+ O(n^{-2})\Big) a_{j}
= \frac {\beta }{2} .
\end{aligned}
%
\end{equation}
A similar calculation yields the asymptotic system for $b_{j}$'s:
%
\begin{equation}
\label{bsys}
\begin{aligned}
& \sum _{j=0}^{q+1} \Big( \zeta _{k}^{j}-
\frac {\sin (\pi d) }{ \pi } \frac{\mu _{0}}{\zeta _{k} -1} n^{-1}+ O(n^{-2})
\Big) b_{j} =0, \qquad k=1,...,q+1,
\\
& \sum _{j=0}^{q+1} \Big(\mathbf{1}_{\{j=0\}} +
\frac {\sin (\pi d) }{ \pi } \mu _{0} n^{-1}+ O(n^{-2})\Big) b_{j} =
\frac {\beta }{2} .
\end{aligned}
%
\end{equation}
Define the square Vandermonde matrix of size $q+2$
%
\begin{equation}
\label{vand}
V =
\begin{pmatrix}
1 & \zeta _{1} & \zeta _{1}^{2} & \cdots & \zeta _{1}^{q+1}
\\
1 & \zeta _{2} & \zeta _{2}^{2} & \cdots & \zeta _{2}^{q+1}
\\
\vdots & & & & \vdots
\\
1 & \zeta _{q+1} & \zeta _{q+1}^{2} & \cdots & \zeta _{q+1}^{q+1}
\\
1 & 0 & 0 & \cdots & 0
\end{pmatrix}
= : V(\zeta _{1},\cdots ,\zeta _{q+1}, 0)
\end{equation}
and the vectors in $\mathbb C^{q+2}$
%
\begin{equation}
\label{uoe}
u =
\begin{pmatrix}
\frac{1}{\zeta _{1}-1}
\\
\frac{1}{\zeta _{2}-1}
\\
\vdots
\\
\frac{1}{\zeta _{q+1}-1}
\\
-1
\end{pmatrix}
, \qquad \mathbf{1} =
\begin{pmatrix}
1
\\
1
\\
\vdots
\\
1
\\
1
\end{pmatrix}
, \qquad e =
\begin{pmatrix}
0
\\
0
\\
\vdots
\\
0
\\
1
\end{pmatrix}
.
\end{equation}
Then systems \eqref{asys} and \eqref{bsys} can be rewritten concisely as
%
\begin{equation}
\label{sysab}
\begin{aligned}
& \Big(V+ \frac {\sin (\pi d) }{ \pi } \lambda _{0}u\mathbf{1}^{\top }n^{-1}
+ O(n^{-2}) \Big)a= \frac {\beta }{2} e,
\\
& \Big(V- \frac {\sin (\pi d) }{ \pi n} \mu _{0}u\mathbf{1}^{\top }n^{-1}+
O(n^{-2})\Big)b= \frac {\beta }{2} e.
\end{aligned}
%
\end{equation}
Since we assumed that all $\zeta _{j}$'s in \eqref{vand} are distinct,
see Remark \ref{rem:3}, the matrix $V$ is invertible and hence the systems
in \eqref{sysab} have unique solutions for all sufficiently large
$n$. Recall that for any invertible matrix $A$ and any square matrix
$B$,
\begin{equation*}
(A-\varepsilon B)^{-1} = A^{-1} + \varepsilon A^{-1} B A^{-1} + O(
\varepsilon ^{2}), \quad \varepsilon \to 0,
\end{equation*}
where $O(\varepsilon ^{2})$ is a matrix whose norm is bounded by a constant
times $\varepsilon ^{2}$. Applying this asymptotic formula to the first
system in \eqref{sysab} we obtain
%
\begin{equation}
\label{aq}
\begin{aligned}
a_{q+1,n} =\, & e^{\top }\Big(V+ \frac {\sin (\pi d) }{ \pi n }
\lambda _{0}u\mathbf{1}^{\top}\Big)^{-1} \frac {\beta }{2} e +O(n^{-2})
=
\\
& \frac {\beta }{2} e^{\top }V^{-1} e -\frac {1} {n}
\frac {\beta }{2} \frac {\sin (\pi d) }{ \pi } \lambda _{0} e^{\top }V^{-1}
u\mathbf{1}^{\top }V^{-1} e + O(n^{-2}), \quad n\to \infty .
\end{aligned}
%
\end{equation}

The entry in the $i$-th row and $j$-th column of the inverse Vandermonde
matrix $V(x_{1},...,x_{n})^{-1}$ equals the coefficient of the power
$x^{i-1}$ in the Lagrange polynomial
$ P_{j}(x) = \displaystyle \prod _{k\ne j}
\frac{ x-x_{k} }{ x_{j}-x_{k} } $, see, e.g., \cite[\S 2.8.1]{NRbook}.
Hence $e^{\top }V^{-1} e$, being the last entry in the last row and column
of the inverse of $V := V(\zeta _{1},...,\zeta _{q+1}, 0)$, is the coefficient
of $x^{q+1}$ of the polynomial $P_{q+2}(x)$:
%
\begin{equation}
\label{eVe}
e^{\top }V^{-1} e =(V^{-1})_{q+2,q+2}= \prod _{k=1}^{q+1}
\frac{ 1 }{ -\zeta _{k} }.
\end{equation}
The leading asymptotic term in \eqref{aq}, therefore, equals
\begin{equation*}
\frac {\beta }{2} e^{\top }V^{-1}e = -\frac {1} {2}\sigma ^{2}_{0} s_{0}
\prod _{j=1}^{q} (-1/z_{j})
\frac{1}{ \prod _{j=1}^{q+1} (-\zeta _{j})}= \frac {1} {2}\sigma ^{2}_{0}
\prod _{j: |z_{j}|<1} \frac {1}{z_{j}^{2}},
\end{equation*}
where we used definition \eqref{zeta}.

The expression $\mathbf{1}^{\top }V^{-1} e$, being the sum over the last
column of $V^{-1}$, equals the sum of coefficients of $P_{q+2}(x)$, that
is,
%
\begin{equation}
\label{1Ve}
\mathbf{1}^{\top }V^{-1} e = P_{q+2}(1) = \prod _{k=1}^{q+1}
\frac{1-\zeta _{k}}{0-\zeta _{k}} = \prod _{k=1}^{q+1}
\frac{\zeta _{k}-1}{ \zeta _{k}}.
\end{equation}
Similarly, $e^{\top }V^{-1} u$ is the scalar product of the last row of
$V^{-1}$ with $u$ and hence
%
\begin{equation}
\label{eVu}
\begin{aligned}
e^{\top }V^{-1} u = & \sum _{k=1}^{q+1} \frac{1}{\zeta _{k}-1}(V^{-1})_{q+2,k}
- (V^{-1})_{q+2,q+2} =
\\
& \sum _{k=1}^{q+1} \frac{1}{\zeta _{k}-1}\frac {1} {\zeta _{k}}
\prod _{j\ne k}\frac {1} {\zeta _{k}-\zeta _{j}} - \prod _{k=1}^{q+1}
\frac {1}{ -\zeta _{k}} = - \prod _{j=1}^{q+1}
\frac {1} {1-\zeta _{j}},
\end{aligned}
%
\end{equation}
where the last equality is obtained by contour integration of the function
$ f(z)= \frac {1} {z(z-1)}\prod _{j=1}^{q+1} \frac {1} {z-\zeta _{j}}
$. Thus, in view of Theorem \ref{thm:const}, the second order asymptotic
term in \eqref{aq} is
\begin{equation*}
-\frac {1} {n} \frac {\beta }{2} \frac {\sin (\pi d) }{ \pi }
\lambda _{0} e^{\top }V^{-1} u\mathbf{1}^{\top }V^{-1} e =
\frac {1} {n} \frac {1} {2} \sigma ^{2}_{0} d (1+d) \prod _{j: |z_{j}|<1}
\frac {1} {z_{j}^{2}}
\end{equation*}
Combining all the expressions, we obtain the first expression in
\eqref{ablim}:
\begin{equation*}
a_{q+1,n} = \frac {1} {2}\sigma ^{2}_{0} \left (\prod _{j:|z_{j}|<1}
\frac{1}{ z_{j}^{2}} \right ) \Big( 1+ \frac {1} {n} d (1+d) \Big) + O(n^{-2}).
\end{equation*}
The second expression is obtained similarly:
\begin{equation*}
b_{q+1,n} = \frac {\beta }{2} e^{\top}\Big(V-
\frac {\sin (\pi d) }{ \pi n } \mu _{0}u\mathbf{1}^{\top}\Big)^{-1} e =
\frac {1} {2} \sigma ^{2}_{0} \left (\prod _{j:|z_{j}|<1}
\frac{1}{ z_{j}^{2}}\right ) \Big(1 - \frac {1} {n} d (1-d) \Big) + O(n^{-2}).
\end{equation*}
This completes the proof of Theorem \ref{thm3} for
$d\in (0,\tfrac {1} {2})$.

\subsection{The case $d\in (-\tfrac {1} {2},0)$}
\label{sec:case2again}

For $d\in (-\tfrac {1} {2},0)$, the function $Q(z)$ has no zeros. As explained
in Section \ref{sec:case2}, each of the systems \eqref{aeq} and
\eqref{beq} has $q+1$ equations and the same number of unknowns. If we
redefine $\beta := \sigma ^{2}_{0} \prod _{j=1}^{q} (-1/z_{j})$ and discard
$\zeta _{q+1}$ in \eqref{zeta}, we arrive at the asymptotic systems
\eqref{sysab}, with $q+1$ equations and unknowns each. The remaining calculations
are unchanged and lead to the same result.


\begin{appendix}

\section{Properties of $\mathbf{Q}(z)$}\label{app:AA}

\subsection{Polylogarithm}
To construct a sectionally holomorphic extension for the fGn density, the covariance sequence of the fGn 
$$
\gamma_0(k) = \tfrac 1 2 \big(|k+1|^{2d+1} -2|k|^{2d+1} + |k-1|^{2d+1}\big),
$$
suggests considering the series
\begin{equation}\label{muser}
\mu(z) :=\sum_{k=1}^\infty k^{2d+1} z^k.
\end{equation}
This series is convergent in the open unit disk $D$, where it defines a holomorphic function, 
known as the polylogarithm $\mathrm{Li}_{s}(z)$ with parameter $s:= -2d-1$. 
The function $\mu(z)$ admits a holomorphic extension to $\mathbb C \setminus [1,\infty)$, which 
can be constructed in a number of ways. One approach is by means of the Lindel\"of-Wirtinger expansion, see \cite[Sec. 5]{T45}, valid for $\Re(s)<0$ -
that is, for all $d\in (-\frac 1 2, \frac 1 2)$ in our case:
\begin{equation}\label{Li}
\mu(z) = \Gamma(2+2d) \sum_{k=-\infty}^\infty (-\log z +2\pi i k)^{-2-2d},\quad z\in \mathbb C\setminus [1,\infty),
\end{equation}
where the principal branch of the argument $\arg(-\log z +2\pi i k)\in (-\pi,\pi)$ is taken for each term.
%
%
In view of this formula, $\mu(z)$ has finite limits as it approaches the interval $(1,\infty)$ from the upper and lower half-planes. This renders $\mu(z)$  sectionally holomorphic in $\mathbb C\setminus [1,\infty)$.

Alternatively, the extension can also be constructed using the integral 
$$
j^{-a} = \frac 1 {\Gamma(a)} \int_0^\infty t^{a-1} e^{-jt}dt, \quad a>0.
$$ 
Plugging it into series \eqref{muser} yields 
\begin{align}
&
\label{mueq}
\mu(z)  =  \sum_{j=1}^\infty j^2 j^{-(1-2d)} z^j =  
\frac 1 {\Gamma(1-2d)} \int_0^\infty t^{-2d}
\bigg(\sum_{j=1}^\infty j^2   (ze^{-t})^j\bigg) dt =\\
&\nonumber
\frac 1 {\Gamma(1-2d)} \int_0^\infty t^{-2d}
 \frac {ze^{-t}(1+ze^{-t})}{(1-ze^{-t})^3} dt =  
 \frac 1 {\Gamma(1-2d)} \int_1^\infty (\log \tau)^{-2d}
 \frac {z (\tau+z )}{(\tau-z )^3} d\tau,
\end{align}
where we used the summation formula 
$$
\sum_{k=1}^\infty k^2 r^k = \frac {r(1+r)}{(1-r)^3}, \quad |r|<1.
$$
Clearly, the obtained integral representation remains well defined for all $z\not \in \mathbb C\setminus (1,\infty)$. 
Both constructions will be instrumental in further calculations.

\subsection{A formula for $\mathbf{Q^+(t)}$}

As argued above, $\mu(z)$ and $\mu(z^{-1})$ are sectionally holomorphic in $\mathbb C \setminus [1,\infty)$ and 
$\mathbb C \setminus [0,1]$, respectively, and hence, cf.  \eqref{Qdef},
\begin{equation}\label{Qdefmu}
Q(z) =  \frac{1}{4\pi} (z^{-1} -2 + z) \big(\mu(z)+\mu(z^{-1})), 
\end{equation}
is sectionally holomorphic in $\mathbb C\setminus \Real_+$ (with a singularity at $z=1$). The next lemma provides a useful formula for its limit.

\begin{lem}\label{lem:3.1}
For $d\in (-\frac 1 2, \frac 1 2)\setminus \{0\}$,  
\begin{equation}\label{Qplus}
Q^+(t) = \frac{1}{8\pi}     \frac {(1-t)^2}{\sin ( \pi d)}  
\Big(
\pi   A^+(t) e^{\pi d i}+2\Re\big(ie^{\pi d i}  B(t)\big)  
\Big),\qquad t\in (1,\infty),
\end{equation}
where  
\begin{align}
\label{Aplus}
A^+(t) = & \frac {4d ( 2d+1)}{\Gamma(1-2d)} \frac {(\log t)^{-2d-2}} t,
\\
\label{Bplus}
B(t)  = &
-\frac {2d}{\Gamma(1-2d)}\frac {(\log t)^{-2d-1}}  t  \int_0^\infty  \left(1+\dfrac{\log   u +\pi i}{\log t}\right)^{-2d-1}\frac {1} {(u  + 1)^2}du.
\end{align}
\end{lem}

\begin{proof}
Define the functions, cf. \eqref{mueq},  
\begin{align*}
\widetilde \mu(z) :=  &\, z^{-1} \mu(z) =
 \frac 1 {\Gamma(1-2d)} \int_1^\infty (\log \tau)^{-2d}
 \frac {   \tau+z  }{(\tau-z )^3} d\tau, 
\\
\widetilde \nu(z) := & -z^{-1} \mu(z^{-1}) = 
   \frac 1 {\Gamma(1-2d)} \int_0^1 (\log \tau^{-1})^{-2d}
 \frac {  (\tau+z )}{(\tau -z )^3}  d\tau.
\end{align*}
Then, by definition \eqref{Qdefmu},
\begin{equation}\label{QAB}
Q(z) =   \frac{1}{4\pi} (1-z)^2  \big( \widetilde \mu(z)- \widetilde \nu(z)), \quad z\in \mathbb C\setminus \Real_+.
\end{equation}

Let $C_+$ be the semi-circular contour in the upper half plane which excludes the origin.
Define 
$$
f(\zeta) = \frac{(\log \zeta)^{-2d} (\zeta+z)}{(\zeta-z)^3}, 
$$
where $\arg(\zeta)\in (-\pi,\pi]$. Then for $z$ inside $C_+$,  
\begin{equation}\label{Res}
\oint_{C_+} \frac{(\log \zeta)^{-2d} (\zeta+z)}{(\zeta-z)^3}d\zeta =2\pi i\, \mathrm{Res}(f;z),
\end{equation}
with the residue 
$$
\mathrm{Res}(f;z) = \frac 1 {2!} \lim_{\zeta\to z}\frac{d^2}{d\zeta^2}\Big( (\log \zeta)^{-2d} (\zeta+z)\Big) =
2d ( 2d+1)\frac 1 z(\log z)^{-2d-2}.
$$

As the radius of the contour tends to infinity and its base approaches the real line, the integral in \eqref{Res} converges:  
\begin{align*}
&
\oint_{C_+} \frac{(\log \zeta)^{-2d} (\zeta+z)}{(\zeta-z)^3}d\zeta \to  
\int_{-\infty}^0 \frac{(\log (-t)+\pi i)^{-2d} (t+z)}{(t-z)^3}dt
+ \\
&
e^{-2\pi d i}\int_{0}^1 \frac{(\log t^{-1})^{-2d} (t+z)}{(t-z)^3}dt
+ 
\int_{1}^\infty \frac{(\log t)^{-2d} (t+z)}{(t-z)^3}dt =\\
&
\int_{0}^\infty \frac{(\log  t +\pi i)^{-2d} (t-z)}{(t+z)^3}dt
+ 
\Gamma(1-2d)\Big(
e^{-2\pi d i}\widetilde \nu(z)
+ 
 \widetilde \mu(z)\Big).
\end{align*}
Similarly, integrating over the semi-circular contour in the lower half plane we get
\begin{align*}
&
0 =  -\oint_{C_-} \frac{(\log \zeta)^{-2d} (\zeta+z)}{(\zeta-z)^3}d\zeta \to 
\int_{-\infty}^0 \frac{(\log (-t)-\pi i)^{-2d} (t+z)}{(t-z)^3}dt
+ \\
&
e^{2\pi di}\int_0^1 \frac{(\log t^{-1})^{-2d} (t+z)}{(t-z)^3}dt
+
\int_1^\infty \frac{(\log t)^{-2d} (t+z)}{(t-z)^3}dt =\\
&
\int_{0}^\infty \frac{(\log   t -\pi i)^{-2d} (t-z)}{(t+z)^3}dt+
\Gamma(1-2d)\Big(
e^{2\pi d i}\widetilde \nu(z)
+ 
 \widetilde \mu(z)\Big).
\end{align*}
Define the functions
\begin{align*}
A(z) := & \frac {1}{\Gamma(1-2d)} \frac {4d ( 2d+1)} z(\log z)^{-2d-2}, \\
B(z) := & \frac 1{\Gamma(1-2d)} \int_{0}^\infty \frac{(\log   \tau +\pi i)^{-2d} (\tau-z)}{(\tau+z)^3}d\tau.
\end{align*}
The obtained equations can be then combined in the linear system
$$
\begin{pmatrix*}
1 & e^{-2\pi d i} \\
1 & e^{\phantom +2\pi d i} 
\end{pmatrix*}
\begin{pmatrix*}
\widetilde \mu(z) 
\\
\widetilde \nu(z)
\end{pmatrix*} =
\begin{pmatrix*}
\pi i A(z)-B(z)
\\
-\overline { B(\overline z)}
\end{pmatrix*}.
$$
Solving this system we obtain    
\begin{align*}
\widetilde \mu(z) = & \frac 1{2i} \frac 1{\sin (2\pi d)} \Big( e^{2\pi di} (\pi i A(z)-B(z)) +e^{-2\pi di} \overline { B(\overline z)}\Big),
\\ 
\widetilde \nu(z) = &
\frac 1{2i} \frac 1{\sin (2\pi d)}
\Big(
-(\pi i A(z)-B(z))-\overline { B(\overline z)}
\Big).
\end{align*}
Substitution into \eqref{QAB} yields the formula 
$$
Q(z) =    \frac{1}{4\pi} (1-z)^2  \frac 1{2i} \frac 1{\sin (2\pi d)}  
\Big(
\pi i A(z) (e^{2\pi di} + 1)- (e^{2\pi di} +1) B(z) +(e^{-2\pi di} +1)\overline{B(\overline z)}
\Big),
$$
which holds in the upper half-plane. 
Taking the limit $z\to t^+$ we arrive at
\begin{align*}
Q^+(t) = &  \frac{1}{4\pi} (1-t)^2  \frac 1{2i} \frac 1{\sin (2\pi d)} 
\Big(
\pi i A^+(t) (e^{2\pi di} + 1)- (e^{2\pi di} +1) B(t) +(e^{-2\pi di} +1) \overline { B(t)}
\Big) =\\
&
  \frac{1}{8\pi}     \frac {(1-t)^2}{\sin ( \pi d)}  
\Big(
\pi   A^+(t) e^{\pi d i}+2\Re(ie^{\pi d i}  B(t))  
\Big).
\end{align*}
The expression for $B(t)$ in \eqref{Bplus} is obtained by integrating by parts. 
The limit $A^+(t)$, for $t\in (1,\infty)$, equals $A(t)$  and is thus given by \eqref{Aplus}.
\end{proof}

\subsection{Proof of Theorem \ref{lem:Q}
}\label{ssec:Qprop}
We are now prepared to derive the properties of $Q(\cdot)$ from \eqref{Qdefmu} as stated in Theorem \ref{lem:Q}.

\medskip

(i) The identities in (3.1): 
$$
Q(z)=Q(z^{-1}), \quad
 \overline{Q(z)}=Q(\overline z),
\qquad z\in \mathbb C\setminus \Real_+.
$$
follow from definition \eqref{Qdefmu} and representation \eqref{mueq}. 
The properties in \eqref{Qpmsym} are derived from these identities by direct calculation.
Equality  \eqref{ImQ} holds due to \eqref{Qplus}-\eqref{Aplus} and \eqref{Qpmsym}.

\medskip

(ii)
Equality \eqref{f0Q} is verified by evaluating  expression \eqref{Qdefmu} on the unit circle, using formula \eqref{Li}.

\medskip

(iii) The series representation in \eqref{Li} implies
\begin{equation}\label{mu1}
\big|(z-1)^{2+2d} \mu(z)\big| \xrightarrow[z\to 1]{}\Gamma(2+2d),
\end{equation}
and, consequently, the estimate as $z\to 1$ in (3.4).
 From Jonqui\'ere's formula \cite[eq. 7.190]{L81},
which relates polylogarithm to the Hurwitz zeta function and its asymptotic expansion  \cite[\S 1.4, p 25]{MOFS66},
it follows that
\begin{equation}\label{muinf}
\mu(z) = - \frac 1{\Gamma(-2d)} (\log z)^{-1-2d} \big(1+o(1)\big), \quad z\to \infty.
\end{equation}
This implies the estimate as $z\to\infty$ in (3.4) and, since $Q(z)=Q(z^{-1})$, also as $z\to  0$.

\medskip
(iv) 
Since $Q(z)$ shares its zeros in $(-1,0)$ with the function $r(z):=\mu(z)+\mu(z^{-1})$, it suffices to show that $r(s)$ changes its
sign as $s$ varies through $(-1,0)$. In view of \eqref{mueq}, $r(s)$ is continuous at $s=-1$ and, for any $d\in (-\tfrac 1 2,\tfrac 1 2)$, 
$$
r(-1)=2\mu(-1) = -\frac 1 {\Gamma(1-2d)} \int_1^\infty (\log \tau)^{-2d}
 \frac {  (\tau-1 )}{(\tau+1 )^3} d\tau < 0,    
$$
where the inequality holds since $\Gamma(1-2d)>0$. 
On the other hand, \eqref{muser} implies that $\mu(s)=O(s)$ as $s\to 0$ and hence, based on \eqref{muinf},
$$
r(s) = \mu(s^{-1})+O(s) =
- \frac 1{\Gamma(-2d)} (\log |s|^{-1})^{-1-2d} \big(1+o(1)\big), \quad s\to 0.
$$
Since $\Gamma(-2d)<0$ for $d\in (0,\frac 1 2)$,  $r(s)$ is positive in a vicinity of $s=0$ and hence it must have at least one zero in $(-1,0)$.

\medskip
(v) Expression \eqref{Qplus} for $Q^+(t)$ defines a plane curve, which is smooth for $t\in (1,\infty)$ and, due to \eqref{ImQ}, does not pass through the origin. 
Thus, the function $\eta(t)=\arg(Q^+(t))$, being the angle drawn by this curve relative to the semi-axis $\Real_+$, has derivatives of all orders.  
Property \eqref{etaeta} 
follows from \eqref{Qpmsym}. Based on \eqref{Li},   
$$
\begin{aligned}
& 
\mu^+(t) =
\Gamma(2+2d)    \big(\log t^{-1}\big)^{-2-2d} + O(1),
\\
& 
\mu^{-}(t^{-1}) =
\Gamma(2+2d)\big(\log t^{-1}\big)^{-2-2d}e^{-2\pi di} + O(1),
\end{aligned} \qquad t\nearrow  1.
$$
Substitution into definition \eqref{Qdefmu} shows that  
$$
Q^+(t) = \frac {\Gamma(2+2d)}{4\pi}\frac{(t-1)^2}{t} \big(\log t^{-1}\big)^{-2-2d} \big(1+e^{-2\pi di} + O ((1-t)^{ 2+2d})\big), \quad t\nearrow 1,
$$
which implies \eqref{etalim1}.
The estimate \eqref{etalim0} is derived using Lemma \ref{lem:3.1}. The expression in \eqref{Bplus} satisfies 
$$
B(t)  = 
-\frac {2d}{\Gamma(1-2d)}\frac {(\log t)^{-2d-1}}  t \Big(1+ c_1(\log t)^{-1}  + O((\log t)^{-2})\Big), \qquad t\to \infty,
$$
for some constant $c_1\in \mathbb C$. Plugging this and \eqref{Aplus} into \eqref{Qplus} gives
$$
 Q^+(t) =   \frac{1}{2\pi}       \frac {  d}{\Gamma(1-2d)}\frac { (1-t)^2  (\log t)^{-2d-1}}  t  \Big(1+  \frac {c_2}{\log t} + O((\log t)^{-2}) \Big), \quad t\to\infty,
$$
for some constant $c_2\in \mathbb C$. In view of \eqref{ImQ}, it follows that   
$$
\eta(t) =   -\pi \one{d<0} + \frac c{\log t}  +O((\log t)^{-2}), \quad t\to\infty,
$$
for some $c\in \Real$. 
The estimate \eqref{etalim0} can now be obtained from \eqref{etaeta}.
\qed

\subsection{Proof of Lemma \ref{lem:X0}
}\label{ssec:X0}
The estimate as $z\to\infty$ in \eqref{X0est} holds because $\eta(\cdot)$ is bounded on $[0,1]$. To derive the estimate as $z\to 1$, let us write 
$$
X_0(z)    =\exp\left( - \int_0^1 \frac{ d }{\tau-z}d\tau
+ 
\frac 1{ \pi  }\int_0^1 \frac{\eta(\tau)+\pi d}{\tau-1}d\tau +
\frac {z-1}{ \pi  }\int_0^1 \frac{\eta(\tau)+\pi d}{(\tau-z)(\tau-1)}d\tau
\right).
$$
Here, the integrals are finite due to \eqref{etalim1}, and, moreover, the last integral converges to a finite limit as $z\to 1$. 
This implies the claimed estimate: 
$$
X_0(z) = c_2\left(\frac{z-1} z\right)^{-d} (1+O(z-1))=
c_2(z-1)^{-d} \big(1+O(z-1)\big), \quad z\to 1,
$$
with $c_2 = \exp\left( \displaystyle \frac 1{ \pi  }\int_0^1 \frac{\eta(\tau)+\pi d}{\tau-1}d\tau \right)$.

To obtain the estimate as $z\to 0$, consider the integral  
\begin{align*}
  \int_0^{ 1/2} \frac { 1}{\log \tau^{-1}}\frac{ 1 }{\tau-z}d\tau   = &
 -  z \int_0^{1/2}\frac{\log \log \tau^{-1}}{(\tau-z)^2}d\tau + C_1 + O(z) = \\
 &
  -   z \frac d {dz}\int_0^{1/2}\frac{\log \log \tau }{ \tau-z  }d\tau + C_2 + O(z), 
 \quad  z\to 0, 
\end{align*}
where the first equality is obtained by integration by parts and in the second, the branch of $\log z$ with $\arg(z)\in [0,2\pi)$ is used.
As shown in \cite{M59},
$$
 \int_0^{1/2}\frac{\log \log \tau }{ \tau-z  }d\tau =   \log z \big(1-\log\log z\big)+\pi i \log \log z + \Phi(z)
$$
where $\Phi(z)$ is a function analytic in the vicinity of 0. Taking the derivative, we get  
$$
\frac d{dz}  \int_0^{1/2}\frac{\log \log \tau }{ \tau-z  }d\tau = 
 -\frac 1 z\log\log z  
 +\pi i \frac 1 {\log z}\frac 1 z + \Phi'(z)
$$ 
and, consequently,
$$
 \int_0^{ 1/2} \frac { 1}{\log \tau^{-1}}\frac{ 1 }{\tau-z}d\tau=
     \log\log z  + O\big((\log z)^{-1}\big), \quad z\to 0.
$$
Due to this estimate and \eqref{etalim0},  
\begin{align*}
X_0(z) = & \exp\left(\frac 1{ \pi  }\int_0^1 \frac{\pi \one{d<0}}{\tau-z}d\tau
+
\frac 1{ \pi  }\int_0^{1/2} \frac{\eta(\tau)-\pi \one{d<0}}{\tau-z}d\tau
+ C_3 +O(z)
\right) =\\
&
\left(\frac {z-1}z\right)^{\one{d<0}}
\exp\left( \frac c{ \pi  }\int_0^{1/2}\frac 1 {\log \tau^{-1}} \frac{1}{\tau-z}d\tau
+ C_4 + O(z)
\right) =\\
&
z^{-\one{d<0}}\exp\Big(
\frac c{ \pi  }
\log\log z  + C_4 + O\big((\log z)^{-1}\big)
\Big),\quad z\to 0,
\end{align*}
which verifies the estimate in \eqref{X0est} as $z\to 0$ with $a:=c/\pi$.
\qed

\section{Proof of Lemma \ref{lem:identity}
}\label{app:A}

We will prove  formula \eqref{XQfla}  
for $d\in (0,\frac 1 2)$, omitting the similar proof of \eqref{XQfla2} 
for the case $d\in (-\frac 1 2,0)$. 
Let us first verify it under the assumption that the zero $s_0\in (-1,0)$ from Theorem  \ref{lem:Q} 
(iv) 
is simple and it is the only zero of $Q(\cdot)$
inside the unit disk. We will argue later that this is indeed the case. 
To this end, consider the simply connected region  $\Omega=\overline D\setminus [s_0,1]$ depicted in Figure \ref{fig1}, in which 
$Q(\cdot)$ is holomorphic and non-vanishing. 
\begin{figure}[h]
\begin{tikzpicture}[scale = 0.8]
\draw[help lines,->] (-3,0) -- (3,0) coordinate (xaxis);
\draw[help lines,->] (0,-3) -- (0,3) coordinate (yaxis);

\fill[black] (-1.5, 0) circle (1pt);
\node[below] at (-1.7, 0) {\small $s_0$};
\draw[line width = 0.8pt] (0, 0) circle (2.5);
\draw[line width = 0.8pt] (-1.5,0) -- (2.5, 0);
\fill[black] (0, 0) circle (1pt);
\node[below] at (-0.23, 0) {\small $0$};
\node[below] at (2.65, 0) {\small $1$};

\end{tikzpicture}
 
\caption{\label{fig1} Simply connected region $\Omega$ }
\end{figure} 

It follows from \eqref{ImQ} that $\Im(Q^+(\frac 12 ))<0$ and, since $Q(z)$ is continuous in $\Omega$, we can find a point 
$z_0$ with $\Re(z_0)=\frac 1 2$ and sufficiently small $\Im(z_0)>0$ such that 
$$
\Im(Q(\tfrac 1 2+iy ))<0, \quad \forall\ 0\le y\le \Im(z_0).
$$
For $z\in \Omega$, define  
\begin{equation}\label{Lz}
L(z) := \log Q(z_0) + \int_{z_0}^z\frac{Q'(\zeta)}{Q(\zeta)}d\zeta,
\end{equation}
where $\log(z) = \log |z| + i \arg(z)$ with $\arg(z)\in (-\pi,\pi]$ and the integration is carried out on an arbitrary simple curve which starts at $z_0$ and ends at $z$. This definition is independent of the curve's choice and the function $L(\cdot)$ is holomorphic in $\Omega$ satisfying, see e.g., \cite[Ch 3,\S 6]{Sbook03},
\begin{equation}\label{ExpL}
\exp(L(z))=Q(z), \quad z\in \Omega.
\end{equation}
The next lemma formulates some of its relevant properties.

\begin{lem}\label{lem:id}
\

\medskip

\begin{enumerate}

\item  For $t\in (0,1)$,
\begin{equation}\label{id1}
L^\pm (t) =  \log Q^\pm(t).
\end{equation}
\item For  $\lambda \ne 0$,
\begin{equation}\label{id2}
L(e^{i\lambda}) = \log f_0(\lambda).
\end{equation}

\item For $t\in (s_0,0)$,
\begin{equation}\label{id3}
L^+(t)-L^-(t)  = -2\pi i.
\end{equation}
\end{enumerate}
 
\end{lem}

\begin{proof}
\

\medskip

(1). Fix a point $t\in (0,1)$ and a small constant $\eps>0$. Integrating in \eqref{Lz} on the curve depicted in red in Figure \ref{fig2} 
 gives
\begin{equation}\label{L1}
L(t+i\eps ) = \log Q(z_0) 
+  \int_{z_0}^{\frac 1 2+i\eps}\frac{Q'(\zeta)}{Q(\zeta)}d\zeta
+  \int_{\frac 1 2+i\eps}^{t+i\eps}\frac{Q'(\zeta)}{Q(\zeta)}d\zeta.
\end{equation}
\begin{figure}[t]
\begin{tikzpicture}[scale = 0.8]
\draw[help lines,->] (-3,0) -- (3,0) coordinate (xaxis);
\draw[help lines,->] (0,-3) -- (0,3) coordinate (yaxis);

\fill[black] (-1.5, 0) circle (1pt);
\node[below] at (-1.7, 0) {\small $s_0$};
\draw[line width = 0.8pt] (0, 0) circle (2.5);
\draw[line width = 0.8pt] (-1.5,0) -- (2.5, 0);
\fill[black] (0, 0) circle (1pt);
\node[below] at (-0.23, 0) {\small $0$};
\node[below] at (2.65, 0) {\small $1$};

\fill[black] (1.25, 1) circle (1pt);
\node[right] at (1.25, 1) {\small $z_0$};
\fill[black] (1.25, 0) circle (1pt);
\node[below] at (1.25, 0) {\small $\frac 1 2$};
\fill[black] (2.0, 0) circle (1pt);
\node[below] at (2.0, 0) {\small $t$};

\draw[line width=0.8pt, color=red,
decoration={markings,
mark=at position 0.4cm with {\arrow[line width=0.5pt]{>}},
mark=at position 1.1cm with {\arrow[line width=0.5pt]{>}}
},
postaction=decorate
] (1.25,1.0) -- (1.25,0.1) -- (2,0.1);

\end{tikzpicture}
\caption{\label{fig2}}
\end{figure}
Let $U(z):=\Re(Q(z))$ and $V(z):=\Im(Q(z))$. Then   
\begin{align*}
&
 \int_{z_0}^{\frac 1 2+i\eps}\frac{Q'(\zeta)}{Q(\zeta)}d\zeta = 
\int_{z_0}^{\frac 1 2+i\eps}\frac{U'(\zeta)+iV'(\zeta)}{U(\zeta)+iV(\zeta)}d\zeta =\\
&
\int_{z_0}^{\frac 1 2+i\eps}\frac{ U'(\zeta)U(\zeta)+ V'(\zeta)V(\zeta) }{U(\zeta)^2+V(\zeta)^2}d\zeta
+
i\int_{z_0}^{\frac 1 2+i\eps}\frac{ U(\zeta)V'(\zeta) -U'(\zeta) V(\zeta)   }{U(\zeta)^2+V(\zeta)^2}d\zeta =\\
&
\int_{z_0}^{\frac 1 2+i\eps}\frac{d}{d\zeta}\log |Q(\zeta)| d\zeta
+
i\int_{z_0}^{\frac 1 2+i\eps}\frac{d}{d\zeta}\arccot \frac{U(\zeta)}{V(\zeta)}d\zeta =
\log Q(\tfrac 1 2+i\eps)-\log Q(z_0).
\end{align*}
The last equality is true due to the identity 
$$
\arg (Q(z))=-\pi + \arccot (U(z)/V(z))\in (-\pi, 0], 
$$
which holds since $\Im(Q(\zeta))<0$ on the integration curve.
Similarly, 
$$
\int_{\frac 1 2+i\eps}^{t+i\eps}\frac{Q'(\zeta)}{Q(\zeta)}d\zeta = \log Q(t+i\eps)-\log Q(\tfrac 12+i\eps).
$$
The identity \eqref{id1} for $L^+(t)$ is obtained by substituting these expressions into \eqref{L1} and taking the limit $\eps\to 0$.
The identity for $L^-(t)$ is verified while proving (2), see below.
\medskip

\begin{figure}[t]
\begin{tikzpicture}[scale = 0.8]
\draw[help lines,->] (-3,0) -- (3,0) coordinate (xaxis);
\draw[help lines,->] (0,-3) -- (0,3) coordinate (yaxis);

\fill[black] (-1.5, 0) circle (1pt);
\node[below] at (-1.7, 0) {\small $s_0$};
\draw[line width = 0.8pt] (0, 0) circle (2.5);
\draw[line width = 0.8pt] (-1.5,0) -- (2.5, 0);
\fill[black] (0, 0) circle (1pt);
\node[below] at (-0.23, 0) {\small $0$};
\node[below] at (2.65, 0) {\small $1$};



\draw[line width=0.8pt, color=red] (2.5,0.3) arc (90:180:0.3);

\draw[line width = 0.8pt, color=red,
decoration={markings,
mark=at position 1cm with {\arrow[line width=0.5pt]{>}}
},
postaction=decorate
] (2.5, 0.0) arc (0:50:2.5);

\draw[line width = 0.8pt] (2.5, 0.0) arc (0:6.7:2.5);

\fill[black] (1.63, 1.9) circle (1pt);
\node[above] at (2.3, 1.9) {\small $z=e^{i\lambda}$};

\end{tikzpicture}
\caption{\label{fig3}}
\end{figure}

(2). Fix a point on the unit circle $z=e^{i\lambda}$ with $\lambda\ne 0$ and a small constant $\eps>0$. Then by integrating along the curve depicted in red in Figure \ref{fig3} we get
\begin{equation}\label{Le}
L(e^{i\lambda}) = L^+(1-\eps) + I(\eps) +  J(\eps),
\end{equation}
where  $I(\eps)$ is the integral over the arc of radius $\eps$ around $z=1$:
$$
I(\eps) := \int_{\pi}^{\pi/2+\mathrm{asin}(\eps/2)} \frac{Q'(1+\eps e^{i s})}{Q(1+\eps e^{i s})}d (\eps e^{i s})
$$
and  $J(\eps)$ is the integral over the arc on the unit circle:
\begin{equation}\label{Jeps}
J(\eps) := \int_{\mathrm{acos}(1-\eps^2/2)}^{\lambda}\frac{Q'(e^{i\alpha })}{Q(e^{i\alpha})}d(e^{i\alpha}) = \log f_0(\lambda) - \log f_0\big(\mathrm{acos}(1-\tfrac 1 2\eps^2)\big),
\end{equation}
where we used \eqref{f0Q}.
By representation \eqref{Li} and definition \eqref{Qdefmu},
$$
Q(z)=   c \frac{(z-1)^2}{z} 
\big(
  (\log z)^{-2-2d} + \phi(z)
\big),
$$
where $c$ is some constant and $\phi(z)$ is a function, holomorphic in a vicinity of $z=1$. 
Consequently, 
$$
\frac{Q'(\zeta)}{Q(\zeta)}_{\big|\zeta=1+\eps e^{is}} = -2d (\eps e^{is})^{-1} (1+o(1)), \quad \eps \to 0,
$$
and, therefore, 
\begin{equation}\label{Ieps}
I(\eps)\xrightarrow[\eps\to 0]{} \pi d i.
\end{equation}

Let us now estimate the first term in the right hand side of \eqref{Le}. To this end,  by \eqref{Qplus}-\eqref{Bplus},
$$
|Q^+(t)| =    
 \frac{1}{8 }     \frac {(1-t)^2}{\sin ( \pi d)}  
 \left|
    A^+(t)  + O(1)
\right|  =  
\frac{\Gamma (2d+2)  \cos(\pi d) }{2\pi } \frac { (1-t)^2  }t\Big( (\log t)^{-2d-2} + O(1)\Big), \quad t\searrow 1.
$$
\begin{figure}[t]
\begin{tikzpicture}[scale = 0.8]
\draw[help lines,->] (-3,0) -- (3,0) coordinate (xaxis);
\draw[help lines,->] (0,-3) -- (0,3) coordinate (yaxis);

\fill[black] (-1.5, 0) circle (1pt);
\node[below] at (-1.7, 0) {\small $s_0$};
\draw[line width = 0.8pt] (0, 0) circle (2.5);
\draw[line width = 0.8pt] (-1.5,0) -- (2.5, 0);
\fill[black] (0, 0) circle (1pt);
\node[below] at (-0.23, 0) {\small $0$};
\node[below] at (2.65, 0) {\small $1$};

\draw[line width = 0.8pt, color=red,
decoration={markings,
mark=at position 3cm with {\arrow[line width=0.5pt]{>}},
mark=at position 6cm with {\arrow[line width=0.5pt]{>}},
mark=at position 9cm with {\arrow[line width=0.5pt]{>}},
mark=at position 12cm with {\arrow[line width=0.5pt]{>}}
},
postaction=decorate
] (2.5, 0.0) arc (0:351:2.5);

\draw[line width = 0.8pt] (2.5, 0.0) arc (0:50:2.5);

\fill[black] (1.63, 1.9) circle (1pt);
\node[above] at (2.3, 1.9) {\small $z=e^{i\lambda}$};

\draw[line width = 0.8pt, color=red] (2.47, -0.4) arc (-90:-180:0.3) -- (1.5,-0.1);
\fill[black] (1.5, 0) circle (1pt);
\node[above] at (1.5, 0) {\small $t$};

\end{tikzpicture}
\caption{\label{fig4}}
\end{figure}
Thus, in view of \eqref{Qpmsym}, 
$$
|Q^+(1-\eps)| = \Big|Q^+\Big(\frac 1 {1-\eps}\Big)\Big| =
\frac{\Gamma (2d+2)  \cos(\pi d) }{2\pi }       \eps^{ -2d}(1 + O(\eps)), \quad \eps\to 0,
$$
and, due to \eqref{fGnf} 
it follows that
\begin{equation}\label{Qf0}
 |Q^+(1-\eps)|/ f_0(\eps) \to 1, \quad \eps\to 0.
\end{equation}
Substitution of \eqref{id1}, \eqref{Jeps}, \eqref{Ieps} and \eqref{Qf0} into \eqref{Le} yields the identity in \eqref{id2}:
\begin{align*}
 &
L(e^{i\lambda}) =\  L^+(1-\eps) + I(\eps) +  J(\eps) = \\
&  \log |Q^+(1-\eps)| + i \arg(Q^+(1-\eps)) +   
   I(\eps)+ \log f_0(\lambda) -  \log f_0\big(\mathrm{acos}(1-\tfrac 1 2\eps^2)\big)
 \xrightarrow[\eps\to 0]{} \log f_0(\lambda),
\end{align*}
where we used asymptotics \eqref{etalim1}.
The identity for $L^-(t)$ in \eqref{id1} is verified by similar calculations, if we integrate over the curve depicted in red in Figure \ref{fig4}.

\begin{figure}[b]
\begin{tikzpicture}[scale = 0.8]
\draw[help lines,->] (-3,0) -- (3,0) coordinate (xaxis);
\draw[help lines,->] (0,-3) -- (0,3) coordinate (yaxis);

\fill[black] (-1.5, 0) circle (1pt);
\node[below] at (-1.7, 0) {\small $s_0$};
\draw[line width = 0.8pt] (0, 0) circle (2.5);
\draw[line width = 0.8pt] (-1.5,0) -- (2.5, 0);
\fill[black] (0, 0) circle (1pt);
\node[below] at (-0.23, 0) {\small $0$};
\node[below] at (2.65, 0) {\small $1$};

\fill[black] (1.25, 1) circle (1pt);
\node[right] at (1.25, 1) {\small $z_0$};
 
\draw[line width=0.8pt, color=red,
decoration={markings,
mark=at position 1.5cm with {\arrow[line width=0.5pt]{<}}
},
postaction=decorate
]
(1.25, 1) .. controls (-1,0.5) and (-1,-1.5) .. (-2.0,-0.5); 

\draw[line width=0.8pt, color=red,
decoration={markings,
mark=at position 1.5cm with {\arrow[line width=0.5pt]{<}}
},
postaction=decorate
]
(-2.0,-0.5) .. controls (-2,1.5) and (1,1.5) .. (1.25,1); 

\fill[black] (-0.54, 0) circle (1pt);
\node[above] at (-0.54, 0) {\small $t$}; 

\node[left] at (0, 1.5) {\small $\Gamma$}; 

\end{tikzpicture}
\caption{\label{fig5}}
\end{figure} 
 
\medskip 
 
(3) Fix a point $t\in (s_0,0)$ and choose any contour $\Gamma$ in $D\setminus [0,1]$ which passes through $z_0$ and crosses the interval $(s_0,0)$
at the point $t$, see Figure \ref{fig5}. By the argument principle,  
\begin{equation}\label{ResQ}
\oint_\Gamma \frac{Q'(\zeta)}{Q(\zeta)}d\zeta = 2\pi i.
\end{equation}
Let $\Gamma_1$ be the subcurve  which starts at $z_0$ and ends at $t$ and $\Gamma_2=\Gamma \setminus \Gamma_1$ the subcurve which starts at $t$ and ends  at $z_0$. Then 
by definition \eqref{Lz},
\begin{align*}
\oint_\Gamma \frac{Q'(\zeta)}{Q(\zeta)}d\zeta =  &  \int_{\Gamma_1}  \frac{Q'(\zeta)}{Q(\zeta)}d\zeta 
+
\int_{\Gamma_2}\frac{Q'(\zeta)}{Q(\zeta)}d\zeta =
L^-(t) - L^+(t).
\end{align*}
Combining this with \eqref{ResQ} yields \eqref{id3}.

\end{proof}

\begin{figure}[b]
\begin{tikzpicture}[scale = 0.8]
\draw[help lines,->] (-3,0) -- (3,0) coordinate (xaxis);
\draw[help lines,->] (0,-3) -- (0,3) coordinate (yaxis);

\fill[black] (-1.5, 0) circle (1pt);
\node[below] at (-1.65, 0) {\small $s_0$};
\draw[line width=0.8pt, color=blue,
decoration={markings,
mark=at position 5cm with {\arrow[line width=0.5pt]{>}},
mark=at position 11cm with {\arrow[line width=0.5pt]{>}},
mark=at position 14cm with {\arrow[line width=0.5pt]{>}}
},
postaction=decorate
] (2.5,0.08) arc (2:358:2.5)  -- (-1.4, -0.1);

\draw[line width=0.8pt, color=blue] (-1.38, 0.08) arc (35:325:0.15);

\draw[line width=0.8pt, color=blue,
decoration={markings,
mark=at position 1cm with {\arrow[line width=0.5pt]{>}}
},
postaction=decorate
] (-1.4, 0.08) -- (2.5, 0.08); 

\node[left] at (-1.9,1.9) {\small $C$};

\end{tikzpicture}
\caption{\label{fig6}}
\end{figure} 

We are now prepared to prove \eqref{XQfla}. Fix a point $z$ inside the contour depicted in Figure \ref{fig6}.
Since $L(\cdot)$ is holomorphic inside this contour, by the Cauchy theorem,
$$
\oint_C \frac{L(\zeta)}{\zeta-z}d\zeta = 2\pi i L(z).
$$
On the other hand, by integrating separately on different parts of the contour and taking the limit towards the boundary of $\Omega$ in Figure \ref{fig1} we obtain
\begin{align*}
\oint_C \frac{L(\zeta)}{\zeta-z}d\zeta = & \int_0^{2\pi} \frac{L(e^{i\lambda})}{e^{i\lambda}-z}de^{i\lambda} 
 + \int_{s_0}^0 \frac{L^+(t)-L^-(t)}{t-z}dt + \int_0^1 \frac{L^+(t)-L^-(t)}{t-z}dt =\\
&
\int_0^{2\pi} \frac{\log f_0(\lambda)}{e^{i\lambda}-z}de^{i\lambda} 
 - \int_{s_0}^0 \frac{2\pi i}{t-z}dt + \int_0^1 \frac{\log Q^+(t)-\log Q^-(t)}{t-z}dt =\\
& 
\int_0^{2\pi} \frac{\log f_0(\lambda)}{e^{i\lambda}-z}de^{i\lambda} 
 - 2\pi i\log \frac{z}{z-s_0} + 2i\int_0^1 \frac{ \arg(Q^+(t))}{t-z}dt,
\end{align*}
where we used the identities from Lemma \ref{lem:id} and the second symmetry in \eqref{Qpmsym}. Equating the two expressions we arrive at 
$$
L(z) = 
\frac {1}{2\pi  i}\int_0^{2\pi} \frac{\log f_0(\lambda)}{e^{i\lambda}-z}de^{i\lambda} 
 -  \log \frac{z}{z-s_0} + \frac 1 \pi\int_0^1 \frac{ \arg(Q^+(t))}{t-z}dt.
$$
If we now compute the exponent of both sides and take into account the definitions  \eqref{X0}
and  \eqref{psizdef} and property \eqref{ExpL}, we obtain \eqref{XQfla}.

\begin{figure}[b]
\begin{tikzpicture}[scale = 0.8]
\draw[help lines,->] (-3,0) -- (3,0) coordinate (xaxis);
\draw[help lines,->] (0,-3) -- (0,3) coordinate (yaxis);

\fill[black] (-1.5, 0) circle (1pt);
\node[below] at (-1.7, 0) {\small $s_0$};
\draw[line width = 0.8pt] (0, 0) circle (2.5);
\draw[line width = 0.8pt] (-1.5,0) -- (2.5, 0);
\fill[black] (0, 0) circle (1pt);
\node[below] at (-0.23, 0) {\small $0$};
\node[below] at (2.65, 0) {\small $1$};

\fill[black] (1.1, 1.1) circle (1pt);
\node[below] at (1.1, 1.1) {\small $s_1$};
\draw[line width = 0.8pt] (0,0) -- (1.1, 1.1);

\fill[black] (0.5, -1) circle (1pt);
\node[below] at (0.5,-1) {\small $s_2$};
\draw[line width = 0.8pt] (0,0) -- (0.5,-1);

\end{tikzpicture}
 
\caption{\label{fig7}}
\end{figure} 

It remains to argue that $s_0$ is the only zero of $Q(z)$ inside the unit disk and it is simple.
The function $Q(z)$, being  holomorphic in the compact set $\Omega$, may have at most finitely many zeros in it, say $k$,  in addition to $s_0$.
In this case, let us redefine the region $\Omega$ by excluding a line segment from the origin to each zero, see Figure \ref{fig7} which illustrates the case $k=2$.

Based on our calculations above, each such segment corresponding to the zero $s_j$ contributes the multiplicative factor $z/(z-s_j)$ to the formula \eqref{XQfla}, which becomes
$$
X_0(z) = \psi(z) Q(z)\dfrac{z}{z-s_0}\prod_{j=1}^k \dfrac{z}{z-s_j}.
$$
If we  divide this equation by $z$ and let $z\to 0$, the expression in the left hand side diverges to $\infty$, in view of \eqref{X0est}, 
while the expression in the right hand side tends to 0 by the estimate in (3.4) unless $k=0$. This contradiction shows that $Q(z)$ has no zeros 
inside the unit disk besides $s_0$.

If $s_0$ has a non-unit multiplicity $\mu_0>1$, the right hand side in \eqref{ResQ} is multiplied by $\mu_0$. Consequently, we get the formula 
$$
X_0(z) = \psi(z) Q(z)\left(\dfrac{z}{z-s_0}\right)^{\mu_0},
$$
to which the same argument by contradiction applies. 

\section{Proof of Lemma \ref{lem:An}
}\label{app:Aop}
 
Let
$\eps:= \frac 1 2 -\frac 1 2 |\sin (\pi d)|\in (0,\frac 1 2)$ and  $\delta := \frac 1 2|\sin (\pi d)|^{-1}-\frac 1 2>0$,
so that $|\sin(\pi d)|(1+\delta)= 1-\eps<1$.
By Lemma \ref{lem:tildeh}, 
the function $h(\cdot)$ is continuous and  $\lim_{r\to 0}h(r)=1$.
Hence there exists $n_0$  such that 
\begin{equation}\label{suphe}
\sup_{r\ge 0}|h(r)e^{-nr}|<1+\delta, \quad \forall n\ge n_0.
\end{equation}
In addition to the operator defined in \eqref{An},
define the operator 
\begin{equation}\label{Aop}
(  A f) (t):=  \frac {1-\eps} \pi \int_0^\infty \frac{ e^{-r}}{ e^{r+t}-1}f(r) dr.
\end{equation}
Then for all $n\ge n_0+1$, 
\begin{equation}\label{Anf}
|(A_nf)(t)| \le   \frac {1-\eps} \pi \int_0^\infty \frac{ e^{-r}}{ e^{r+t}-1} |f(r)| dr =   (A|f|)(t), \quad t\in \Real_+,
\end{equation}
and therefore it suffices to verify \eqref{contr} 
for $A$:
\begin{equation}\label{contrA}
\|Af\|\le (1-\eps) \|f\|, \quad \forall f\in L^2(\Real_+).
\end{equation}
To this end, for $f, g\in L^2(\Real_+)$,
\begin{align*}
&
\big|\langle g,A  f\rangle\big| = 
\left|
 \int_0^\infty g(t) \frac {1-\eps} \pi \int_0^\infty \frac{  e^{-  r}}{ e^{r+t}-1}f(r) drdt\right| \le \\
&
\frac {1-\eps} \pi \int_0^\infty\int_0^\infty  \Big(\frac{1-e^{-t}}{1-e^{-r}}\Big)^{1/4}  \frac{|g(t)| }{ (e^{r+t}-1)^{1/2}}
\Big(\frac{1-e^{-r}}{1-e^{-t}}\Big)^{1/4}\frac{|f(r)| }{ (e^{r+t}-1)^{1/2}} drdt  \le \\
&
\frac {1-\eps} \pi 
\Big(\int_0^\infty\int_0^\infty  \Big(\frac{1-e^{-t}}{1-e^{-r}}\Big)^{\frac 1 2}  \frac{g(t)^2 }{  e^{r+t}-1 }drdt\Big)^{\frac 1 2}
\Big( \int_0^\infty\int_0^\infty    
\Big(\frac{1-e^{-r}}{1-e^{-t}}\Big)^{\frac 1 2}\frac{f(r)^2 }{  e^{r+t}-1 } drdt
\Big)^{\frac 1 2}
\end{align*}
where the latter bound holds by the Cauchy-Schwarz inequality.
The terms on the right hand side satisfy the bound:
\begin{align*}
&
\int_0^\infty\int_0^\infty  \Big(\frac{1-e^{-t}}{1-e^{-r}}\Big)^{\frac 1 2}  \frac{g(t)^2 }{  e^{r+t}-1 }drdt = \\
&
\int_0^\infty
g(t)^2 (1-e^{-t})^{\frac 12 }
\left(\int_0^\infty \frac{(1-e^{-r})^{-\frac 12}}{  e^{r+t}-1 }dr\right) dt =\\
&
\int_0^\infty
g(t)^2 (1-e^{-t})^{\frac 12 }
\left(\int_0^1 \frac{(1-s)^{-\frac 12}}{  e^{t}-s }ds \right) dt =\\
&
\int_0^\infty
g(t)^2 
\left(\frac{1-e^{-t}}{e^t-1}\right)^{1/2}
\left(\int_1^{e^t/(e^t-1)} \frac{ (u-1)^{-\frac 12}}{   u }du \right) dt \le \\
&
\int_0^\infty
g(t)^2  e^{-t/2}
\left(\int_1^{\infty} \frac{ (u-1)^{-\frac 12}}{   u }du \right) dt \le 
B(\tfrac 1 2, \tfrac 12)\int_0^\infty
g(t)^2    dt = \pi \|g\|^2.
\end{align*}
Therefore, we obtain 
$
\big|\langle g,A  f\rangle\big| \le (1-\eps) \|g\|\|f\|
$
and consequently 
\begin{equation}\label{prein}
\|A f\|^2 = \langle A f,A f \rangle \le (1-\eps) \|A f\|\|f\|.
\end{equation}
It remains to argue that $\|A  f\|<\infty$ for all $f\in L^2(\Real_+)$, in which case \eqref{contrA} follows from \eqref{prein}, 
and consequently, the assertion of the lemma holds in view of \eqref{Anf}.

By linearity of $A$, no generality will be lost if $f(x)\ge 0$ is assumed. For $R>0$, define the bounded function $f_R(x):=f(x)\wedge R$.
Then 
\begin{align*}
\|Af_R\|^2   
\le\, & R^2 \int_0^\infty \int_0^\infty e^{-r-s}  \left(\int_0^\infty \frac{ 1 }{ e^{r+t}-1}     \frac{ 1}{ e^{s+t}-1}dt\right)  drds \le \\
&
R^2 \int_0^\infty \int_0^\infty e^{-r-s}  \left(\int_0^\infty \frac{ 1 }{ (t+r)(t+s)}dt     \right)  drds =
R^2 \int_0^\infty \int_0^\infty e^{-r-s}   \frac{\log(s/r)}{s-r}  drds = \\
&
2 R^2 \int_0^\infty e^{-s} \int_0^s e^{-r }   \frac{\log(s/r)}{s-r}  drds \le 
2 R^2 \int_0^\infty e^{-s}ds \int_1^\infty     \frac{\log u}{u(u-1)}  du  <\infty.
\end{align*}
Consequently,  \eqref{prein} implies 
\begin{equation}\label{inR}
\|A f_R\|      \le (1-\eps)  \|f_R\|.
\end{equation}
Since $f_R(x)$ is nondecreasing in $R$ for each $x\in \Real_+$, by the monotone convergence theorem, $\|f_R\|\to \|f\|$ as $R\to\infty$.
By definition \eqref{Aop}, the function $(A f_R)(t)$ also increases in $R$ for all $t$ and hence $\|A f_R\|\to \|Af\|$ 
as well, including the case $\|Af\|=\infty$.   Taking $R\to\infty$ in \eqref{inR} we conclude that $\|Af\|<\infty$.

\section{Proof of Theorem \ref{thethm}
}\label{app:B}

The key element of the proof is the auxiliary integral operator, cf. \eqref{An},
$$
(B_nf)(t) =  \frac {\sin (\pi d) }{ \pi } \int_0^\infty \frac{e^{-n \tau}}{\tau+t}   f(\tau)  d\tau.
$$
As in Lemma \ref{lem:An} 
(see also \cite[Lemma 5.6]{ChK}) it can be argued that 
$B_n$ is a contraction in $L^2(\Real_+)$, i.e., there exists $\eps\in (0,1)$ such that, for any $n\ge 1$, 
\begin{equation}\label{Bneq}
\|B_n f\|\le (1-\eps) \|f\|, \quad \forall f\in L^2(\Real_+).
\end{equation}
Using this estimate, we can show, as in Lemma \ref{lem:eqs}, 
that the equation $q  = B_n q + 1$, that is, 
\begin{equation}\label{qneq}
q  (t) = \frac {\sin (\pi d) }{ \pi }
\int_0^\infty \frac{e^{-n r}}{r+t}   q  (r)  dr +1, \quad t\in \Real_+,
\end{equation}
has a unique solution, denoted by $q_n$, such that $q_n-1\in L^2(\Real_+)$.
Changing the integration variable in  \eqref{qneq} yields
$$
q  (t/n) = \frac {\sin (\pi d) }{ \pi }
\int_0^\infty \frac{e^{-\tau }}{\tau +t }   q  (\tau/n)  d\tau  +1, \quad t\in \Real_+.
$$
Therefore, by uniqueness of the solution, it follows that $q_n(t/n)= q_1(t)$.

\subsection{Approximation of $\mathbf{S_{j,n}(z)}$}

Let us  rewrite the definition of $S_{j,n}(z)$ in (2.22)
as
\begin{align}\label{Sjnu}
S_{j,n}(z) =\, & z^j+ \frac {\sin (\pi d) }{ \pi }
\int_0^\infty 
    \frac{h(r)e^{-n r}}{z e^{r }-1}  u_{j,n}(r)   dr = \\
&
\nonumber
z^j+ 
\frac {\sin (\pi d) }{ \pi n } \frac{1}{z  -1}\int_0^\infty q_1(\tau) e^{-\tau}  d\tau
+ \frac {\sin (\pi d) }{ \pi }\Big(J_1(n)  + J_2(n) + J_3(n)\Big),
\end{align}
where we defined 
\begin{equation}\label{Jbnd}
\begin{aligned}
J_1(n) := & 
\int_0^\infty 
\frac{h(r) e^{-n r}}{z e^{r }-1} q_n(r)  dr -   n^{-1}\frac{1}{z  -1}\int_0^\infty q_1(\tau) e^{-\tau}  d\tau, 
\\
J_2(n) := &  
\int_0^\infty 
    \frac{h(r)e^{-n r}}{z e^{r }-1} (e^{jt}  -1)   dr,   \\
J_3(n) := &  \int_0^\infty 
    \frac{h(r)e^{-n r} }{z e^{r }-1} \big(u_{j,n}(t) -(e^{jt}  -1)-q_n(t)\big)  dr.   
\end{aligned}
\end{equation}
Our goal is to  show that each one of the  quantities in \eqref{Jbnd} is of order $O(n^{-2})$ as $n\to\infty$.
To  this end, we will need several estimates.

\subsubsection{Auxiliary estimates}

We start with an important implication of Lemma \ref{lem:An}.

\begin{cor}\label{lem:uj}
The solutions to equations \eqref{uweq}, guaranteed by Lemma \ref{lem:eqs}, 
satisfy the norm estimates 
\begin{align*}
&
\left(\int_0^\infty \big(u_{j,n}(t)-e^{jt}\big)^2dt\right)^{1/2} \le Cn^{-1/2},
\\
&
\left(\int_0^\infty \big(w_{j,n}(t)-e^{jt}\big)^2dt\right)^{1/2} \le Cn^{-1/2},
\end{align*}
for some constant $C$.
\end{cor}
\begin{proof}
The solutions $u_{j,n}$ to the first series of equations in \eqref{uweq} satisfy
$$
(u_{j,n}-\phi_j) = A_n (u_{j,n}-\phi_j) + A_n\phi_j
$$
where $A_n$ is the operator from  \eqref{An} and we defined $\phi_j(t):=e^{jt}$. 
The free term satisfies
\begin{align*}
\|A_n\phi_j\| = & 
\left(\int_0^\infty \left(\frac {\sin (\pi d)} \pi \int_0^\infty \frac{ h(r)e^{-n r}}{ e^{r+t}-1}\phi_j(r) dr\right)^2 dt\right)^{1/2} \le \\
&
\int_0^\infty \left(  \int_0^\infty \left(\frac{ h(r)e^{-n r}}{ e^{r+t}-1}\phi_j(r)\right)^2  dt\right)^{1/2}dr \le  \\
&
C_1 \int_0^\infty e^{-(n-n_0) r}\phi_j(r)\left(  \int_0^\infty  \frac{ 1}{ (t+r)^2}   dt\right)^{1/2}dr\le  \\
&
C_2 \int_0^\infty e^{-n r}r^{-1/2}dr\le C_3 n^{-1/2},
\end{align*}
where we applied the integral Minkowski inequality and used estimate  \eqref{suphe}.
The claim now follows from \eqref{contr}:
$$
\|u_{j,n}-\phi_j\| \le \eps^{-1}\|A_n \phi_j\| \le C n^{-1/2}.
$$
The estimates for the second series of equations in \eqref{uweq} are derived by the same argument.
\end{proof}

\begin{lem}
There exists a constant $C$ such that 
\begin{equation}\label{ubnd}
\big|u_{j,n}(t) -(e^{jt}  -1)-q_n(t)\big| \le 
 Cn^{-1}  q_n(t),\quad \forall t\in \Real_+,
\end{equation}
for all sufficiently large $n$.
\end{lem}

\begin{proof}
The first series of equations in \eqref{uweq} can be rewritten as
\begin{align}
\label{equj}
u_{j,n}(t) -(e^{jt}  -1) = &
 \frac {\sin (\pi d) }{ \pi }
\int_0^\infty 
    \frac{h(r)e^{-n r}}{e^{r+t}-1}   u_{j,n}(r)  dr +   1 =
\\
& \nonumber
\frac {\sin (\pi d) }{ \pi }
\int_0^\infty \frac{e^{-n r}}{r+t}  \big( u_{j,n}(r) -(e^{jt}  -1)\big) dr + 1+ \phi(t)    + \psi(t),
\end{align}
where we defined 
\begin{align}
\label{psidef}
\psi(t): = &
 \frac {\sin (\pi d) }{ \pi }
\int_0^\infty 
\Big(
    \frac{h(r)}{e^{r+t}-1} -\frac 1 {r+t}\Big)e^{-n r}  u_{j,n}(r)  dr,
\\
\label{phieq}
\phi(t) := &
\frac {\sin (\pi d) }{ \pi }
\int_0^\infty \frac{e^{-n r}}{r+t}  \big(e^{jr}-1\big)    dr.
\end{align}
 Let us first check that these functions satisfy the bound 
\begin{equation}\label{phipsibnd}
\|\psi\|_\infty \vee \|\phi\|_\infty  \le C_1 n^{-1},
\end{equation}
for some constant $C_1$.
For $\phi$ from \eqref{phieq}, this bound is verified directly:    
$$
|\phi(t)| \le 
\int_0^\infty \frac{e^{-n r}}{r}  (e^{jr}-1)    dr \le 
j \int_0^\infty  e^{-(n-j) r}        dr  \le C_2 n^{-1},
$$
where  $C_2$ is some constant.
 To check \eqref{phipsibnd} for $\psi$, note that,  
by Lemma 5.1, the function $h(t)=\widetilde h(e^t)$ satisfies  $h(t)=1+ O(t)$ as $t\to 0$  and  
$
\sup_{t\ge 0}e^{-\alpha t} |h(t)|<\infty
$
with $\alpha :=  (2q-p)\vee 0$.
Consequently the kernel 
$$
L(r,t):= \Big(\frac{h(r)}{e^{r+t}-1} -\frac 1 {r+t}\Big)e^{-\alpha r}, \quad r,t\in \Real_+,
$$
is uniformly bounded:
$$
|L(r,t)| \le \sup_{r\ge 0}|h(r)|e^{-\alpha r} \sup_{x>0}\Big(\frac {1} {x}-\frac{1}{e^{x}-1} \Big)
+
 \sup_{r\ge 0}\frac {|h(r)-1|} {r} e^{-\alpha r} <\infty.
$$
Therefore, for all $n$ large enough, the function in \eqref{psidef} satisfies 
$$
|\psi(t)| \le  
\|L\|_\infty\int_0^\infty    e^{-(n-\alpha) r}  \big|u_{j,n}(r)-e^{jr}\big|  dr + 
\|L\|_\infty \int_0^\infty    e^{-(n-\alpha-j) r}       dr \le C_3 n^{-1}
$$
for some constant $C_3$. The latter bound is obtained by applying the Cauchy-Schwarz inequality  to the first integral and by using the bound from Lemma \ref{lem:uj}. This verifies \eqref{phipsibnd}.

A calculation as in the proof of Lemma \ref{lem:uj} shows that
$B_n (1+\phi +\psi)\in L^2(\Real_+)$.
Thus, in view of estimate \eqref{Bneq}, the Neumann series for \eqref{equj} yields  
\begin{equation}\label{ujnu}
\begin{aligned}
u_{j,n}(t) -(e^{jt}  -1) = &  1+ \phi(t)    + \psi(t) +  R_n(1+\phi+\psi) (t) =\\
& 
1+ (R_n 1)(t) + \phi(t)    + \psi(t) +  R_n (\phi+\psi)(t),
\end{aligned}
\end{equation}
where 
$
R_n = \sum_{j=1}^\infty B_n^j
$
is the resolvent operator.
The last term satisfies  
$$
\big| R_n(\phi+\psi) (t)\big|  \le 
  R_n(|\phi|+|\psi|) (t)   \le  2C_1n^{-1} (R_n 1)(t) = 2C_1n^{-1} (q_n(t)-1),
$$
where the first inequality holds since $R_n$ is a series of integral operators with nonnegative kernels,
the second bound is due to \eqref{phipsibnd} and the equality holds since the function $q_n := 1+R_n1$ is the solution to equation \eqref{qneq}.
Substituting this bound along with \eqref{phipsibnd} into \eqref{ujnu} we get \eqref{ubnd}:
$$
\begin{aligned}
\big|u_{j,n}(t) -(e^{jt}  -1)-q_n(t)\big| = &   \big|\phi(t)    + \psi(t) +  R_n (\phi+\psi) (t)\big|\le\, \\  
&  2C_1n^{-1} + 2C_1n^{-1} (q_n(t)-1)= 
2C_1n^{-1}  q_n(t).
\end{aligned}  
$$
\end{proof}

\subsubsection{Bounds for expressions in \eqref{Jbnd}}\label{ssec:B11}
Due to the bound in \eqref{ubnd}, the last integral in \eqref{Jbnd} satisfies
\begin{align*}
|J_3(n)| \le\, &
 2Cn^{-1}\int_0^\infty 
    \left|\frac{h(r)}{z e^{r }-1} \right|   |q_n(r)| e^{-n r}  dr = \\
&    
2Cn^{-2}\int_0^\infty 
    \left|\frac{h(\tau/n)}{z e^{\tau/n }-1}\right|    |q_1(\tau)| e^{-\tau }  d\tau  =  \\
&
2Cn^{-2}\left|\frac{1}{z  -1}\right| \int_0^\infty 
       |q_1(\tau)| e^{-\tau }  d\tau (1+o(1)), \quad n \to \infty.    
\end{align*}
The second integral admits the estimate 
\begin{align*}
|J_2(n)|\le\, &
\int_0^\infty 
    \frac{|h(r)|}{|z e^{r }-1|} (e^{jr}  -1) e^{-n r}  dr= \\
&    
n^{-1}  \int_0^\infty 
    \frac{|h(\tau/n)|}{|z e^{\tau/n }-1|} (e^{j\tau/n}  -1) e^{-  \tau }  d\tau \le \\
&
n^{-2}  j \int_0^\infty 
    \frac{|h(\tau/n)|}{|z e^{\tau/n }-1|}  \tau e^{j\tau/n}  e^{-  \tau }  d\tau   
    = \\
&    
    n^{-2}
     \frac{j}{|z  -1|} \int_0^\infty 
      \tau    e^{-  \tau }  d\tau (1+o(1)), \quad n \to \infty.      
\end{align*}
Finally, 
\begin{align*}
&
n^2 |J_1(n)|  =  
n\left|
\int_0^\infty 
\left(
 \frac{h(\tau/n)}{z e^{\tau/n }-1}    -    \frac{1}{z  -1}  \right) q_1(\tau)e^{-\tau}  d\tau
\right|  \le \\
&
\int_0^\infty 
\left|\frac{n\big(h(\tau/n)-1\big)}{z e^{\tau/n }-1}  \right|   |q_1(\tau)|e^{-\tau}  d\tau
+ 
|z| \int_0^\infty 
\left| \frac{n(1 - e^{\tau/n }) }{(z e^{\tau/n }-1)(z-1)}    \right| |q_1(\tau)| e^{-\tau}  d\tau \\
&  \xrightarrow[n\to \infty]{} 
\frac{C_4}{|z  -1|}   
\int_0^\infty 
 \tau |q_1(\tau)|e^{-\tau}  d\tau
+ 
\frac{|z| }{ |z-1|^2} \int_0^\infty 
  \tau |q_1(\tau)| e^{-\tau}  d\tau.
\end{align*}
Plugging these estimates into \eqref{Sjnu} we obtain the asymptotic approximation for $S_{j,n}(z)$ in (6.1):
\begin{equation}\label{prove1-S}
S_{j,n}(z)  =\,   z^j+ \frac {\sin (\pi d) }{ \pi  } \frac{\lambda_0}{z  -1} n^{-1}+ O(n^{-2}), \quad \text{as}\ n \to \infty.
\end{equation}

\subsection{Approximation of $\mathbf{D_{j,n}(z)}$ }

The  approximation for $D_{j,n}(z)$ is obtained similarly. Let us write, cf. \eqref{uweq}, 
$$
w_{j,n}(t) =  
-\frac {\sin (\pi d) }{ \pi }\int_0^\infty \frac{e^{-n r}}{ r+t } \big( w_{j,n}(r) - (e^{jt}-1)\big)  dr 
+ \phi(t) + \psi(t) +  e^{jt},
$$
where 
\begin{align*}
\psi(t) := &
-\frac {\sin (\pi d) }{ \pi }\int_0^\infty \Big(\frac{h(r) e^{-n r}}{e^{r+t}-1}-\frac{e^{-n r}}{ r+t }\Big)  w_{j,n}(r)  dr,
\\
\phi(t) := & -\frac {\sin (\pi d) }{ \pi }\int_0^\infty \frac{e^{-n r}}{ r+t }     (e^{jt}-1)   dr.
\end{align*}
The estimates \eqref{phipsibnd} remain valid for these functions as well by the very same arguments.
We can rewrite 
$$
w_{j,n}(t) -(e^{jt}-1)= 
-\frac {\sin (\pi d) }{ \pi }\int_0^\infty \frac{e^{-n r}}{ r+t } \big( w_{j,n}(r) - (e^{jt}-1)\big)  dr 
+1 + \phi(t) + \psi(t)
$$
and as before  
\begin{align*}
w_{j,n}(t) -(e^{jt}-1) =\, & 1 + \phi(t) + \psi(t) +  \widetilde R_n(1+\phi+\psi) (t)= \\
&
1 + (\widetilde R_n1)(t)+ \phi(t) + \psi(t) +  \widetilde R_n(\phi+\psi) (t),
\end{align*}
where $\widetilde R_n = \sum_{j=1}^\infty (-B_n)^j$. 
Note that, for any $f\in L^2(\Real_+)$,
$$
|\widetilde R_nf| =  \left|\sum_{j=1}^\infty (-B_n )^jf\right| \le \sum_{j=1}^\infty  B_n^j|f| =  R_n|f| (t),
$$
and hence 
$$
\big| \widetilde R_n(\phi+\psi) (t)\big| \le  R_n(|\phi|+|\psi|)  (t) \le 2Cn^{-1} (R_n1)(t).
$$
Thus we obtain the bound analogous to \eqref{ubnd}: 
\begin{equation}\label{abnd}
\begin{aligned}
&
|w_{j,n}(t) -(e^{jt}-1) - p_n(t)| \le 
 \big|\phi(t)\big| + \big|\psi(t)\big| + \big| \widetilde R_n(\phi+\psi) (t)\big| \le \\
 &
2Cn^{-1}  + 2Cn^{-1} (R_n1)(t) =  
 2Cn^{-1} + 2Cn^{-1} (q_n(t)-1) \le Cn^{-1} q_n(t),
\end{aligned}
\end{equation}
where the function $p_n := 1+\widetilde R_n1$ is the unique solution to the integral equation 
$$
p_n = -B_n p_n +1.
$$
Now we can decompose $D_{j,n}(z)$ similarly to  \eqref{Sjnu} and estimate each of the obtained terms as in subsection \ref{ssec:B11} using \eqref{abnd}.
This yields the approximation for $D_{j,n}(z)$ in (6.1):  
\begin{equation}\label{prove1-D}
D_{j,n}(z) =\,   z^j  
-\frac {\sin (\pi d) }{ \pi   }\frac{\mu_0}{z-1} n^{-1}  +  O(n^{-2}), \qquad \text{as}\ n \to \infty.
\end{equation}

\section{Proof of Theorem \ref{thm:const}
}\label{app:C}
The constants in (6.2):
$$
\lambda_0 = \int_0^\infty  q_1 (\tau) e^{-  \tau}  d\tau, \quad 
\mu_0 = \int_0^\infty  p_1(\tau) e^{-  \tau}  d\tau,  
$$
are computed by leveraging a non-obvious connection between equations \eqref{q1p1}  
and the following integral equations on the unit interval:
\begin{equation}\label{C1}
\int_0^1 u(y) |x-y|^{-\alpha} \sign(x-y) dy =1, \quad x\in (0,1),
\end{equation}
and 
\begin{equation}\label{C2}
\int_0^1 u(y) |x-y|^{-\alpha} dy = 1, \quad x\in (0,1),
\end{equation}
with $\alpha \in (0,1)$.
Closed-form solutions to these equations are available due to \cite{LC67} and they can be used to study the properties of solutions to \eqref{q1p1}. 
In particular, this will allow us to find the exact values in Theorem \ref{thm:const}.

We provide a detailed account of the computation for the case $d\in (0,\frac 1 2)$, demonstrating how the value of $\lambda_0$
in \eqref{values} is derived. To this end, we will use equation \eqref{C1} with $\alpha:= 2d$. 
The other cases can be treated similarly with minor adjustments. In particular, the value of $\mu_0$ in (6.4) is obtained by analyzing equation \eqref{C2} with  $\alpha := 1-2d$. 
The same values for the constants are also obtained for  $d\in (-\frac 1 2, 0)$ by switching the roles of the equations in \eqref{q1p1}.

\subsection{The solution to \eqref{C1}.}
As shown in \cite{LC67} the general solution to equation 
$$
\int_0^1 u(y) |x-y|^{-\alpha} \sign(x-y) dy =f(x), \quad x\in (0,1),
$$
has the form 
$
u =c u_0+u_1
$
with $c\in \Real$, where 
$$
u_0(x)= x^{-1+\alpha/2} (1-x)^{-1+\alpha/2},
$$
and
$$
u_1(x)=
\frac 1 {h(\alpha) \Gamma(\alpha/2)^2} 
\frac d{dx} x^{\alpha/2}\int_x^1 t^{-\alpha} (t-x)^{-1+\alpha/2}dt 
\int_0^t s^{\alpha/2}(t-s)^{-1+\alpha/2}f(s) ds,
$$
with the constant  
$$
h(\alpha) = 2 \sin \big(\tfrac {1-\alpha} 2\pi\big)\Gamma(1-\alpha).
$$ 
For the particular free term $f(z)=1$ in \eqref{C1}  a direct calculation reduces the latter expression  to 
\begin{equation}\label{g1}
u_1(x)= c(\alpha) x^{\alpha/2-1}(1-x)^{\alpha/2-1}(1-2x)
\end{equation}
with the constant 
$$
c(\alpha)= \frac{B(\frac\alpha 2+1, \frac\alpha 2) }{ h(\alpha)\Gamma(\frac\alpha 2)^2}.
$$
For our purposes, it will be convenient to use a solution to \eqref{C1}, antisymmetric around $\frac 1 2$, i.e., such that $u(x)=-u(1-x)$.
Since $u_0$ is symmetric and $u_1$ is antisymmetric, such solution is unique, corresponding to $c=0$, i.e. $u(x)= u_1(x)$.

\subsection{The solution to \eqref{q1p1}.}
In this subsection, we  consider the equation
\begin{equation}\label{qalpha}
q(t) = \frac{\sin (\pi \alpha/2) }{\pi} \int_0^\infty \frac{e^{-\tau}}{\tau+t}q(\tau)d\tau +1, \quad t\in \Real_+.
\end{equation}
For $\alpha := 2d$, this is the first equation in \eqref{q1p1}. 
As discussed above, cf. \eqref{qneq}, this equation has the unique solution such that $q(\cdot)-1\in L^2(\Real_+)$. 
We will express this solution in terms  of the solution to \eqref{C1} and, thus, 
will be able to study its properties using the explicit formula \eqref{g1}. 

The construction essentially follows the same approach that we applied to the prediction problem in this paper. We will argue that the Laplace transform of the solution to \eqref{C1} solves a specific Hilbert boundary value problem. We will then show that this problem has a unique solution and relate it to the unique solution of \eqref{qalpha}.
Finally, we will use this relation to derive the exact value of the constant in question.

\medskip

\subsubsection{The Laplace transform}
The following lemma provides the key representation formula for the Laplace transform of the solution to \eqref{C1}.
\begin{lem}\label{lem:Lap}
The Laplace transform of the antisymmetric solution to \eqref{C1}:
\begin{equation}\label{ULap}
U(z) := \int_0^1 u(x) e^{-zx}dx
\end{equation}
satisfies the representation 
\begin{equation}\label{Uz}
z U(z)=\frac{\Psi(z)  -e^{-z} \Psi(-z)}{\Lambda(z)}, \quad z\in \mathbb C,
\end{equation}
where 
\begin{equation}\label{PsiL}
\Psi(z) = -\Gamma(\alpha) + z \int_0^\infty \frac{t^{\alpha-1}}{t-z} U(t)dt, \quad z\in \mathbb C\setminus \Real_+,
\end{equation}
and
$$
\Lambda(z) = \frac \pi{\sin (\pi \alpha)} \Big((-z)^{\alpha-1}-z^{\alpha-1}\Big), \quad z\in \mathbb C\setminus \Real,
$$
with the principle branch $\arg(z)\in (-\pi,\pi]$.
\end{lem}

\begin{proof}
Substitute the identity 
$$
|x-y|^{-\alpha}=\frac 1{\Gamma(\alpha)} \int_0^\infty t^{\alpha-1} e^{-t|x-y|}dt 
$$
into \eqref{C1} and apply the Laplace transform to its left hand side to obtain:
\begin{align*}
&
\int_0^1 e^{-zx}\left(
\int_0^1 u(y)\left(\frac 1{\Gamma(\alpha)} \int_0^\infty t^{\alpha-1} e^{-t|x-y|}dt\right) \sign(x-y) dy \right)dx
 =\\
&
\frac 1{\Gamma(\alpha)}\int_0^\infty t^{\alpha-1} \int_0^1 u(y) \left(
\int_0^1 e^{-zx} 
  e^{-t|x-y|} \sign(x-y)  dx
 \right)dydt =\\
& 
\frac 1{\Gamma(\alpha)}\int_0^\infty  \frac {t^{\alpha-1}} {t-z} \int_0^1 u(y)  \Big(e^{-ty}-e^{-zy}\Big)dydt + \\
&
\frac 1{\Gamma(\alpha)}\int_0^\infty  \frac {t^{\alpha-1}} {t+z} \int_0^1 u(y) \Big(e^{-yz}-e^{-z-t(1-y)}\Big)dydt  =\\
&
\frac 1{\Gamma(\alpha)}\int_0^\infty  \frac {t^{\alpha-1}} {t-z}
\Big(U(t)-U(z)\Big) dt + 
\frac 1{\Gamma(\alpha)}\int_0^\infty  \frac {t^{\alpha-1}} {t+z} 
\Big(U(z)+e^{-z} U (t)\Big)dt,
\end{align*}
where  the last equality is due to antisymmetry of $u(\cdot)$.
The Laplace transform of the right hand side of \eqref{C1} is $(1-e^{-z})/z$. Equating these two expressions and rearranging, we arrive at
\eqref{Uz}, with  $\Psi(z)$ defined in \eqref{PsiL} and
$$
\Lambda(z)= \int_0^\infty  \frac {t^{\alpha-1}} {t-z}  dt- \int_0^\infty  \frac {t^{\alpha-1}} {t+z} dt = 
\frac \pi{\sin (\pi\alpha )} \Big((-z)^{\alpha-1}-z^{\alpha-1}\Big).
$$
\end{proof}

The following lemma summarizes some relevant properties of  $\Lambda(z)$.

\begin{lem}
The function $\Lambda(z)$ is non-vanishing and sectionally holomorphic in $\mathbb C\setminus \Real$.
Its limits 
$$
\Lambda^\pm(t):=\lim_{z\to t^{\pm}}\Lambda(z), \quad t\in \Real\setminus\{0\},
$$
satisfy 
$
\Lambda^\pm(-t)=-\Lambda^\mp(t)
$
and  
\begin{equation}\label{Lpm}
\frac{\Lambda^+(t)}{\Lambda^-(t)} = e^{-\pi i \alpha }, \quad t\in \Real_+.
\end{equation}
\end{lem}

\begin{proof}
Let $z=r e^{i\theta}$ with $r\in \Real_+$ and $\theta \in (0,\pi)$, then 
$$
\Lambda(z) = \frac \pi{\sin (\pi\alpha )} r^{\alpha-1}\Big(e^{-i\theta(\alpha-1)}-e^{i\theta(\alpha-1)}\Big)=
 \frac {2\pi i}{\sin (\pi\alpha )} r^{\alpha-1}\sin(\theta(1-\alpha))\ne 0.
$$
Similarly, $\Lambda(z)\ne 0$ for $z = re^{i\theta}$ with $\theta\in (-\pi, 0)$. 
By definition, $\Lambda(-z)=-\Lambda(z)$, which implies $\Lambda^\pm(-t)=-\Lambda^\mp(t)$.
Identity \eqref{Lpm} follows by direct calculation:
$$
\frac{\Lambda^+(t)}{\Lambda^-(t)} =  \frac{
e^{-\pi i(\alpha-1)}t^{\alpha-1}-t^{\alpha-1}
}
{
e^{\pi i(\alpha-1)}t^{\alpha-1}-t^{\alpha-1}
} = e^{-\pi i\alpha }, \quad t\in \Real_+.
$$ 
 
\end{proof}

\subsubsection{The Hilbert problem}

The next lemma formulates the Hilbert problem to which the function $\Psi(z)$  from Lemma \ref{lem:Lap} is a particular solution.

\begin{lem}\label{lem:H}
The function $\Psi(\cdot)$ defined in \eqref{PsiL} is sectionally holomorphic in $\mathbb C\setminus \Real_+$. 
Its limits on $\Real_+$ satisfy the boundary condition 
\begin{equation}\label{Psibnd}
\Psi^+(t) -\frac{\Lambda^+(t)}{\Lambda^-(t)} \Psi^-(t) = -e^{-t} \Psi(-t) \Big(\frac{\Lambda^+(t)}{\Lambda^-(t)}-1\Big), \quad t\in \Real_+
\end{equation}
and the growth estimates  
\begin{equation}\label{Psiest}
\Psi(z) = \begin{cases}
-\Gamma(\alpha) + O(z), & z\to 0, \\
O(z^{\alpha/2}), & z\to \infty.
\end{cases}
\end{equation}
\end{lem}

\begin{proof}
Since the integration in  \eqref{ULap} is carried out over a finite interval,  the Laplace transform  defines an entire function.
It follows that all singularities in the right hand side of \eqref{Uz} must be removable. This implies 
$$
\lim_{z\to t^+}\frac{\Psi(z)  -e^{-z} \Psi(-z)}{\Lambda(z)} = 
\lim_{z\to t^-}\frac{\Psi(z)  -e^{-z} \Psi(-z)}{\Lambda(z)}, \quad t\in \Real_+, 
$$
and, consequently,  condition \eqref{Psibnd}. If follows from \eqref{g1} and \eqref{ULap} that 
$$
U(t) = \begin{cases}
O(t), & t\to 0,\\
O(t^{-\alpha/2}), & t\to\infty.
\end{cases}
$$
Combining these estimates with the definition \eqref{PsiL} yields \eqref{Psiest}. 
\end{proof}

\subsubsection{Solution to the Hilbert problem}

As mentioned in Appendix \ref{app:B}, 
integral equation \eqref{qalpha} has a unique solution such that $q-1\in L^2(\Real_+)$. 
The following lemma establishes the relation between this solution and the Hilbert problem from Lemma \ref{lem:H}.

\begin{lem}\label{lem:Psiq}
Any solution to the Hilbert problem \eqref{Psibnd}-\eqref{Psiest} satisfies  
$$
\Psi(-t) t^{-\alpha/2} = b q(t), \quad t\in \Real_+,
$$
for some constant $b\in \Real$, where $q(\cdot)$ is the unique solution to \eqref{qalpha}.
\end{lem}

\begin{rem}
The function defined in \eqref{PsiL} is a particular solution which corresponds to a special value of $b$, which will be found,
as part of the proof of Lemma \ref{lem:E5} below, see the equations \eqref{sle}. 
\end{rem}

\begin{proof}
Define the sectionally holomorphic function
$$
X(z) = (-z)^{\alpha/2}, \quad z\in\mathbb C\setminus \Real_+,
$$
where the principle branch with $\arg(z)\in (-\pi,\pi]$ is taken. The limits of this function on $\Real_+$ 
$$
X^\pm(t) = \exp\Big(\mp \frac \pi 2 i\alpha \Big)t^{\alpha/2}, \quad t\in \Real_+,
$$ 
satisfy, cf. \eqref{Lpm},   
$$
\frac{X^+(t)}{X^-(t)}=e^{-\pi i \alpha}= \frac{\Lambda^+(t)}{\Lambda^-(t)}, \quad t\in \Real_+.
$$
It follows from Lemma \ref{lem:H} that the function $D(z)=\Psi(z)/X(z)$ 
satisfies the boundary condition 
$$
D^+(t) -  D^-(t) = 2i \sin\Big(\frac {\pi \alpha} 2\Big)e^{-t} D(-t), \quad t\in \Real_+.
$$
In view of estimates \eqref{Psiest}, it follows that
\begin{equation}\label{Dest}
D(z) = \begin{cases}
O(z^{-\alpha/2}), & z\to 0, \\
O(1), & z\to \infty.
\end{cases}
\end{equation}
Thus by the Sokhotski-Plemelj theorem 
$$
D(z) = \frac {\sin \frac {\pi \alpha} 2}\pi \int_0^\infty \frac{e^{-\tau}}{\tau-z}D(-\tau)d\tau + b, \quad z\in \mathbb C\setminus \Real_+,
$$
for some constant $b$.
The verification of this representation is based on estimates \eqref{Dest} and is carried out as in (5.7).
In particular, by setting $z:=-t$, we see that the function $D(-t)$, $t\in \Real_+$ solves the integral equation 
$$
D(-t) = \frac {\sin \frac {\pi \alpha} 2}\pi \int_0^\infty \frac{e^{-\tau}}{\tau+t}D(-\tau)d\tau + b, \quad t \in \Real_+.
$$
By linearity of this equation and uniqueness of its solution we conclude that $D(-t)=bq(t)$ and assertion of the lemma follows.
\end{proof}

\subsection{Computation of $\lambda_0$.}
We are now in a position to calculate the constant in question. 
The value of $\lambda_0$ in \eqref{values} is provided by the next lemma with $\alpha =2d$.

\begin{lem}\label{lem:E5}
The solution to \eqref{qalpha} satisfies 
$$
\int_0^\infty e^{-t}q(t)dt = \frac{\pi}{\sin (\pi \alpha  /2)}\frac \alpha 2 \left(\frac \alpha 2 +1\right).
$$
\end{lem}

\begin{proof}
It follows from \eqref{qalpha} that
\begin{equation}\label{limt}
\frac{\sin (\pi \alpha  /2)}{\pi}\int_0^\infty e^{-\tau}q(\tau)d\tau = \lim_{t\to\infty}t(q(t)-1),
\end{equation}
and it remains to compute the value of the limit. By Lemma \ref{lem:Psiq},
$$
t(q(t)-1)= t\left(\frac{1}{b }t^{-\alpha/2}\Psi(-t) -1\right), \quad t\in \Real_+,
$$
where $b$ is some constant. This equality holds for $\Psi(z)$ defined in \eqref{PsiL} for 
a particular value of this constant, which will be found in the calculations below, see \eqref{sle}.
Thus we need to establish the precise asymptotics as $t\to \infty$ of 
$$
\Psi(-t) =  
-\Gamma(\alpha) -t \int_0^\infty \frac{\tau^{\alpha-1}}{\tau+t} U(\tau)d\tau.
$$
Substitution of \eqref{g1} and \eqref{ULap} into the  integral on the right hand side gives 
\begin{align*}
&
t \int_0^\infty \frac{\tau^{\alpha-1}}{\tau+t} U(\tau)d\tau = \\
&
t \int_0^\infty \frac{\tau^{\alpha-1}}{\tau+t} 
\left(\int_0^1 c(\alpha) x^{\alpha/2-1}(1-x)^{\alpha/2-1}(1-2x) e^{-\tau x}dx\right)
d\tau =\\
&
c(\alpha)t \int_0^\infty \frac{\tau^{\alpha/2-1 }}{\tau+t} 
\int_0^\tau  e^{-s}s^{\alpha/2-1}(1-s/\tau  )^{\alpha/2-1}(1-2s/\tau ) ds 
d\tau =: I(t)+J(t),
\end{align*}
where we defined 
\begin{align*}
I(t) :=\, &
c(\alpha)t \int_0^\infty \frac{\tau^{\alpha/2-1 }}{\tau+t} 
\int_0^\tau  e^{-s}s^{\alpha/2-1}  ds 
d\tau,
\\
J(t) :=\, &
c(\alpha)t \int_0^\infty \frac{\tau^{\alpha/2-1 }}{\tau+t} 
\varphi(\tau) 
d\tau,
\end{align*}
and 
\begin{equation}\label{varphi}
\varphi(\tau):=\int_0^\tau  e^{-s}s^{\alpha/2-1}\Big((1-s/\tau  )^{\alpha/2-1}(1-2s/\tau )-1\Big) ds.
\end{equation}

To estimate $I(t)$ as $t\to\infty$,  let us split it  into $I(t)=I_1(t)-I_2(t)$ where 
$$
I_1(t) =\, 
c(\alpha)t \int_0^\infty \frac{\tau^{\alpha/2-1 }}{\tau+t} 
\int_0^\infty  e^{-s}s^{\alpha/2-1}  dsd\tau =  
 c(\alpha) \Gamma(\tfrac\alpha 2) B(\tfrac\alpha 2, 1-\tfrac\alpha 2)t^{\alpha/2}=: c_0t^{\alpha/2},
$$
and 
\begin{align*}
I_2(t) =\, & 
c(\alpha) \int_0^\infty \frac{t }{\tau+t} \tau^{\alpha/2-1 }
\int_\tau^\infty  e^{-s}s^{\alpha/2-1}  dsd\tau =\\
&
c(\alpha) \int_0^\infty \bigg(1-\frac \tau t+\frac{\tau^2}{t(t+\tau)}\bigg) \tau^{\alpha/2-1 }
\int_\tau^\infty  e^{-s}s^{\alpha/2-1}  dsd\tau =:\\
&
c_1 - c_2 \frac 1 t + O(t^{-2}), \quad t\to\infty.
\end{align*}

To estimate $J(t)$, let us write 
\begin{equation}\label{Jint}
J(t) =\,   
c(\alpha)  \int_0^\infty  \tau^{\alpha/2-1 } 
\varphi(\tau) 
d\tau -
c(\alpha)  \int_0^\infty \frac{\tau^{\alpha/2  }}{\tau+t} 
\varphi(\tau) 
d\tau.
\end{equation}
A standard calculation shows that the function defined in \eqref{varphi} satisfies 
$$
\varphi(\tau) =\begin{cases}
b_1\tau^{\alpha/2} (1+o(1)), & \tau \to 0, \\
b_2\tau^{-1} (1+o(1)) , & \tau\to \infty,
\end{cases}
$$
with some constant $b_1$ and
\begin{align*}
b_2 = & \lim_{\tau\to\infty}\tau \varphi(\tau) = 
\lim_{\tau\to\infty} \int_0^\tau  e^{-s}s^{\alpha/2 }\frac{(1-s/\tau  )^{\alpha/2-1}(1-2s/\tau )-1}{s/\tau} ds
= \\
&
-(\tfrac \alpha 2 +1) \int_0^\infty  e^{-s}s^{\alpha/2 }  ds = -(\tfrac \alpha 2 +1)\Gamma(\tfrac \alpha 2 +1).
\end{align*}
Therefore, both integrals in the decomposition \eqref{Jint}
are well defined and the second integral satisfies
\begin{align*}
&
c(\alpha)\int_0^\infty \frac{\tau^{\alpha/2  }}{\tau+t} 
\varphi(\tau) 
d\tau =
t^{\alpha/2}c(\alpha)\int_0^\infty \frac{v^{\alpha/2  }}{v +1} 
\varphi(vt) 
dv = \\
&
t^{\alpha/2-1}c(\alpha)\int_0^\infty \frac{v^{\alpha/2-1  }}{v +1} 
 \big(vt\varphi(vt)\big)
dv = c_4 t^{\alpha/2-1} (1+o(1)), \quad t\to\infty,  
\end{align*}
with the constant 
$
c_4 = c(\alpha) b_2 B(\frac \alpha 2, 1-\frac \alpha 2).
$
Thus we get 
$$
J(t) = c_3 - c_4 t^{\alpha/2-1} (1+o(1)), \quad t\to\infty,
$$ 
where $c_3$  stands for the value of the first integral in \eqref{Jint}.

Gathering all parts together we obtain the following asymptotic expansion:
$$
q(t) = \frac 1 b t^{-\alpha/2} \Psi(-t)=
\frac 1 b t^{-\alpha/2}
\Big(
-c_0t^{\alpha/2} -\Gamma(\alpha)+c_1 -c_3  + c_4   t^{\alpha/2-1} (1+o(1))  
\Big).
$$
Existence of the limit in \eqref{limt} implies 
\begin{equation}\label{sle}
\begin{aligned}
&
-\frac {c_0} b -1 =0,
\\
&
-\Gamma(\alpha)+c_1 -c_3 =0,
\end{aligned}
\end{equation}
in which case 
$$
t(q(t)-1)\xrightarrow[t\to\infty]{}  \frac{c_4} b = -\frac{c_4} {c_0} =
\frac
{
c(\alpha)   (\tfrac \alpha 2 +1)\Gamma(\tfrac \alpha 2 +1) B(\frac \alpha 2, 1-\frac \alpha 2)
}
{
c(\alpha) \Gamma(\tfrac\alpha 2) B(\tfrac\alpha 2, 1-\tfrac\alpha 2)
} =  
\frac
{
    (\tfrac \alpha 2 +1)\Gamma(\tfrac \alpha 2 +1)  
}
{
  \Gamma(\tfrac\alpha 2)  
} = \tfrac \alpha 2 (\tfrac \alpha 2+1).
$$

\end{proof}

\section{Zeros with multiplicities}\label{app:D}
In this section, we elaborate on the adjustments needed in order to extend the proof of Theorem \ref{main-thm} 
to the case when zeros have non-unit multiplicities. 
Let $z_i$ be a zero of the polynomial $\theta(z)$ with multiplicity $\mu_i\in \mathbb N$. Then $z_i^{-1}$ is a zero of 
the reciprocal polynomial $z^q\theta(z^{-1})$  with the same multiplicity.
In view of (4.5): 
$$
2\pi \big(1- G(z)\big)\theta(z)z^q\theta(z^{-1})=   \frac{\Phi_0(z)\phi(z) + z^{n+2q} \Phi_1(z^{-1}) \phi(z^{-1})}{ Q(z)}, \quad z\in \mathbb C \setminus \Real_+,
$$ 
both $z_i$ and $z_i^{-1}$ are zeros with  multiplicity $\mu_i$ of the function in the numerator in the right hand side. 
Therefore condition \eqref{algcond}, 
corresponding to $z_i$, is replaced with $\mu_i$ conditions 
$$
\begin{aligned}
&
\frac {d^j}{dz^j} \Big(
\Phi_0(z)\phi(z) + z^{n+2q} \Phi_1(z^{-1}) \phi(z^{-1})
\Big)_{\big| \displaystyle  z =z_i } \ =0, \\
&
\frac {d^j}{dz^j} \Big(
\Phi_0(z)\phi(z) + z^{n+2q} \Phi_1(z^{-1}) \phi(z^{-1})
\Big)_{\big| \displaystyle  z =z_i^{-1} } =0,
\end{aligned}  \qquad j=0,...,\mu_i-1.
$$
If $|z_i|<1$, then asymptotically as $n\to \infty$, these equations reduce, after applying the general Leibniz rule, to
$$
\begin{aligned}
&
 \Phi_0^{(j)}(z_i) = O(|z_i|^n), \\
&
 \Phi_1^{(j)}(z_i) = O(|z_i|^n),
\end{aligned}
\qquad j=0,...,\mu_i-1.
$$
Since $X(z)\ne 0$, by definitions \eqref{SD}, 
these conditions further reduce to 
$$
\begin{aligned}
&
 S^{(j)}(z_i)+D^{(j)}(z_i)  = O(|z_i|^n),   \\
&
 S^{(j)}(z_i)-D^{(j)}(z_i)  = O(|z_i|^n),
\end{aligned}
\qquad j=0,...,\mu_i-1,
$$
or, equivalently, to
$$
\begin{aligned}
&
 S^{(j)}(z_i)   = O(|z_i|^n),   \\
&
 D^{(j)}(z_i)  = O(|z_i|^n),
\end{aligned}
\qquad j=0,...,\mu_i-1.
$$
Similarly, if $|z_i|>1$,  
$$
\begin{aligned}
&
 S^{(j)}(z_i^{-1})   = O(|z_i|^{-n}),   \\
&
 D^{(j)}(z_i^{-1})  = O(|z_i|^{-n}),
\end{aligned}
\qquad j=0,...,\mu_i-1.
$$
With $\zeta_i$ being defined in \eqref{zeta},
the two types of conditions can be jointly written as, cf. \eqref{geom},
\begin{equation}\label{dSjDj}
\begin{aligned}
&
    \sum_{j=0}^q a_j  
       S^{(\ell)}_{j,n}(\zeta_i )      
=O(|\zeta_i|^n),\\
&
  \sum_{j=0}^q b_j   
      D^{(\ell)}_{j,n}(\zeta_i )    
=O(|\zeta_i|^n),
\end{aligned}  \qquad \ell = 0,...,\mu_i-1,
\end{equation}
where $S_{j,n}(\cdot)$ and $D_{j,n}(\cdot)$ are defined in (2.22).

A close look at the proof of Theorem \ref{thethm} 
shows that the derivatives of the quantities defined in (2.22) satisfy
\begin{equation}\label{dSDlim}
\begin{aligned}
S^{(\ell)}_{j,n}(z)  =\, & \frac{j!}{(j-\ell)!}z^{j-\ell}\one{j\ge \ell}+ 
\frac {\sin (\pi d) }{ \pi  } \frac{(-1)^\ell\ell! }{(z  -1)^{\ell+1}}\lambda_0 n^{-1}+ O(n^{-2}),  \\
D^{(\ell)}_{j,n}(z) =\, & \frac{j!}{(j-\ell)!}z^{j-\ell}\one{j\ge \ell} 
-\frac {\sin (\pi d) }{ \pi   }\frac{(-1)^\ell\ell! }{(z  -1)^{\ell+1}}\mu_0 n^{-1}
 +  O(n^{-2}),
\end{aligned}\qquad n\to\infty.
\end{equation}

Let us now describe how asymptotic conditions  \eqref{asys} and \eqref{bsys}, 
which determine $a_j$'s and $b_j$'s,
change when some zeros have non-unit multiplicities. 
Suppose $z_1,...,z_{r-1}$ are distinct zeros of $\theta(z)$ with multiplicities $\mu_1,...,\mu_{r-1}$ respectively
and let $z_r:= s_0$ and $\mu_r =1$, where $s_0$ is the only zero  of $Q(z)$ inside the unit disk, see Lemma \ref{lem:identity}.  
In order to keep the previous notations intact as much as possible, we will assume, without loss of generality, that $\sum_{j=1}^r \mu_j=q+1$.

With $\zeta_j$'s as in \eqref{zeta}, define the vector $v(\zeta) = (1,\zeta,\zeta^2,..., \zeta^{q+1})$ and denote by $v^{(j)}(\zeta)$ its 
$j$-th entrywise derivative with respect to $\zeta$. Let $B_i$ be  $\mu_i\times (q+2)$ matrices 
$$
B_i = \begin{pmatrix*}[c]
v(\zeta_i) \\
v^{(1)}(\zeta_i) \\
\dots \\
v^{(\mu_i-1)}(\zeta_i)
\end{pmatrix*}, \quad i=1,...,r,
$$
and $B_{r+1}=v(0)$. 
Finally, define the square matrix $\widetilde V$ of size $q+2$
\begin{equation}\label{tildeV}
\widetilde V = 
\begin{pmatrix*}[c]
B_1 \\
\vdots \\
B_r\\
B_{r+1}
\end{pmatrix*}.
\end{equation}
In words, the first $\mu_1$ rows of this matrix consist of the first row of the Vandermonde matrix in (6.9)
and its $\mu_1-1$ derivatives with respect to $\zeta_1$.
The next $\mu_2$ rows consist of the second row of the matrix (6.9) and its derivatives, etc. The last row in $\widetilde V$ is the same as in (6.9).

Define the vector $\widetilde u$ similarly, cf.  \eqref{uoe}: 
let the first $\mu_1$ entries of the vector $\widetilde u$ be $1/(\zeta_1-1)$ and its $\mu_1-1$ derivatives, the next $\mu_2$ entries be 
$1/(\zeta_2-1)$ and its $\mu_2-1$ derivatives, etc. and the last entry be equal $-1$. Finally, let 
$\mathbf{1}$ and $e$ be the vectors defined in \eqref{uoe}. 
Then in view of \eqref{dSjDj} and \eqref{dSDlim}
the vectors $a$ and $b$ satisfy equations \eqref{sysab} 
with $V$ being replaced with $\widetilde V$ and $u$ replaced with $\widetilde u$. 

Our aim is to argue that, cf. \eqref{eVe}, \eqref{1Ve} and \eqref{eVu}, 
\begin{equation}\label{tri}
\begin{aligned}
& e^\top \widetilde V^{-1} e = \prod_{j=1}^r \left(\frac 1{ -\zeta_j }\right)^{\mu_j}, \\
& 1^\top \widetilde V^{-1} e = \prod_{j=1}^r \left(\frac{\zeta_j-1}{  \zeta_j}\right)^{\mu_j},\\
&   e^\top \widetilde V^{-1} \widetilde u =  
- \prod_{j=1}^r \left(\frac 1 {1-\zeta_j}\right)^{\mu_j},
\end{aligned}
\end{equation}
in which case asymptotics  \eqref{ablim} 
and, consequently, the assertion of Corollary \ref{cor:main} 
remain intact 
when some of the zeros have non-unit multiplicities. 

To this end, define the difference operator 
\begin{equation}\label{Dn}
(D^n_\eps f)(x) = \eps^{-n} \sum_{j=0}^n (-1)^{n-j} \begin{pmatrix}
n \\ j
\end{pmatrix}f(x+j\eps).
\end{equation}
This operator approximates the $n$-th order derivative in the sense
\begin{equation}\label{Dlim}
(D^n_\eps f)(x) \xrightarrow[\eps\to 0]{} f^{(n)}(x).
\end{equation}
Let us now define  $\zeta_i^{(k\eps)} := \zeta_i + k\eps$ and the Vandermonde matrices 
$$
B_i^\eps = V\Big(\zeta_i^{(0)},\zeta_i^{(\eps)}, ..., \zeta_i^{((\mu_i-1)\eps)}\Big)\in \Real^{\mu_i\times (q+2)}, \quad i=1,...,r.
$$
Define also the square $(q+2)\times (q+2)$ Vandermonde matrix, cf. \eqref{tildeV},
$$
V^\eps =\begin{pmatrix*}[c]
B_1^\eps \\
B_2^\eps \\
\dots \\
B_r^\eps\\
B_{r+1}
\end{pmatrix*}.
$$
Note that all the points $\zeta_i^{(k\eps)}$ are pairwise distinct for all $\eps>0$ small enough and hence $V^\eps$ is invertible.

The elementary row operation, cf. \eqref{Dn}, 
$$
 \sum_{j=0}^{\mu_1-1}  (-1)^{(\mu_1-1)-j} \begin{pmatrix}
\mu_1-1 \\ j
\end{pmatrix}R_{j+1}\to R_{\mu_1}
$$
applied to $V^\eps$,  changes the last row of the block $B_1^\eps$ to 
$$
\eps^{\mu_1-1} \begin{pmatrix}
0, &  (D_\eps^{\mu_1-1} p_1)(\zeta_1), &  (D_\eps^{\mu_1-1} p_2)(\zeta_1), & ..., &  (D_\eps^{\mu_1-1} p_{q+1})(\zeta_1)    
\end{pmatrix},
$$
where we denoted $p_j(x)=x^j$.  
The elementary row operation  
$$
 \sum_{j=0}^{\mu_1-2}  (-1)^{(\mu_1-2)-j} \begin{pmatrix}
\mu_1 -2\\ j
\end{pmatrix}R_{j+1}\to R_{\mu_1-1}
$$
change the preceding row to 
$$
\eps^{\mu_1-2} \begin{pmatrix}
0, &  (D_\eps^{\mu_1-2} p_1)(\zeta_1), &  (D_\eps^{\mu_1-2} p_2)(\zeta_1), & ..., &  (D_\eps^{\mu_1-2} p_{q+1})(\zeta_1)   
\end{pmatrix}.
$$
We continue in the same manner until the second line $B_1^\eps$ is transformed to 
$$
\eps \Big(\begin{matrix}
0, &  (D_\eps^1 p_1)(\zeta_1), &  (D_\eps^1 p_2)(\zeta_1), & ..., &  (D_\eps^1 p_{q+1})(\zeta_1)
\end{matrix}\Big),
$$
and proceed by applying similar transformations to the next blocks $B_2^\eps,...,B_r^\eps$. 

Let $E_1,...,E_N$ be the sequence of elementary matrices corresponding to all the row operations which have been applied so far. 
Then, in view of \eqref{Dlim}, the following asymptotic formula is obtained:
\begin{equation}\label{fla}
E_N ... E_1 V_\eps = D_\eps \widetilde V (I+O(\eps))
\end{equation}
where $O(\eps)$ is a matrix satisfying $\varlimsup_{\eps\to 0}\|O(\eps)\|/\eps<\infty$, and $D_\eps$ is the diagonal matrix  
$$
D_\eps = \mathrm{diag}\big(1, \eps ,..., \eps^{\mu_1-1}, 1, \eps, ..., \eps^{\mu_2-1}, ..., 1, \eps, ..., \eps^{\mu_{r-1}-1}, 1, 1\big).
$$
Inverting the equation \eqref{fla} we get  
\begin{equation}\label{VtV}
 V_\eps^{-1} E_1^{-1} ... E_N^{-1} =  (I+O(\eps))^{-1} \widetilde V^{-1} D_\eps^{-1}.
\end{equation}
By the above construction none of $E_j$'s affects the last line of $V_\eps$.
Hence $E_j^{-1}$'s are also  elementary matrices which do not affect the last line, i.e., 
 $E_1^{-1} ... E_k^{-1} e=e$. Consequently, in view of (6.13), 
$$
e^\top V_\eps^{-1} E_1^{-1} ... E_k^{-1}e = e^\top V_\eps^{-1} e =\prod_{k=0}^{\mu_1-1} \frac{ 1 }{ -\zeta^{(k\eps)}_1 }...\prod_{k=0}^{\mu_r-1} \frac{ 1 }{ -\zeta^{(k\eps)}_r }\xrightarrow[\eps\to 0]{} \prod_{j=1}^r \frac 1{(-\zeta_j)^{\mu_j}}.
$$
In addition, since $D_\eps^{-1} e=e$ as well,
$$
e^\top (I+O(\eps))^{-1} \widetilde V^{-1} D_\eps^{-1}e=e^\top (I+O(\eps))^{-1} \widetilde V^{-1} e \xrightarrow[\eps\to 0]{} e^\top  \widetilde V^{-1} e.
$$
Combining these two limits with \eqref{VtV} proves the first identity in \eqref{tri}. The second identity is verified similarly.

To check the last identity in \eqref{tri}, define the vector 
$$
u_\eps^\top  =
\bigg(
\frac{1}{\zeta_1^{(0)}-1},    \frac{1}{\zeta^{(\eps)}_1-1},  ...,  \frac{1}{\zeta^{((\mu_1-1)\eps)}_1-1},   ...,   
\frac{1}{\zeta_r^{(0)}-1},   
...,   \frac{1}{\zeta^{((\mu_r-1)\eps)}_r-1} ,   -1  
\bigg) 
$$
and note that 
$$
E_N...E_1 u_\eps = D_\eps \widetilde u (I+O(\eps)).
$$
Hence, in view of \eqref{VtV},
$$
V_\eps^{-1} E_1^{-1} ... E_N^{-1} E_N...E_1 u_\eps =  (I+O(\eps))^{-1} \widetilde V^{-1} D_\eps^{-1}D_\eps \widetilde u (I+O(\eps))
$$
and, consequently, 
$$
e^\top V_\eps^{-1}   u_\eps = e^\top   \widetilde V^{-1}   \widetilde u (I+O(\eps)), \quad \eps \to 0.
$$
Taking the limit $\eps\to 0$ yields the third identity in \eqref{tri}.

\section{Calculations in Example \ref{ex:1}
}\label{app:E}
 
For $d\in (-\frac 1 2,0)$, the function $Q(z)$ has no zeros and 
condition \eqref{algcond} 
becomes   
$$
\Phi_0(-1)\phi(-1) + (-1)^{n} \Phi_1(-1) \phi(-1) =0.
$$ 
In view of \eqref{SD}, this is equivalent to 
\begin{equation}\label{SDm1}
\begin{aligned}
&
S(-1) = 0,  \quad \ \text{if } n \text{\ is even}, \\
&
D(-1) = 0, \quad \text{if } n \text{\ is odd}. 
\end{aligned}  
\end{equation}
In addition, since $z_1=-1$ is also a zero of the  polynomial, reciprocal to $\theta(\cdot)$,
$$
\frac{d}{dz}\Big(\Phi_0(z)\phi(z) + z^{n+2q} \Phi_1(z^{-1}) \phi(z^{-1})\Big)_{\big| z=-1}=0.
$$
In view of \eqref{SD} and \eqref{SDm1}, this is equivalent to
%
\begin{equation}\label{E2}
\begin{aligned}
&
2    D'(-1) +(n+2+2c)      D(-1)   =0,   \quad \text{if   } n\ \text{is even},\\
&
2    S'(-1)\ +(n+2+2c)      S(-1)\   =0,   \quad   \text{if   } n\ \text{is odd},
\end{aligned}
\end{equation}
where we defined 
$$
c :=  \frac{\phi'(-1) }{\phi(-1) } +\frac{ X'(-1)}{ X(-1)}.
$$
Finally, as before, we have the conditions, cf. \eqref{aeq}-\eqref{beq},
\begin{equation}\label{E3}
S(0)=D(0)= \frac 1 2 \sigma_0^2.
\end{equation}

By substituting, cf. (5.13),
\begin{equation}\label{SDjn_ex}
\begin{aligned}
S(z) =\, & a_0 S_{0,n}(z) + a_1 S_{1,n}(z), \\
D(z) =\, & b_0 D_{0,n}(z) + b_1 D_{1,n}(z), 
\end{aligned}
\end{equation}
and using the approximations in \eqref{prove1-S}-\eqref{prove1-D} and \eqref{dSDlim}, we obtain asymptotic systems of linear equations
for the coefficients $a_j$ and $b_j$.
For even $n$, coefficients $a_0$ and $a_1$ satisfy the two dimensional Vandermonde system, cf.  \eqref{sysab}, 
with $\zeta_1=-1$ and, as before, we get 
$$
a_{1,n} =  
  \frac {\sigma^2_0 } 2  
\Big( 
1+ \frac {d (1+d)} n   
\Big)     
   + O(n^{-2}), \quad n\to\infty.
$$   
The asymptotic system for the coefficients $b_0$ and $b_1$ take a different form. Plugging \eqref{SDjn_ex} into the first equation in \eqref{E2} 
gives 
\begin{equation}\label{b0b1}
  \Big(     2 D'_{0,n}(-1)  + (n+2+2 c)    D_{0,n}(-1) \Big)b_0
   + 
   \Big(       
  2 D'_{1,n}(-1)   + (n+2+2 c)    D_{1,n}(-1) \Big) b_1    =0.
\end{equation}
The approximations from \eqref{prove1-S}-\eqref{prove1-D} and  \eqref{dSDlim} take the form  
\begin{align*}
D_{0,n}(-1) =\, & 1  + \frac{\widetilde \mu_0}{2} n^{-1}  +  O(n^{-2}), \\
D_{1,n}(-1) =\, & -1 + \frac{\widetilde \mu_0}{2} n^{-1}  +  O(n^{-2}), \\
D'_{0,n}(-1) =\, &    \frac{\widetilde \mu_0}{4} n^{-1}  +  O(n^{-2}), \\
D'_{1,n}(-1) =\, & 1  + \frac{\widetilde \mu_0}{4} n^{-1}  +  O(n^{-2}),  
\end{align*}
where we defined 
$$
\widetilde \mu_0 :=\frac {\sin (\pi d) }{ \pi   }\mu_0 =  d(1-d).
$$
Thus equation \eqref{b0b1} becomes 
\begin{align*}
&
  \Big( n +(2+2 c   +   \tfrac{1}{2}\widetilde \mu_0)  +(\tfrac 3 2+  c) \widetilde \mu_0  n^{-1} + O(n^{-2}) \Big)b_0
   + \\
&   \Big(  - n         +   (\tfrac{1}{2}\widetilde \mu_0-2 c)  +(\tfrac 3 2+  c) \widetilde \mu_0  n^{-1} + O(n^{-2})  \Big) b_1    =0 
\end{align*}
Another equation is obtained from \eqref{E3}:
$$
\big(1  +\widetilde \mu_0  n^{-1}+ O(n^{-2}) \big) b_0+\big(\widetilde \mu_0 n^{-1}+ O(n^{-2}) \big) b_1= \frac 1 2 \sigma^2_0.
$$
Solving these two equations for $b_0$ and $b_1$ gives 
\begin{align*}
b_{1,n} = & \frac 1 2 \sigma^2_0
\frac
{
  n +(2+2 c   +   \tfrac{1}{2}\widetilde \mu_0)  + O(n^{-1}) 
}
{
n
+  (\frac 3 2\widetilde \mu_0    +2 c)  + O(n^{-1})
} = \\
&
\frac 1 2 \sigma^2_0 \Big(1+ (2-\widetilde \mu_0)n^{-1}\Big) + O(n^{-2}), \quad n\to\infty.
\end{align*}
Plugging the obtained asymptotic expressions for $a_{1,n}$ and $b_{1,n}$ into  \eqref{ab2salpha} 
we get 
$$
\sigma^2(n)-\sigma^2_0=
  \sigma^2_0   
\frac{
  d^2 + 1   
} n
   + O(n^{-2}), \quad \alpha(n) = \frac{d-1}{n} + O(n^{-2}), \quad n\to\infty.
$$

Similarly, for odd $n$,
$$
b_{1,n} = 
 \frac 1 2 \sigma^2_0  
\Big(1
-
\frac 1 n   d (1-d) \Big) + O(n^{-2}), \quad n\to\infty. 
$$
The second equation in \eqref{E2} becomes 
$$
  \Big(     2 S'_{0,n}(-1)  + (n+2+2 c)    S_{0,n}(-1) \Big)a_0
   +  
   \Big(       
  2 S'_{1,n}(-1)   + (n+2+2 c)    S_{1,n}(-1) \Big) a_1    =0.
$$
The approximations from \eqref{prove1-S}-\eqref{prove1-D} and  \eqref{dSDlim} yield
\begin{align*}
S_{0,n}(-1)  =\, & 1-   \frac{\widetilde\lambda_0}{2} n^{-1}+ O(n^{-2}),
\\
S_{1,n}(-1)  =\, & -1-   \frac{\widetilde\lambda_0}{2} n^{-1}+ O(n^{-2}),
\\
S'_{0,n}(-1)  =\, &  -   \frac{\widetilde\lambda_0}{4 } n^{-1}+ O(n^{-2}),
\\
S'_{1,n}(-1)  =\, & 1 -   \frac{\widetilde\lambda_0}{4 } n^{-1}+ O(n^{-2}),
\end{align*}
where 
$
\widetilde \lambda_0 := \dfrac {\sin (\pi d) }{ \pi  }  \lambda_0 = d(1+d).
$
Substitution into the above equation gives 
$$
  \Big( n+(2+2 c- \tfrac{1}{2} \widetilde\lambda_0)   - ( \tfrac 3 2+  c)  \widetilde\lambda_0n^{-1} \Big)a_0
   + \\
    \Big(       
      - n-( 2 c+  \tfrac{1}{2}\widetilde\lambda_0)   -(\tfrac 32+ c)  \widetilde\lambda_0 n^{-1}  \Big) a_1    =0.
$$

The second equation follows from \eqref{E3}
$$
(1 -   \widetilde \lambda_0  n^{-1}) a_0 - \widetilde \lambda_0  n^{-1}a_1 =\frac 1 2 \sigma_0^2.
$$
Solving for $a_0$ and $a_1$ yields
$$
a_{1,n} = 
\frac 1 2 \sigma_0^2
\frac{
n+ 2+2 c- \tfrac{1}{2} \widetilde\lambda_0 +   O(n^{-1})
}
{
 n  + 2 c -  \tfrac{3}{2}\widetilde\lambda_0 +O(n^{-1})    
} = \frac 1 2 \sigma_0^2\Big(1 +
(2+     \widetilde\lambda_0)n^{-1}
\Big)+O(n^{-2}).
$$
Plugging the obtained asymptotic expressions for $a_{1,n}$ and $b_{1,n}$ into  \eqref{ab2salpha} 
we find that 
$$
\sigma^2(n)-\sigma^2_0 =   
  \sigma_0^2  \frac { d^2+1} n  + O(n^{-2}), \quad
\alpha(n) =  \frac {d+1} n  + O(n^{-2}), \qquad n\to\infty.
$$
\qed

\section{Numerical experiments}
In this section, we present the outcomes of some numerical experiments that illustrate the main result of this paper, Theorem \ref{main-thm}.
For each fixed $n$, the predictor coefficients were found by solving  equation \eqref{maineq}. 
and  $\sigma^2(n)$ and $\alpha(n)$ were 
computed using equations    \eqref{gLgR} and   \eqref{sa}. 
The limiting error $\sigma^2_0$ from \eqref{deluc}, 
given by the Szeg\"{o}-Kolmogorov formula \eqref{SKfla},  
was approximated by numerical integration using suitably truncated series in \eqref{fGnf}.
As can be seen from the plots in Figures \ref{fig:fGn} and \ref{fig:farima1}, the transient behavior of the corresponding quantities is somewhat different for fGn and FARIMA$(0,d,0)$ processes.
For larger $n$, the results agree with the asymptotic theory. 


\begin{figure}[h!]
    \centering
    \includegraphics[width=0.8\textwidth]{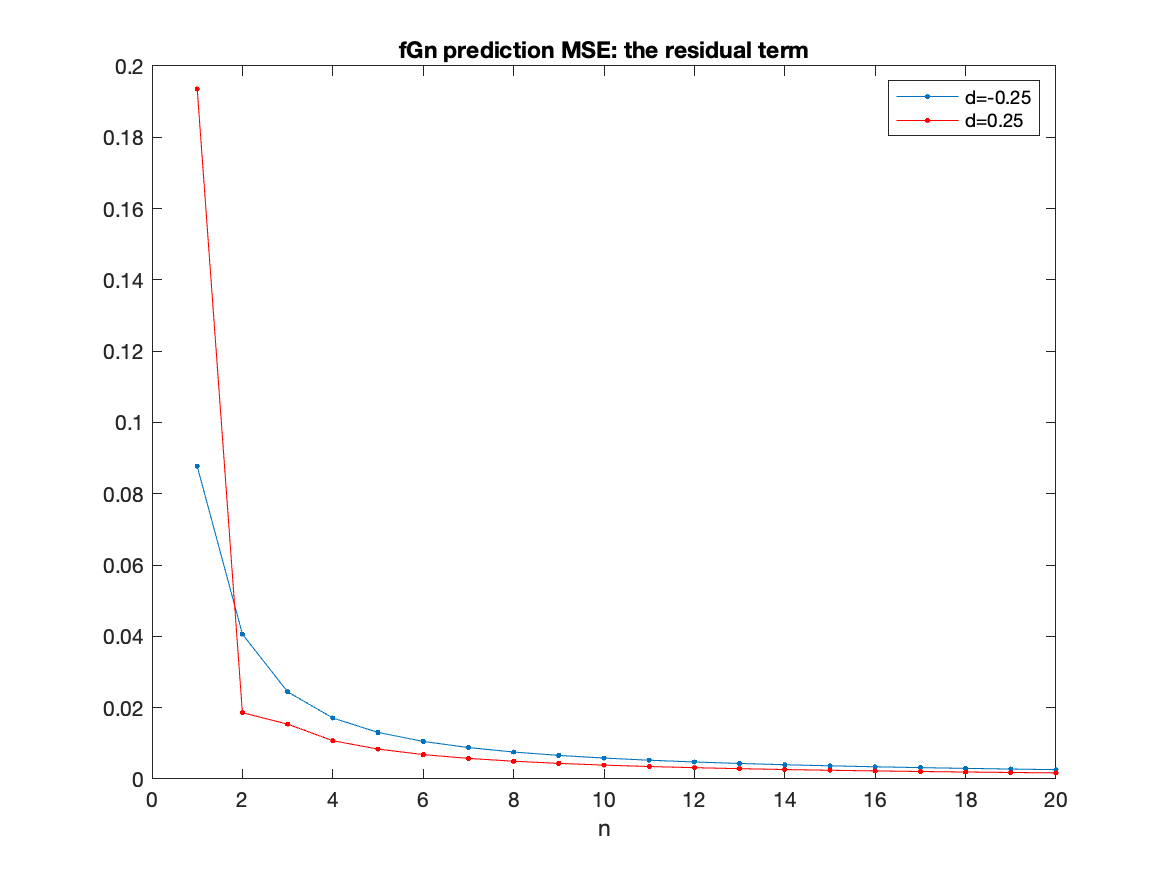}
    \includegraphics[width=0.8\textwidth]{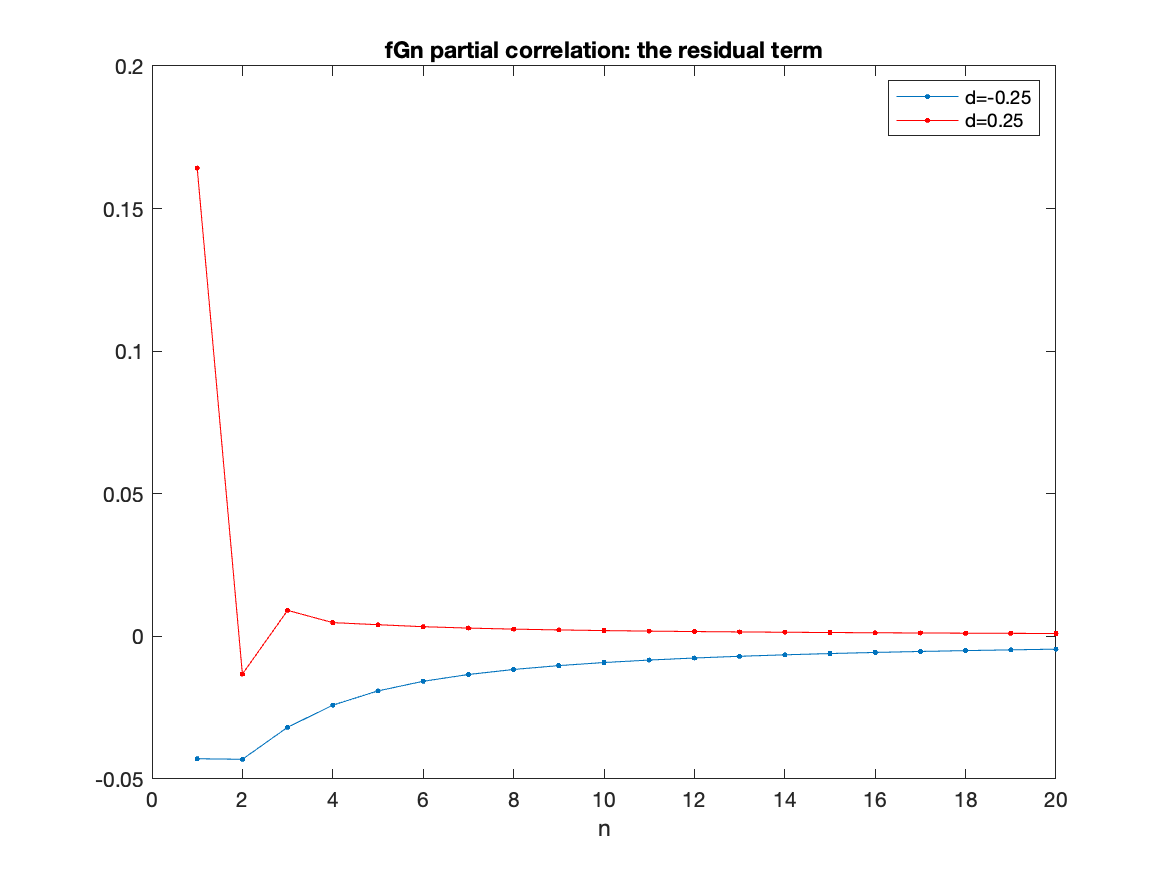}
    \caption{The top figure shows the sequence $n(\sigma^2(n)/\sigma_0^2-1)-d^2$, i.e., the residual term in  \eqref{deluc} 
    for the fGn, plotted for two values of $d$. The bottom figure depicts the sequence $\alpha(n)n-d$ for the same values of $d$.  }
    \label{fig:fGn}
\end{figure}


\begin{figure}[h!]
    \centering
    \includegraphics[width=0.8\textwidth]{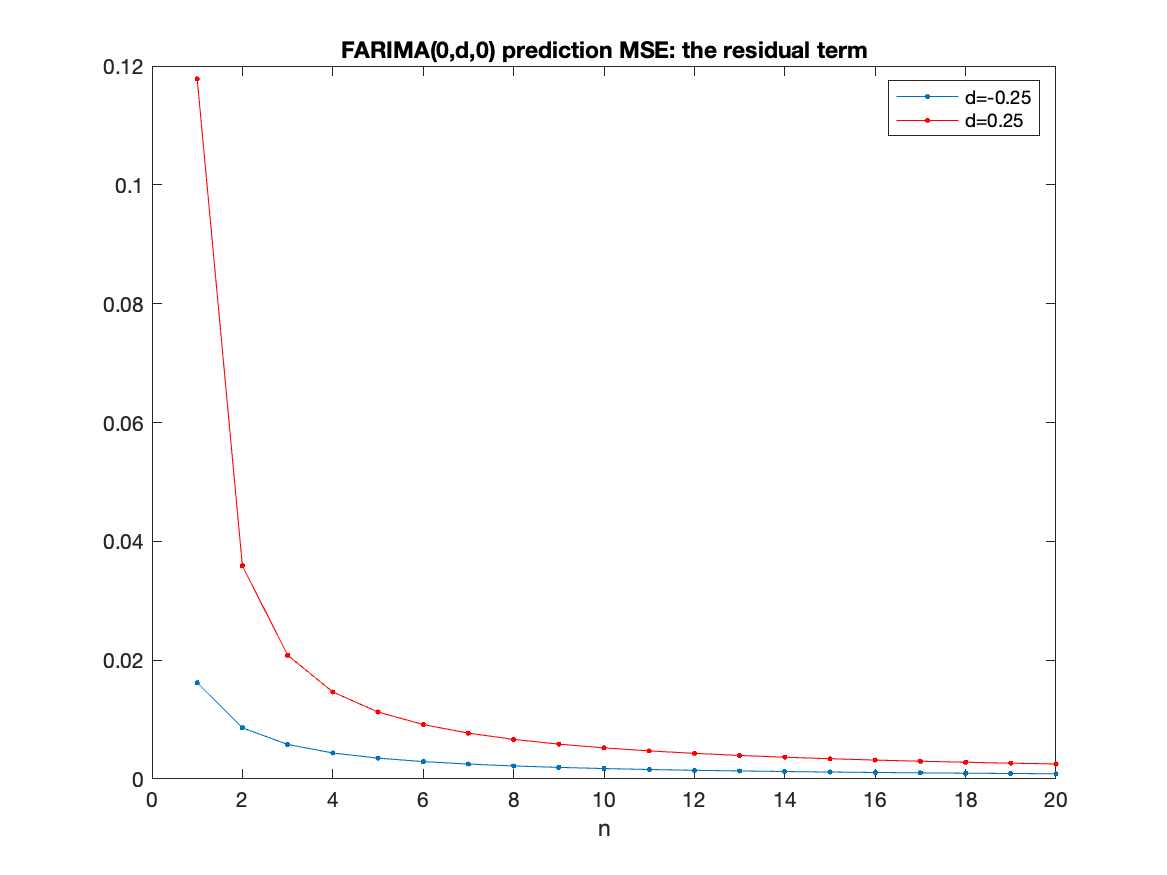}
    \includegraphics[width=0.8\textwidth]{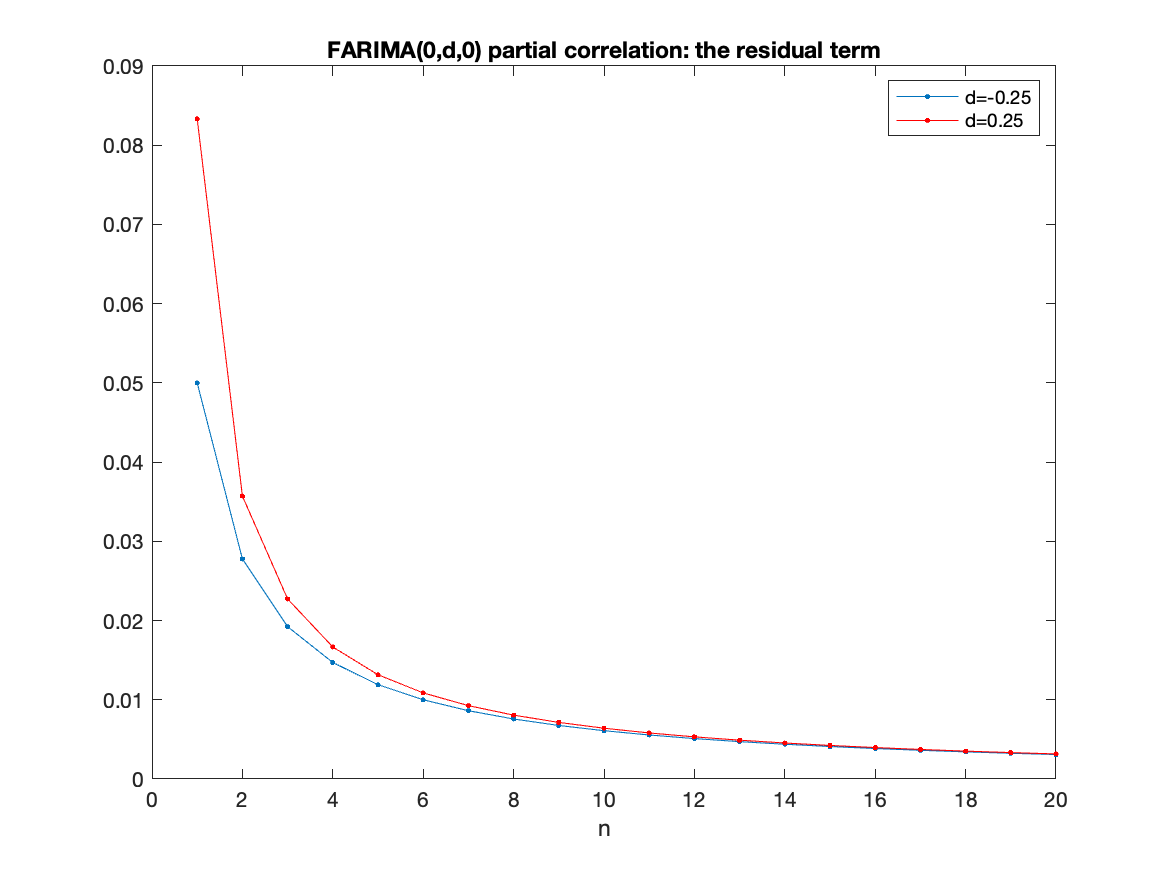}
    \caption{The corresponding plots for the FARIMA$(0,d,0)$ process (for which $\sigma_0^2=1$).}
    \label{fig:farima1}
\end{figure}

\end{appendix}

%
\begin{funding}
The first author was supported by the Israel Science Foundation (ISF),
grant No. 2821/25.
\end{funding}

\clearpage 


\def\cprime{$'$} \def\cprime{$'$} \def\cydot{\leavevmode\raise.4ex\hbox{.}}
  \def\cprime{$'$} \def\cprime{$'$} \def\cprime{$'$}
  


\end{document}